\documentclass[11pt]{amsart}
\usepackage{amsmath}
\usepackage[scaled]{helvet}
\usepackage[T1]{fontenc}
\usepackage{textcomp}
\normalfont
\usepackage{amssymb}
\usepackage{mathrsfs}
\usepackage{color}

\usepackage{graphicx}
\usepackage[utf8]{inputenc}
\usepackage[english]{babel}
\usepackage{csquotes}
\usepackage{amssymb}
\usepackage{tikz-cd}
\usepackage{csquotes}

\usepackage{fullpage}
\usepackage{hyperref}
\usepackage[normalem]{ulem}
\usepackage{enumerate}
\usepackage{enumitem}
\usepackage[textwidth=20mm]{todonotes}
\usepackage{xspace}

\hypersetup{
  colorlinks   = true,
  urlcolor     = blue,
  linkcolor    = blue,
  citecolor   = red
}

\makeatletter
\newtheorem*{rep@theorem}{\rep@title}
\newcommand{\newreptheorem}[2]{%
	\newenvironment{rep#1}[1]{%
		\def\rep@title{#2 \ref{##1}}%
		\begin{rep@theorem}}%
		{\end{rep@theorem}}}
\makeatother

\newtheorem{theorem}{Theorem}[section]
\newreptheorem{theorem}{Theorem}
\newtheorem{lemma}[theorem]{Lemma}
\newtheorem{proposition}[theorem]{Proposition}
\newtheorem{corollary}[theorem]{Corollary}
\newreptheorem{corollary}{Corollary}
\newtheorem{conjecture}[theorem]{Conjecture}

\newtheorem{questions}[theorem]{\em Questions}

\newtheorem{definition}[theorem]{Definition}

\newtheorem{example}[theorem]{Example}
\newtheorem{examples}[theorem]{Examples}

\theoremstyle{remark}
\newtheorem{remark}[theorem]{Remark}

\newtheorem{StandingAssumptions}[theorem]{Standing Hypotheses}

\numberwithin{equation}{section}

\newcommand{\mc}{\mathcal}
\newcommand{\mb}{\mathbb}
\newcommand{\mbf}{\mathbf}
\newcommand{\mf}{\mathfrak}
\newcommand{\scr}{\mathscr}

\newcommand{\la}{\lambda}
\newcommand{\La}{\Lambda}

\newcommand{\norm}[1]{\left\lVert \vcenter{\hbox{$\displaystyle #1 $}} \right\rVert}
\newcommand{\inner}[1]{\left\langle #1 \right\rangle}
\newcommand{\ds}{\displaystyle}

\newcommand{\Isom}{\mathrm{Isom}}

\newcommand{\Aut}{\mathrm{Aut}}
\newcommand{\Der}{\mathrm{Der}}

\newcommand{\Z}{\mathbb{Z}}

\newcommand{\Q}{\mathbb{Q}}
\newcommand{\R}{\mathbb{R}}
\newcommand{\C}{\mathbb{C}}

\newcommand{\N}{\mathbb{N}}
\newcommand{\KH}{\mathbb{H}}
\newcommand{\KO}{\mathbb{O}}

\newcommand{\GL}{\mathrm{GL}}
\newcommand{\SL}{\mathrm{SL}}
\newcommand{\SO}{\mathrm{SO}}
\newcommand{\OO}{\mathrm{O}}

\DeclareMathOperator{\proj}{proj}
\DeclareMathOperator{\Span}{span}
\DeclareMathOperator{\Hom}{Hom}
\newcommand{\Sp}{\mathrm{Sp}}
\newcommand{\Spin}{\mathrm{Spin}}
\newcommand{\Diff}{\mathrm{Diff}}
\DeclareMathOperator{\im}{im}
\DeclareMathOperator{\id}{id}

\renewcommand{\bar}{\overline}
\newcommand{\what}{\widehat}
\newcommand{\til}{\widetilde}

\newcommand{\op}{\operatorname}
\newcommand{\of}{\circ}
\newcommand{\set}[1]{\left\{ #1 \right\}}

\newcommand{\cout}[1]{}

\definecolor{darkcyan}{rgb}{0. 0.65, 0.65}

\newcommand{\CC}{Carnot-Carath\'{e}odory\xspace}
\newcommand{\bs}{\backslash}

\newcommand\rest[1]{{\raisebox{-.5ex}{$|$}_{#1}}}

\newcommand{\eps}{\epsilon}

\newtheorem{Structural Stability Theorem}[theorem]{Structural Stability Theorem}

\def\bt{\begin{theorem}}
\def\et{\end{theorem}}
\def\bd{\begin{definition}}
\def\ed{\end{definition}}
\def\bl{\begin{lemma}}
\def\el{\end{lemma}}

\def\be#1\ee{\begin{align}\begin{split} #1 \end{split}\end{align}}
\def\beq#1\eeq{\begin{align*}\begin{split} #1 \end{split}\end{align*}}

\definecolor{grey}{rgb}{0.22, 0.32, 0.51}

\begin{document}

\title{Carnot metrics, Dynamics and Local Rigidity}

\author{Chris Connell, Thang Nguyen, Ralf Spatzier$^{\ddagger }$}

\address{Department of Mathematics,
Indiana University, Bloomington, IN 47405}
\email{connell@indiana.edu}

\address{Department of Mathematics, University of Michigan, 
    Ann Arbor, MI, 48109}
\email{qtn@umich.edu}

\address{Department of Mathematics, University of Michigan, 
    Ann Arbor, MI, 48109.}
\email{spatzier@umich.edu}

\thanks{$^{\ddagger}$ Supported in part by NSF grants DMS 1607260 and DMS 2003712.}

\subjclass[2010]{Primary 53C17, 53C24, 37D40; Secondary 53C20}

\date{}

\begin{abstract}

This paper develops new techniques for studying smooth dynamical systems in the presence of a \CC metric. Principally, we employ the theory of Margulis-Mostow, M\'etivier, Mitchell and Pansu on tangent cones to establish  resonances between
Lyapunov exponents. We apply these results in three different settings. First, we explore rigidity properties of smooth dominated splittings for Anosov diffeomorphisms and flows via associated smooth \CC metrics. Second, we obtain local rigidity properties of higher hyperbolic rank metrics in a neighborhood of a locally symmetric one. For the latter application we also prove structural stability of the Brin-Pesin asymptotic holonomy group for frame flows. Finally, we obtain local rigidity properties for uniform lattice actions on the ideal boundary of quaternionic and octonionic symmetric spaces.
  \end{abstract}

\maketitle

\tableofcontents
\section{Introduction}

Totally nonintegrable $k$-plane distributions $E$ on a manifold $N$ have been of considerable interest in partial differential equations, geometry and control theory. In particular they play a central role in H\"ormander's work on hypo-elliptic partial differential operators (e.g. \cite{Hormander67}). Such operators arise naturally from nilpotent groups. Fundamental work of Folland, Goodman, Kohn, Rothchild-Stein developed a converse. 

Later on, Gromov (\cite{Gromov81}) introduced tangent cones for general metric spaces and found criteria for their existence. Namely, for a metric space $(X,d)$ a tangent cone at $p\in X$ is a convergent limit in the pointed Gromov-Hausdorff topology of the sequence of pointed metric spaces $(X,p,\la d)$ as $\la\to \infty$. Note that tangent cones do not always exist, and when they do exist, they are not always unique. Spun from a key result of M\'etivier (\cite{Metivier}), Gromov, Mitchell, Margulis-Mostow, and Pansu, amongst others, have investigated the existence, uniqueness and structure of the pointed tangent cones for \CC metrics (\cite{Gromov81,Gromov96,MargulisMostow00,MargulisMostow95,Mitchell,Pansu}). 

In the present paper we explore these tangent cones with the goal of establishing new techniques for the study of Lyapunov exponents of dynamical systems respecting sub-Riemannian metrics. 
We suspect that such systems will typically possess algebraic features. For example, we use these techniques to establish criteria for the Lyapunov spectra to be arithmetic progressions. This fits into the scope of the recent duelling programs between flexibility versus rigidity of Lyapunov exponents. Indeed, pointwise arithmeticity of the Lyapunov spectrum  sometimes leads to complete rigidity of the dynamics \cite{Butler:17,Butler18}. Further afield, constancy or extremality of the Lyapunov spectrum in certain cases yields local rigidity (cf. for example \cite{Connell03,DeSimoiEtAl20,Gogolev11,GogolevEtAl20,SaghinYang19,DeWitt19}). These contrast with results about flexibility of the Lyapunov spectrum for tori and surfaces, (e.g. see \cite{ErchenkoKatok19,CarrascoSaghin21,BochiEtAl19,BarthelmeErchenko17}).

The arithmeticity of Lyapunov spectra have ramifications to both Riemannian geometry and group actions. Indeed, we will discuss rigidity of hyperbolic rank structures in Sections \ref{subsec:intro_hypRank} and local rigidity of projective actions in Section \ref{subsec:intro_boundary}.

\subsection{Sub-Riemannian Geometry}\label{subsec:intro_subRiem}

We recall that a smooth manifold $N$ is called a {\em sub-Riemannian manifold} if
\begin{enumerate}
\item $N$ is equipped with   a smooth distribution $E$, called a {\em horizontal distribution} which satisfies H\"{o}rmander's condition; that is, vector fields tangent to $E$ and their iterated brackets generate $TN$, and
\item $E$ is endowed with a smooth Riemannian metric $\inner{\cdot,\cdot}_x$.
\end{enumerate}
The \CC metric $d_C$ on $N$ between a pair of points $x$ and $y$ is defined as the infimum of length of curves tangent to $E$ from $x$ to $y$. By a theorem of Chow (\cite{Chow}), the \CC metric $d_C$ is finite on connected components of $N$.

For a horizontal distribution $E$, at each point $x\in N$ and $i\geq 1$ we define $E^i_x$ to be the subspace of $T_xN$ spanned by all $i$-fold commutators of vector fields tangent to $E$.
We call a horizontal distribution $E$ {\em generic (of order $r$)}  at a point $p\in N$ if on a neighborhood $U$ of $p$ the subspaces $E^i_x$ fit together to form a strictly increasing sequence of (smooth) subbundles
$$E=E^0\subset E^1 \subset \dots \subset E^r=TU.$$

In this setting, Mitchell (\cite{Mitchell}) showed that the tangent cones at generic points of $N$ are graded nilpotent Lie groups equipped with left-invariant \CC metrics which are unique by Margulis-Mostow \cite{MargulisMostow00}. Following Pansu (\cite{Pansu}), Margulis-Mostow also developed the notion of a Pansu derivative of a quasi-conformal map on a sub-Riemannian manifold. We will be applying this construction to maps preserving a distribution. 

We develop these ideas to understand the dynamics of differentiable maps preserving suitable distributions. 
Combining the main theorem of Mitchell (\cite{Mitchell}) together with Proposition \ref{prop:cont_of_tan_cone} and analyzing (cf. Proposition \ref{prop:df_isom}) the proof of the main theorem of Margulis-Mostow (\cite{MargulisMostow95}) provides the catalyst for our investigations:

\begin{theorem}[Tangent Cone Structure Theorem]\label{thm:tangentconeN}
	Let $f:N\to N$ be a local $C^\infty$ diffeomorphism of a sub-Riemannian manifold $N$ preserving the horizontal distribution $E$. Then the set 
	$\Omega:=\{x\in N: E \text{ is generic at } x\}$ is open dense and $f$-invariant. Moreover,
	\begin{enumerate}
		\item For every $x\in \Omega$, the tangent cone $TC_xN$ of $N$ at $x$ exists, is a graded nilpotent Lie group with a left invariant \CC metric and the tangent cones vary continuously on $\Omega$ with respect to the Gromov-Hausdorff topology.
		\item The map $f$ induces a graded Lie group isomorphism $(f_*)_x:TC_xN\to TC_{f(x)}N$ between the tangent cones $TC_xN$ of $N$ at $x\in \Omega$ and $TC_{f(x)}N$ of $N$ at $f(x)$. We call $(f_*)_x$ the {\em Carnot derivative} of $f$ at $x$.
	\end{enumerate}
\end{theorem}

With substantially more work, we later prove a foliated version of this theorem (Theorem \ref{thm:tangentcone}). Applying this to an unstable foliation will provide applications in the dynamical setting.

Next we consider the relation between the Carnot derivative and the ordinary one. The following three theorems we employ a detailed investigation of the proof of the Margulis-Mostow Theorem (\cite{MargulisMostow95}). 

\begin{theorem}[Subadditivity of the Spectrum at Periodic Points]\label{thm:spectrum_ineq}

	Let $f:N\to N$ be a local $C^\infty$ diffeomorphism of a sub-Riemannian manifold $N$ preserving the horizontal distribution $E$. Assume that $E$ is generic at a point $p\in N$ where $f(p)=p$. Suppose the tangent cone $TC_pN$ is $(r+1)$-step and set $n_{-1}=0$ and $n_i=\dim E^{i}(p)$ for $i\in \set{0,\dots,r}$. If $D_pf\rest{E^{i}_p}$ has Lyapunov exponents $(\log \la_1,\log \la_2,\dots,\log \la_{n_0},\log \la_{{n_0}+1},\dots,\log \la_{n_1},\dots,\log \la_{n_{i}})$ listed with multiplicity and with $\log \la_{n_{j-1}+1},\dots,\log \la_{n_{j}}$ in nondecreasing order for each $j\in\set{0,\dots,i}$, then for each $i\in \set{0,\dots, r}$ and $j\in\set{n_{i-1}+1,\dots,n_{i}}$, we have $(i+1) \log \la_1\leq \log \la_{j}\leq (i+1)\log \la_{n_0} $.

\end{theorem}

Note that at a fixed point, the Lyapunov exponents simply amount to being the logs of the magnitude of the eigenvalues of $D_pf$. In the special case when the tangent cone is the Heisenberg group with its standard contact structure and \CC metric, we have a stronger statement.
\begin{theorem}[Additivity of the Spectrum for Heisenberg Cones at Periodic Points]\label{thm:Hei-spec}
	Let $f:N\to N$ be a local $C^\infty$ diffeomorphism of a sub-Riemannian manifold $N$ preserving the horizontal distribution $E$. Assume that $E$ is generic at a point $p\in N$ where $f(p)=p$. Suppose the tangent cone $TC_pN$ is isomorphic to the Heisenberg group $H^{2n+1}$. If $\log \lambda_1, \log \lambda_2, \dots, \log \lambda_{2n}$ are the Lyapunov exponents of $D_pf\rest{E}$ listed with multiplicity and $\log \la_{2n+1}$ is the remaining Lyapunov exponent of $D_pf$, then $\log \lambda_1+\log \lambda_2+\dots+\log \lambda_{2n}=n \log \lambda_{2n+1}$.
\end{theorem}

For our next result, given an $(r+1)$-step graded nilpotent Lie group $\mc{N}$ we denote its graded Lie algebra by $\mf{n}=\bigoplus_{i=0}^{r} \mf{n}^i$ where $\mf{n}^0,\mf{n}^1=[\mf{n}^0,\mf{n}^0],\dots,\mf{n}^{r}=[\mf{n}^{r},\mf{n}^0]$. Suppose $f:N\to N$ is a diffeomorphism of a sub-Riemannian manifold $N$ preserving its distribution $E$. If the Carnot derivative of $f$ at a fixed point where $E$ is generic is a homothety, then the Lyapunov spectrum is arithmetic:

\begin{theorem}[Arithmeticity of the Spectrum at Periodic Points]\label{thm:Lyap_exp}
	Let $f:N\to N$ be a local $C^\infty$ diffeomorphism of a sub-Riemannian manifold $N$ preserving the horizontal distribution $E$. Assume that  $E$ is generic at a point $p\in N$ where $f(p)=p$, and that the graded automorphism $f_*:TC_pN\to TC_pN$ induced from $f$ is a homothety. Then there exists a $\la>1$ such that $D_pf\rest{E^{i}_p}$ has Lyapunov exponents $\log \la,2\log \la,\dots,(i+1)\log\la$ with corresponding multiplicities $\dim \mf{n}^0,\dots,\dim \mf{n}^{i}$ for $i=0,\dots,r$.

\end{theorem}

We call a graded nilpotent Lie group {\em asymmetric} (cf. Definition \ref{def:asymmetric} and Lemma \ref{lem:asymmetric_homotheties}) if its group of graded automorphisms consists of homotheties with respect to some \CC metric induced by an inner product on $\mf{n}^0$. The asymmetric groups are open in the space of isomorphism classes of nilpotent groups with a given grading, and sometimes nonempty (cf. Proposition \ref{examp:nilrigid} and Corollary \ref{cor:stable-asymmetry}). Since Theorem \ref{thm:Lyap_exp} holds for any choice of sub-Riemannian metric on $N$, induced from an inner product on $E$, we have the following.
\begin{corollary}
 In the above theorem, when $TC_pN$ is an asymmetric nilpotent group, the conclusion holds without the assumption that the induced graded automorphism $f_*$ is a homothety. 
\end{corollary}

We will apply the natural extension of this theorem to the case of a $C^1$ diffeomorphism $f:X\to X$ (or flow $\varphi_t: X\to X$) leaving invariant a foliation $\mc{F}$ with $C^\infty$ leaves tangent to a continuous distribution $D$. 
If $E\subseteq D$ is a $df$-invariant continuous subdistribution, we define the set 
\[
\Omega=\set{v\in X: E \text{ is } C^\infty \text{ in a neighborhood of } v \text{ in } \mc{F}(v) \text{ and is horizontal and generic there}}.
\]
Suppose that the distribution $E$ is uniformly $C^k$ along $\mc{F}$ (cf. Definition \ref{def:Ckunif}) for a sufficiently large $k$ and that $\mc{F}$ is transversally H\"{o}lder continuous. Also assume that $f$ (or the flow $\varphi_t$) is $C^\infty$ along leaves, topologically transitive, satisfies the stable closing property (cf. Definition \ref{def:SCP}), and that $df\rest{D}$ is transversally H\"{o}lder continuous.

In this setting, the Foliated Tangent Cone Structure Theorem (Theorem \ref{thm:tangentcone} and cf. Corollary \ref{cor:uniqueclass}) provides a common isomorphism class, associated to the distribution $D$, of graded nilpotent Lie group structures for all tangent cones $TC_v\mc{F}(v)$ with $v\in \Omega$ and whose metrics vary continuously on $\Omega$. This leads to our main rigidity result for Lyapunov exponents along leaves of the foliation. 

\begin{theorem}[Arithmeticity of the Lyapunov Spectrum]\label{thm:equal}
	 Suppose that the graded nilpotent group associated to $D$ on $\Omega$ is asymmetric and $(r+1)$-step. Let $\mu$ be any finite $f$-invariant (resp. $\varphi_t$-invariant) ergodic measure whose support $\op{Supp}(\mu)$ satisfies $\op{Supp}(\mu)\cap\Omega\neq \emptyset$. Then the Lyapunov exponents of $f$ (resp. $\phi_1$) along $\mc{F}$ are $\log\lambda, 2\log\lambda, \dots, (r+1)\log\lambda$ with multiplicities $\dim E,\dim E^{1}-\dim E,\dots,\dim E^{r}-\dim E^{r-1}$ for some $\la=\la_\mu>1$.

\end{theorem}

Now we consider some applications of the above results.

\subsection{Smooth Slow Distributions}\label{subsec:smooth_slow}

One family of examples where asymmetric nilpotent groups naturally arise are the quaternionic hyperbolic and octonionic hyperbolic symmetric spaces. Here  the tangent cones of the stable and unstable foliations of a geodesic flow on compact quotients of these spaces are naturally isomorphic to certain asymmetric nilpotent groups of Heisenberg type (see \ref{examp:nilrigid} below) corresponding to the nilradical of the full isometry group.

In the case when a flow nearby one of these symmetric flows still has smooth slow distribution we show it is orbit equivalent to the symmetric one.

\begin{theorem}\label{thm:locally_rigid_flow}
	Let $\phi_t^0$ be the geodesic flow on a locally quaternionic hyperbolic or octonionic hyperbolic closed manifold $M$. Then if $\phi_t$ is any $C^\infty$ flow $C^1$ close to $\phi_t^0$ for which $E^u_{slow,\phi_t}$ remains $C^\infty$ along unstable leaves and is sufficiently uniformly $C^1$ close (in the sense of Definition \ref{def:unifCr}) to that of $\phi_t^0$, then $\phi_t$ is $C^\infty$ orbit equivalent to $\phi_t^0$. 
\end{theorem}

We suspect that strong rigidity properties hold whenever slow (un)stable distributions are smooth along leaves of the (un)stable foliation.

\begin{questions}\label{ques:splitting}
	Can we characterize Anosov or partially hyperbolic diffeomorphisms or flows with smooth $E^u_{slow}$ distributions along the unstable foliation? When are these smoothly orbit equivalent to algebraic ones? Is it sufficient that $E^u_{slow}$ be smooth along unstable leaves and completely nonintegrable there?
\end{questions}

Note that there are simple nonalgebraic examples coming from smooth time changes and suspensions of products (see Section \ref{subsec:time-change}).
We also remark that there are numerous algebraic Anosov examples, e.g. on tori or products, where the slow distribution is integrable.

We obtain stronger rigidity properties if we restrict to geodesic flows.

\begin{theorem}\label{thm:locally_rigid}
	Let $g_0$ be a locally quaternionic hyperbolic or octonionic hyperbolic metric on a smooth closed manifold $M$. Then $g_0$ is locally rigid within the family of $C^2$ close $C^\infty$ metrics whose $E^u_{slow}$ remains $C^\infty$ along unstable leaves and is sufficiently uniformly $C^1$ close (in the sense of Definition \ref{def:unifCr}) to that of $g_0$.

\end{theorem}

Another class we can handle are Anosov diffeomorphisms on nilmanifolds arising from asymmetric nilpotent groups: let $M$ be a closed manifold and let $f_0$ be a transitive $C^\infty$ Anosov diffeomorphism such that tangent cones of unstable leaves exist everywhere and are isomorphic to a fixed asymmetric $(r+1)$-step Carnot nilpotent Lie group $N$. Let $E^u_{slow,f_0}$ be the horizontal distribution along unstable leaves that gives rise to the (graded) structure of $N$. (As it turns out, $E^u_{slow,f_0}$ is the slow distribution of smallest dimension which is still horizontal.) We obtain the following local spectral rigidity theorem.

\begin{theorem}\label{thm:samespec}
	There is a $C^1$ open neighborhood $U$ of $f_0$ in $\Diff^\infty(M)$ such that if $f \in U$ admits a smooth splitting $E^u_f=E^u_{fast,f}\oplus E^u_{slow,f}$ along unstable leaves with $\dim(E^u_{slow,f})=\dim(E^u_{slow,f_0})$, and $E^u_{slow,f}$ is sufficiently uniformly $C^{r}$ close along unstable leaves to $E^u_{slow,f_0}$, then for any invariant ergodic measure $\mu$ there is $\lambda_\mu>1$ such that the unstable Lyapunov exponents of $f$ with respect to $\mu$, are $\log\lambda_\mu,2\log\lambda_\mu,\dots,(r+1)\log\lambda_\mu$ occurring with the same multiplicitites as for $f_0$.
\end{theorem}

We may apply the above theorem to certain Anosov diffeomorphisms that arise on compact nilmanifolds arising as quotients of products of asymmetric nilpotent groups.

\begin{corollary}\label{cor:nilexamp}
	Let $N$ be an asymmetric $r$-step Carnot rational nilpotent group. Let $\Gamma$ be the lattice and $f_0$ the Anosov automorphism of $M=(N\times N)/\Gamma$ obtained from Example \ref{ex:alg-Anosov}. Then there is a $C^1$ open neighborhood $U$ of $f_0$ in $\Diff^\infty(M)$ such that if $f \in U$ admits a smooth splitting $E^u_f=E^u_{fast,f}\oplus E^u_{slow,f}$ along unstable leaves and $E^u_{slow,f}$ is sufficiently uniformly $C^{r}$ close along unstable leaves to $E^u_{slow,f_0}$, then 

	for any invariant ergodic measure $\mu$ there is $\lambda_\mu>1$ such that the unstable Lyapunov exponents of $f$ with respect to $\mu$, are $\log\lambda_\mu,2\log\lambda_\mu,\dots,(r+1)\log\lambda_\mu$ with the corresponding multiplicities as for $f_0$.
\end{corollary}

Using Theorem \ref{thm:Hei-spec}, we obtain a similar corollary to the one above, but for an Anosov automorphism of a nilmanifold arising as a quotient of a non-asymmetric nilpotent group.

\begin{corollary}\label{cor:nilexamp2}
	Let $N$ be the 3-dimensional Heisenberg group. Let $\Gamma$ be the lattice and $f_0$ the Anosov automorphism of $M=(N\times N)/\Gamma$ obtained from Example \ref{ex:Smale-Anosov}. Then there is a $C^1$ open neighborhood $U$ of $f_0$ in $\Diff^\infty(M)$ such that if $f \in U$ admits a smooth splitting $E^u_f=E^u_{fast,f}\oplus E^u_{slow,f}$ along unstable leaves and $E^u_{slow,f}$ is sufficiently uniformly $C^{1}$ close along unstable leaves to $E^u_{slow,f_0}$, then 

	for any invariant ergodic measure $\mu$ the unstable Lyapunov exponents for $f$ are $\log \lambda_1,\log \lambda_2$ and $\log \lambda_1+\log \lambda_2$ for some $\la_1>1$ and $\la_2>1$ depending on $\mu$.
\end{corollary}

\subsection{Hyperbolic Rank Rigidity}\label{subsec:intro_hypRank}

A  Riemannian manifold $M$ has {\em higher hyperbolic rank}  if  every geodesic $c(t)$ in $M$ has   a  Jacobi field $J(t)$ that makes sectional curvature $\kappa \cong -1$ with $c'(t)$ where $J(t) \neq 0$.
	Clearly, all rank one locally symmetric spaces of negative curvature have higher hyperbolic rank after a trivial rescaling of the metric. We made the following conjecture regarding the converse in our recent paper \cite[Conjecture 1.3]{CNS2018}:
	
\begin{conjecture}\label{conj:hyp}
	A closed Riemannian manifold with sectional curvatures $\kappa\geq -1$ has higher hyperbolic rank only if it is locally symmetric.
\end{conjecture}
	
In that paper, we also proved the special case of Conjecture \ref{conj:hyp} when the sectional curvatures satisfy  $-\frac 14 \geq \kappa \geq -1$ \cite[Theorem 1.1]{CNS2018}.  Using ergodicity of 2-frame flows, Constantine \cite[Corollary 1]{Constantine08} had already characterized constant curvature manifolds by the hyperbolic rank condition for either odd dimensional manifolds, just assuming non-positive curvature, or under strong pinching assumptions on the curvature, $-(.93)^2 \geq \kappa \geq -1$.  We refer to \cite{CNS2018} for more historical discussion, in particular of Hamenst\"{a}dt's hyperbolic rank rigidity theorem when the curvature is bounded above by -1. 

The current paper began as a sequel to our previous work on the hyperbolic rank rigidity conjecture \cite{CNS2018}, and indeed we do achieve new results towards Conjecture \ref{conj:hyp}. Along the way, we also advance techniques surrounding the Brin-Pesin asymptotic holonomy group for frame flows. We  further establish various versions of local rigidity of geodesic flows near the rank one locally symmetric ones.

One of the main results of the current paper, arising as an application of the results above, is the following local rigidity result for perturbations of locally symmetric metrics with higher hyperbolic rank.

\begin{theorem}\label{thm:local-hyp}
	Let $(M, g_0)$ be a closed quaternionic or octonionic hyperbolic locally symmetric manifold. Then there is an open $C^3$ neighborhood $U$ of $g_0$ such that for any $g\in U$, if $(M,g)$ has higher hyperbolic rank and $\kappa_g\ge -1$ then $(M,g)$ is locally symmetric. 
\end{theorem}

By Mostow Rigidity, it follows that $(M,g)$ is isometric to $(M,g_0)$. This of course fails for hyperbolic surfaces. Nevertheless, local rigidity in the real hyperbolic case already follows from our previous paper \cite{CNS2018}.

Our methods do not give us full local rigidity of higher hyperbolic metrics nearby the complex hyperbolic metric. However, we still obtain equality of the Liouville measure and the Bowen-Margulis measure (the unique measure of maximal entropy) for higher rank perturbations. 

\begin{theorem}\label{thm:complexcase}
	Let $(M,g_0)$ be a closed complex hyperbolic manifold. There is an open neighborhood $U$ of $g_0$ in the $C^3$-topology among $C^\infty$ metrics with the following property: for $g\in U$ with higher hyperbolic rank and sectional curvature $\kappa\ge -1$, the Liouville measure on $SM$ coincides with the (unique) measure of maximal entropy for the geodesic flow of $g$ on $SM$.
\end{theorem}

\subsection{Local Rigidity of Projective Actions}\label{subsec:intro_boundary}

Sullivan initiated the study of local rigidity of lattice actions on boundary spheres in \cite{Sullivan85}. There have since appeared a number of results of this kind in various contexts (see e.g. \cite{Ghys93,Yue95,Asaoka17}). We use our methods to establish local rigidity for certain lattice actions on spheres. First recall that the identity component $G$ of the isometry group of a symmetric space acts by diffeomorphisms on its ideal boundary sphere preserving a smooth distribution $E_0$, the projection of the slow distribution. This action may also be described in an algebraic way as follows. If $P<G$ is any minimal parabolic subgroup, then we may identify the ideal boundary sphere with $G/P$ on which $G$ naturally acts on the left. We will refer to this action by $G$, or any subgroup of $G$, as a projective action.

\begin{theorem}\label{thm:boundary_rigid}
	Let $\rho_0:\Gamma\to \op{Diff}^\infty(S^{k})$ for $k=4n-1$ (resp. $k=15$) be the projective representation of a cocompact lattice $\Gamma<Sp(n,1)$ (resp.  $\Gamma<F_4^{-20}$). Let $\rho:\Gamma\to \op{Diff}^\infty(S^{k})$ be a $C^1$ close perturbation of  $\rho_0$. If $\rho$ preserves a $C^\infty$ distribution $E$, $C^1$ close to $E_0$, then $\rho$ is $C^\infty$ conjugate to $\rho_0$.
\end{theorem}

We note that for a sufficiently small $C^1$ perturbation $\rho$ of $\rho_0$ there always exists a distribution $E$ which is $C^0$ close to $E_0$ by stability of dominated splittings (see Section \ref{subsec:dom}).
  
\vspace{1em}  

\bigskip
\noindent {\em Acknowledgments:} We thank Boris Kalinin, Karin Melnick and Victoria Sadovskaya for discussions about normal forms, Jeremy Tyson and Anton Lukyanenko for discussions on \CC metrics, Boris Hasselblatt and Kurt Vinhage for discussions on the regularity of distributions, Livio Flaminio for discussions on local entropy rigidity for complex hyperbolic space, Tracy Payne for discussions on nilmanifolds and in particular for pointing us to the Lauret examples, and finally Clark Butler for discussions on Lyapunov exponent rigidity. We thank the Departments of Mathematics at Indiana University and the University of Michigan for their hospitality while parts of this work were completed.

\section{Background}\label{section:background}

\subsection{Dynamics: Dominated Splittings}\label{subsec:dom}

We begin by reviewing some basic definitions and results about dominated splittings. We will be using these in some of the arguments below. 
\begin{definition}\label{def:DS}
	Let $V\to X$ be a $C^0$ vector bundle equipped with a continuously varying norm on fibers over a compact smooth manifold $X$ and $\La\subset X$ be an closed invariant subset for a flow $\varphi_t:X\to X$. We say that $\La$ admits a {\em dominated splitting of index $k$} for a cocycle $\alpha:V\times \R \to V$ over $\varphi_t$ provided:
	\begin{enumerate}
		\item $V$ restricted to $\La$ splits into two continuous subbundles $V\rest{\La}=E\oplus F$, both invariant under $\alpha_t$ and with $\dim E=k$,  and
		\item there exist $C>0$ and $0<\la<1$ such that for any $x\in\La$,
		\begin{equation}\label{eq:DS_norm}
			\norm{\alpha_t\rest{E_x}} \norm{\alpha_{-t}\rest{F_{\varphi_t(x)}}}\le C\la^t \quad\text{for all}\quad t\ge 0.
		\end{equation}
	\end{enumerate} 
\end{definition}

Morally, this definition states that any vector not in $E$ converges uniformly exponentially fast to the subbundle $F$ under $\alpha_t$.

Examples of dominated splittings include:
\begin{itemize}
	\item  Anosov geodesic flows $\varphi_t$ on $X$ where $V=TX$, $E=E^{cs}=E^s\oplus  E^c$ the center (weak) stable, $F=E^u$ the strong unstable and $\alpha=d\varphi_t$,
	\item  geodesic flows of locally symmetric manifolds $M$ of nonconstant negative curvature restricted to the unstable bundle where $X=SM$, the unit tangent bundle,  $V=E^u$ the strong unstable, $E=E^u_{slow}$, $F=E^u_{fast}$ and $\alpha=d\varphi_t$,
	\item a one parameter semisimple subgroup $\varphi_t$ of a semisimple Lie group where $X=G/\Gamma$ for a cocompact lattice $\Gamma<G$, $V=E^u$ the strong unstable, $E=E^u_{<\la}$, $F=E^u_{\geq \la}$ for any $\la$ in the interior of the Lyapunov spectrum and $\alpha=d\varphi_t$,
	\item geodesic flows of nonconstant negatively curved manifolds with higher hyperbolic rank where $V=E^u$, $E=E^u_{slow}$, $F=E^u_{fast}$ and $\alpha=d\varphi_\cdot$ (cf. Section \ref{sec:hyprank}),
\end{itemize}

An important property of dominated splittings is that they are stable under perturbations (cf. Proposition 2.3 of \cite{Sambarino} for the case of a derivative of a map and Corollary 2.8 of \cite{CrovisierPotrie} applied to the time one map of a flow). The proofs of the following proposition were written for the case of tangent bundles and derivative cocycles but apply equally well to the case of general vector bundles and cocycles as in Definition \ref{def:DS} (see also \cite{BochiPotrieSambarino19}).

\begin{proposition}[\cite{Sambarino} and \cite{CrovisierPotrie}]\label{prop:dom-perturb}
	Any sufficiently small $C^0$ perturbation in the space of a cocycle admitting a dominated splitting also admits a dominated splitting of the same index. 

\end{proposition}

\subsection{Algebra: Varieties of Lie Algebras and Automorphism Groups}\label{subsec:nil_structures}

The space of $n$-dimensional Lie algebras can be identified with the space of Lie brackets on $\R^n$. Thus it is
$$\mc{L}_n=\set{c\in \Hom(\Lambda^2\R^n,\R^n):c(c(u,v),w)+c(c(v,w),u)+c(c(w,u),v)=0~\forall u,v,w\in \R^n }.$$

Since the defining condition is a quadratic polynomial in the structure coefficients of the bracket $c$ relative to the standard basis of $\R^n$, $\mc{L}_n$ is an affine subvariety of $\Hom(\Lambda^2\R^n,\R^n)$. 
The space of $n$-dimensional nilpotent Lie algebras is
$$\mc{NL}_n=\set{c\in \mc{L}_n: c(c(\dots c(u_1,u_2),u_3),\dots),u_n)=0 \text{ for all } u_1,\dots,u_n\in \R^n}.$$

This is an algebraic affine subvariety of $\mc{L}_n$ since the defining condition is polynomial in the structure coefficients. 

Two Lie algebras on $\R^n$ are isomorphic if and only if there is general linear transformation of $\R^n$ that maps one Lie bracket to the other Lie bracket. Thus, the moduli space of isomorphism classes of nilpotent Lie algebras of dimension $n$ is $\mc{NL}_n/\GL(n,\R)$.

We will need the following in Section \ref{sec:subRiem}.
\begin{theorem}[\cite{BorelSerre64} cf. Theorem 3.1.3 of \cite{Zimmer}]\label{thm:BorelSerre}
Orbits of real points of an algebraic group acting on the real points of a real algebraic variety are locally closed in the Hausdorff topology.
\end{theorem}

As a consequence, the orbits of $\GL(n,\R)$ on $\mc{NL}_n$ are locally closed and $\mc{NL}_n/\GL(n,\R)$ is countably separated and hence Hausdorff.

We now discuss the automorphism group of a Lie group. We start with the following proposition.
\begin{proposition}
	Let $G$ be a connected and simply connected Lie group with Lie algebra $\mf{g}$. Then the automorphism group $\Aut(G)$ is naturally isomorphic as a Lie group to $\Aut(\mf{g})$ and the latter is a real algebraic variety in $\GL(\mf{g})$.
\end{proposition}

\begin{proof}
	For connected $G$, $\Aut(G)$ is a closed subgroup of $\Aut(\mf{g})$ with Lie algebra $\Der(\mf{g})$ and these coincide when $G$ is simply connected (cf. Theorem 2.75 of \cite{Varadarajan}).  Picking a basis of the Lie algebra, a Lie algebra automorphism is completely determined by its associated matrix with respect to this basis. The bracket relations on basis elements give us quadratic polynomial relations in these matrix coefficients. Moreover, any matrix satisfying these relations induces an automorphism. Hence, these polynomial relations provide $\Aut(\mf{g})$ with the structure of a real affine variety in $\GL(\mf{g})\subset \mb{P}_{\R}^{(\dim \mf{g})^2}$.
\end{proof}

A nilpotent Lie algebra $\mf{g}$ is {\em graded} if it admits a decomposition $\mf{g}=\bigoplus_{i=0}^r \mf{g}^i$ where $[\mf{g}^i,\mf{g}^j]\subseteq \mf{g}^{i+j}$ for all $i,j \geq 0$ with $i+j\leq r$ and $[\mf{g}^i,\mf{g}^j]=0$ otherwise. We call a nilpotent Lie algebra $\mf{g}$ {\em Carnot} if it is graded and $[\mf{g}^i,\mf{g}^j] = \mf{g}^{i+j}$ for all $i,j \geq 0$ and $i+j\leq r$. A nilpotent Lie group $G$ is called {\em  graded} (resp. {\em Carnot}) if its Lie algebra is graded (resp. Carnot). We also call an automorphism of a graded Lie algebra $\mf{g}$ graded if it respects the grading. Finally an automorphism of a connected graded nilpotent group $G$ is {\em graded} if its corresponding automorphism in $\Aut(\mf{g})$ is graded. 

\begin{proposition}\label{prop:Aut_algebraic}
	Let $G$ be a connected and simply connected graded nilpotent Lie group with Lie algebra $\mf{g}=\bigoplus_{i=0}^r \mf{g}^i$. Then the graded automorphism group $\Aut_g(G)$ is isomorphic as a Lie group to the graded automorphism group $\Aut_g(\mf{g})$ of the Lie algebra and the latter is a real algebraic variety in $\bigoplus_{i=0}^r \GL(\mf{g}^i)<\GL(\mf{g})$. In particular, $\Aut_g(\mf{g})$ is a real algebraic group.
\end{proposition}

\begin{proof}
	It is clear that automorphisms preserving the grading are a subgroup of automorphisms lying in $\bigoplus_{i=0}^r \GL(\mf{g}^i)<\GL(\mf{g})$. It is easy to check that the subgroup $\Aut_g(\mf{g})<\Aut(\mf{g})$ corresponds to the subgroup $\Aut_g(G)<\Aut(G)$ under the natural isomorphism given by the previous proposition. It remains to show that $\Aut_g(\mf{g})$ is a subvariety of $\Aut(\mf{g})$. However, as the graded subspaces are linear, the additional defining equations for the corresponding matrix elements are linear relations. Hence $\Aut_g(\mf{g})$ is also a real algebraic subvariety. Lastly, groups which are real algebraic subvarieties of an algebraic group are themselves algebraic groups.
\end{proof}

When $G$ is a simply connected graded nilpotent Lie group we can define a one parameter group of automorphisms $\delta_s:G\to G$ called the {\em  dilations}  whose associated Lie algebra automorphisms $\delta_s:\mf{g}\to\mf{g}$ are the linear maps given by $\delta_s(x)=s^i x$ for all $x\in \mf{g}^i$. 
We denote the group of all dilations by $\mc{D}=\set{\delta_s: s\in \R_+}$ and observe that it is isomorphic to the multiplicative group $\R_+$. The dilations are always in the center of $\Aut_g(\mf{g})$ since they are multiples of the identity on each level.

If $\mf{g}$ is a Carnot Lie algebra then any automorphism in $\Aut_g(\mf{g})<\bigoplus_{i=0}^r \GL(\mf{g}^i)$ is completely determined by its component in $\GL(\mf{g}^0)$.  Consequently we may identify $\Aut_g(\mf{g})$ with a subgroup of $\Aut_g(\mf{g}^0)$ and write $\Aut_g(\mf{g})=L \mc{D}$ where $L<\SL_{\pm}(\mf{g}^0)$. (Here $\SL_{\pm}(\mf{g}^0)$ is the subgroup of $\GL(\mf{g}^0)$ with elements of determinant $\pm 1$.) Henceforth we set $d_i=\dim \mf{g}^i$ for $i=0,\dots,r$.

\begin{definition}\label{def:asymmetric}
We call a graded nilpotent Lie group $G$, and its Lie algebra $\mf{g}$, \emph{asymmetric} if its graded automorphism group has the form $\Aut_g(\mf{g})=L\mc{D}$ where $L$ belongs to a compact subgroup of $\SL_{\pm}(d_0,\R)$. 
\end{definition}

Recall that a surjective map $f:X\to Y$ between metric spaces is a homothety if $d(f(x),f(y))=C d(x,y)$ for some constant $C>0$ and every $x,y\in X$. If a graded nilpotent Lie group $G$ is Carnot, then we may equip $G$ with a left-invariant \CC metric arising from the right invariant distribution corresponding to $\mf{g}^0$ equipped with any inner product. In this case, the dilations $\delta_s$ are homotheties as they simply act by the appropriate scalar multiple of the identity on each level preserving the bracket relations. While other nontrivial automorphisms may be homotheties, for example those that restrict to an isometry on the first level, in general they need not all be homotheties. The next lemma shows that being asymmetric is equivalent to $\Aut_g(\mf{g})$ consisting of homotheties for some \CC metric.

\begin{lemma}\label{lem:asymmetric_homotheties}
Let $\mf{g}$ be a Carnot Lie algebra equipped with a right invariant \CC metric. If $\Aut_g(\mf{g})$ consists entirely of homotheties, then $\mf{g}$ is asymmetric. Conversely, if $\mf{g}$ is an asymmetric Carnot Lie algebra then $\Aut_g(\mf{g})$ consists of homotheties with respect to some right invariant \CC metric.
\end{lemma}

\begin{proof}
If $\Aut_g(\mf{g})=L\mc{D}$ consists of homotheties of a \CC metric, then $L<\SL_{\pm}(\mf{g}^0)$ and each element of $L$ is uniquely determined by its action on $\mf{g}^0$. Moreover, each element of $L$ must be an isometry with respect to the given inner product $\inner{\cdot,\cdot}_0$ on $\mf{g}^0$. Hence $L$ belongs to the compact group $\OO(d_0,\inner{\cdot,\cdot}_0)$.

Conversely, if $\Aut_g(\mf{g})=L\mc{D}$ where $L<K$ for some compact subgroup $K<\SL_{\pm}(d_0,\R)$, then $K$ is conjugate into $\OO(d_0,\R)$. Equivalently, $K$ belongs to $\OO(d_0,\inner{\cdot,\cdot}_0)$ for the inner product on $\mf{g}^0$ conjugate to the standard one. This inner product then induces a \CC metric for which $\Aut_g(\mf{g})$ are homotheties.
\end{proof}

Asymmetric Carnot nilpotent groups may at first appear to be somewhat special. However, there are many examples and these algebras are generic in many situations. 

\begin{examples}\label{examp:nilrigid}
The following classes of graded nilpotent Lie groups $G$ are asymmetric.
\begin{enumerate}
\item (\cite[Proposition 13.1]{Pansu}) $G$ belongs to a certain Zariski open dense subset of all 2-step groups with grading $\mf{g}=\mf{g}^0\bigoplus \mf{g}^1$ where $\dim \mf{g}^0\geq 10$ and is even and $3\leq \dim \mf{g}^1 \leq 2 \dim \mf{g}^0-4$,
\item (\cite[Proposition 10.1]{Pansu}) $G$ is the (2-step) maximal unipotent subgroup of the isometry group of $\mb{H}_{\mb{H}}^n$ or $\mb{H}_{\KO}^2$ for any $n\geq 2$, 
\item (\cite[Section 13]{Pansu}) $G$ has an exceptional $r+1$-step {\em filiform} algebra of even dimension $r+3$ given by $\mf{g}=\bigoplus_{i=0}^{(r+1)}\mf{g}^i$ with $\mf{g}^0=\Span\set{y_0,z_0}$ and $\mf{g}^i=\Span\set{y_{i}}$ for $1\leq i\leq r+1$ with bracket relations $[z_0,y_i]=y_{i+1}$ and $[y_i,y_{r-i}]=(-1)^i y_{r+1}$ for $0\leq i\leq r$ and all other brackets $0$. (The filiform algebras are the graded algebras with minimal possible dimensions of each grading strata, there are two in each even dimension and one in each odd dimension \cite{Vergne}. However the nonexceptional ones admit nontrivial unipotent graded automorphisms.) Or,
\item (\cite[Theorem 1.2]{LeDonneEtAl14}) any $G$ with $\Aut(G)=\mc{D}$. Such groups are called {\em ultrarigid} nilpotent groups. There are examples of these in at least dimensions 16 and 17 (Examples 3.4 and 3.5 in the cited paper).
\end{enumerate}
\end{examples}

We do not know if asymmetry is always a generic property of Carnot nilpotent groups with respect to the Zariski topology in sufficiently high dimensions. It it is an open condition in the Hausdorff topology (see Corollary \ref{cor:stable-asymmetry}).

We observe that the underlying space of all Lie algebras $\mf{g}\in \mc{NL}_n$ is $\R^n$. In particular we have $\Aut{\mf{g}}<\GL(\mf{g})=\GL(n,\R)$. We will thus use the topology of pointed Hausdorff convergence on (closed) subgroups of $\GL(n,\R)$ which by abuse of notation we call the ``Hausdorff topology.'' In this topology, a sequence of groups $G_i$ converges to $G$ if the Hausdorff distance $d_H(G_i\cap B(1,r),G\cap B(1,r))$ tends to $0$ as $i\to\infty$ for each $r>0$, where $B(1,r)$ is the ball of radius $r$ around the identity in $\GL(n,\R)$ with respect to the operator norm.

We now show that automorphism groups vary semicontinuously in this topology for nilpotent $G$. 
\begin{proposition}\label{prop:auto-semicont}
Let $(G_k)$ be a sequence of connected $n$-dimensional graded nilpotent Lie groups whose Lie algebras converge to the Lie algebra $\mf{g}$ of $G$ in $\mc{NL}_n$. Then the corresponding graded automorphism groups $Aut_g(\mf{g}_k)$ converge to a Lie subgroup of $Aut_g(\mf{g})<\GL(\mf{g})$ with respect to the topology of pointwise convergence.
\end{proposition}
\begin{proof}
	 By continuity of the defining polynomial relations satisfied by the matrices in $\Aut_g(\mf{g}_k)$, any element of the limit will satisfy the polynomial relations required to be in $\Aut_g(\mf{g})$. Let $H\subset  \Aut_g(\mf{g})$ be the collection of limit points. By definition this is a closed subset of $\Aut_g(\mf{g})$. Considered as equivalence classes of convergent sequences, it is elementary to verify that the product structure given by $[\set{a_n}]*[\set{b_n}]=[\set{a_n b_n}]$ is well defined and gives $H$ a group structure compatible with that of $\Aut_g(\mf{g})$. Hence $H$ is a Lie subgroup of $\Aut_g(\mf{g})$.
\end{proof}

As a corollary, we deduce that the property of being asymmetric is stable in the variety of nilpotent Lie groups/algebras.
\begin{corollary}\label{cor:stable-asymmetry}
Let $G$ be a connected and simply connected Carnot Lie group of dimension $n$. If $G$ is asymmetric then there is a neighborhood $\mathcal U$ of $G$ in the variety of $n$-dimensional Carnot Lie groups such that $\mathcal U$ consists of asymmetric groups.
\end{corollary}

\begin{proof}
	Let $(\mf{g}^k)$ be a sequence of Carnot Lie algebras converging to $\mf{g}$ in $\mc{NL}_n$. Write  $\Aut_g(\mf{g})=L\mc{D}$ where $L=\Aut_g(\mf{g})\cap \SL_{\pm}(d_0,\R)$ lies in some compact group $K$. Without loss of generality, we may assume after a conjugation that $L<K=\OO(d_0,\R)$. Similarly, let $\Aut_g(\mf{g}^k)=L_k\mc{D}$ for $L_k=\Aut_g(\mf{g}^k)\cap \SL_{\pm}(d_0,\R)$.
	
	By Proposition \ref{prop:auto-semicont}, the $L_k$ converge to a subgroup of $L$. We wish to show that $L_k$ belongs to a compact subgroup. We will first consider $L_k^0$, the connected component of the identity of $L_k$. We fix a norm on $\R^n$. If the $L_k^0$ are not contained in a compact group, then the operator norm, $\norm{\cdot}_{op}$ is unbounded on $L_k^0$. Then $L_k^0$ admits a path whose elements have norms taking on all values in $[1,\infty)$, and in particular there is an element $a_k$ with $\norm{a_k}_{op}=2$. The sequence $\set{a_k}$ admits a convergent subsequence to an element of determinant $\pm 1$, but which does not have norm one, leading to a contradiction.
	
	By Proposition \ref{prop:Aut_algebraic}, $\Aut_g(\mf{g}^k)$ and thus $L_k$  are algebraic groups. An algebraic group has only finitely many connected components by Theorem 3 of \cite{Whitney}. Therefore $L_k$ must lie in a compact subgroup.
\end{proof}

\section{Tangent Cones of Sub-Riemannian Manifolds}

In this section we analyze the structure of tangent cones and relate the derivatives of smooth maps preserving the horizontal distribution to the Carnot derivatives. Under suitable assumptions, this leads to resonances and even arithmeticity of the Lyapunov spectrum. 

Throughout this section we will be exploiting the properties of sub-Riemannian manifolds equipped with the generic and horizontal distribution as defined in Section \ref{subsec:intro_subRiem}. Toward this end, we first recall the notion of a tangent cone. 

\begin{definition}
Let $(X,d)$ be a metric space and $x\in X$. We call any pointed Gromov-Hausdorff limit of $(X,p, td)$ when $t\to \infty$ a {\em tangent cone} of $X$ at $p$.
\end{definition}

In general, tangent cones of metric spaces at a point may not exist, and when they do, they may not be unique. However, in the setting of sub-Riemannian manifolds, Mitchell \cite{Mitchell} showed that they exist, and Margulis and Mostow \cite{MargulisMostow00} proved that they are unique.

\begin{theorem}[Theorem 1 of \cite{Mitchell}, \cite{MargulisMostow00}]\label{thm:Mitchell}
Let $(N,d_C)$ be a sub-Riemannian manifold.  Suppose that the metric $d_C$ comes from a smooth horizontal distribution $E$. If the distribution $E$ is  generic in a neighborhood of $p\in N$ then the tangent cone of $N$ at $p$ exists, is unique and is a Carnot nilpotent Lie group with a right-invariant \CC metric induced from the distribution of homogeneous degree one vector fields tangent to $E$ with the inner product inherited from $E$. 
\end{theorem}

The very last statement of the above theorem is not stated explicitly in the cited references, but follows from the construction as will be explained in the next subsection (\ref{subsec:Carnot}). There we will also define the notion of homogeneous vector fields.

We next discuss the construction and structure of this nilpotent group and its relationship to the tangent cone at each point $p\in N$ (cf. \cite[Section 3]{Mitchell}). We will need these details for our application to sub-Riemannian dynamics in Section \ref{sec:subRiem}.

\subsection{Carnot Metrics and the M\'etivier Correspondence}\label{subsec:Carnot}

M{\'e}tivier developed a correspondence between iterated brackets of horizontal vector fields in a neighborhood of a given point $p\in N$ and certain homogeneous vector fields possessing nilpotent bracket relations. We here present a compressed summary of the details.

Let $X_i$, $i=1,\dots,d_0=\dim E$ be smooth vector fields on $N$ that are tangent to and form a basis for the horizontal distribution in a neighborhood $U$ of a point $p\in N$. After possibly shrinking the neighborhood $U$, among the commutators of the $X_i$ we may choose an ordered subset $Y_1,\dots,Y_n$ that form a basis of $T_yN$ at each point $y\in U$ and such that $Y_1=X_1,\dots,Y_{d_0}=X_{d_0}$, and together with the next ordered subcollection ($Y_{d_0+1},\dots,Y_{d_0+d_1}$) form a basis for the first commutator subbundle $E^1(y)$, and so on. We call such a moving frame a {\em graded basis}. We sometimes index this basis according to the level as $(Y_{0,1},\dots,Y_{r,d_r})$.

For any smooth vector field $X$ on $N$, let $s\mapsto \exp s X$ be the one parameter subgroup of diffeomorphisms acting on $N$ induced by the flow of $X$, i.e. along integral curves of $X$. For each $y$ in the neighborhood $U$ of $p$ the map $y=\exp(\sum_i (y)_i Y_i)(p)\mapsto ((y)_i)\in \R^n$ is a diffeomorphism from $U$ to some neighborhood of $0$ in $\R^n$. We call the collection of components $(y)_i$ {\em normal coordinates} on $U$. We emphasize that normal coordinates do not define (commuting) coordinates in the ordinary sense since the distribution is not integrable. 

For a generic horizontal distribution $E$, we define the {\em partition floor function} to be $[i]_E:=k+1$ if $\dim E^{k-1}(x)<i\leq \dim E^k(x)$ for $k=0,\dots,r$ and where we adopt the convention that $\dim E^{-1}=0$. When the distribution $E$ is understood we simply write $[i]$ instead of $[i]_E$. We define a dilatation $\delta_s:U\to N$ for $s\in [0,\infty)$ in terms of normal coordinates by taking $\delta_s(y_i)=s^{[i]}y_i$. The $\delta_s$ are local diffeomorphisms for $s>0$.

For a vector field of the form $X=y_{i_1}^{\alpha_1}\dots y_{i_k}^{\alpha_k}Y_j$, written in terms of the fields $Y_j$ and their associated normal coordinate functions $y_j$, we define the {\em degree} of $X$  to be $[j]-\sum_{m=1}^k \alpha_m [i_m]$. By taking the maximum degree over monomial terms we can extend the notion of degree to vector fields of the form $X=\sum_i p_i(y_1,\dots,y_n) Y_i$  for polynomials $p_i$. Finally, we can extend this notion to arbitrary smooth vector fields in a neighborhood of $p$ by considering the degree of vanishing at $p$ (see Section 3 of \cite{MargulisMostow00} for details). A vector field is {\em homogeneous} of degree $i$ if each of its component terms has degree exactly $i$. We can decompose each vector field $X$ into its homogeneous components. If $X$ has degree at most $q$ and $X^{(q)}$ denotes the homogeneous degree $q$ part of $X$ at $p$, and similarly $Y$ has degree at most $s$ at $p$ and $Y^{(s)}$ is its homogeneous degree $s$ part, then by Section 3 of \cite{MargulisMostow00} the bracket $[X,Y]$ has degree at most $q+s$ at $p$ and moreover,
\begin{align}\label{eq:brackhat}
[X,Y]^{(q+s)}=[X^{(q)},Y^{(s)}].
\end{align}

For each horizontal vector field $X$ on a neighborhood of $p$, we obtain a homogeneous degree one vector field $\what{X}_p=\lim\limits_{t\to \infty} t^{-1}(\delta_t)_*(X)$ on $N$ (see equation (3.2) on p.487 of \cite{Metivier}).
The vector field $\what{X}_p$ is precisely the homogeneous degree 1 part at $p$ of the horizontal vector field $X$. Moreover, the following lemma shows that if $X_1$ and $X_2$ are two horizontal fields with $X_1(p)=X_2(p)$, then $\what{X}_{1;p}=\what{X}_{2;p}$.

\begin{lemma}\label{lem:constantsum}
If $X_1,\dots,X_k$ is a basis of horizontal vector fields in a neighborhood of $p\in N$ for a sub-Riemannian manifold $N$, and $X=\sum_{i=1}^{d_0} a_i X_i$ is any horizontal vector field on $N$, then $\what{X}_p=\sum_{i=1}^{d_0} a_i(p) \what{X}_{i;p}$.
\end{lemma}

\begin{proof}
Since $X$ is horizontal it has degree at most one, and we can write it in the basis as $X(y)=\sum_i a_i(y) X_i$ for functions $a_i$ on the neighborhood $U$ of $p$. Express this as $X(y)= \sum_i a_i(p)X_i+ \sum_i (a_i(y)-a_i(p)) X_i$. Since the functions $a_i(y)-a_i(p)$ vanish at $p$, they have degree at most $-1$ at $p$ and $X_i$ have degree $1$. Hence $\sum_i (a_i(y)-a_i(p)) X_i$ has degree at most $0$ and $\what{X}_p=\sum_i a_i(p)\what{X}_{i;p}$. 
\end{proof}

M\'etivier \cite[Th\'eor\`eme 3.1]{Metivier} showed that the homogeneous vector fields $\what{X}_{i;p}$ generate a nilpotent Lie algebra through taking commutators, and Mitchell \cite[Lemmas 3.1 and 3.2]{Mitchell} showed that the corresponding simply connected nilpotent group with its corresponding \CC metric is isometric and isomorphic to the tangent cone at $p$. By Lemma \ref{lem:constantsum}, for any horizontal field $X$ we have $\what{X}_p$ is in the span of $\what{X}_{i;p}$.

The homogeneous degree one part $\what{X}_p$ of a horizontal field $X$ thus corresponds to a right invariant field $\bar{X}$ on $TC_pN$. In other words, $\bar{X}$ belongs to the the Lie algebra of $TC_pN$, which we denote by $\mf{n}_p=\op{Lie}(TC_pN)$. Let $\mf{n}^0_p\subset \mf{n}_p$ be the subspace spanned by the set of all $\bar{X}$, and inductively set $\mf{n}^i_p=[\mf{n}^{0}_p,\mf{n}^{i-1}_p]\subset \mf{n}_p$ for $i\in\set{1,\dots,r}$. The degree of homogeneity of the vector fields provides a natural grading for the Lie algebra $\mf{n}_p$ as pointed out in Section 8.2 of \cite{MargulisMostow95} and Section 4.0 of \cite{MargulisMostow00}, furnishing $\mf{n}_p$ with a Carnot nilpotent structure. 

Among the $k$-fold commutators of $\what{X}_{i;p}$ we may choose a set of vector fields $\what{X}_{k,j;p}$ whose corresponding right invariant fields $\bar{X}_{k,j}$ form a basis of the $k$-th commutator space $\mf{n}^{k}_p\subset \mf{n}_p$ of $\mf{n}^0_p$ respecting the grading. In particular, we may choose indices so that $\what{X}_{0,j;p}=\what{X}_{j;p}$. This provides an isomorphism of Lie algebras, $\mc{M}_p: \op{span}\set{\what{X}_{0,1;p},\dots,\what{X}_{r,d_r;p}}\to \mf{n}_p$, where the domain is a Lie algebra of vector fields under the ordinary bracket. We name this identification between linear combinations of the $\set{\what{X}_{k,j;p}}$ and elements of $\mf{n}_p$ the {\em M\'etivier correspondence} at $p$. Note that this correspondence depends on the choice of graded basis. By Lemma \ref{lem:constantsum}, the M\'etivier correspondence induces a linear inclusion $\imath:E(p)\to \mf{n}_p$ with image $\mf{n}^0_p$ by $X(p)\mapsto \bar{X}=\mc{M}_p(\what{X}_p)$. 
We will omit the ``$p$'' subscript on vector fields when the base point of the tangent cone is understood.

Following \cite{MargulisMostow00}, define an equivalence class on curves of $M$ by $\alpha\sim\beta$ if and only if 
\[
\lim\limits_{s\to 0} \frac{1}{s}d_C(\alpha(s),\beta(s))=0.
\]
 The tangent cone $TC_pN$ can be considered as the space of equivalence classes of curves $[\alpha]$ with $\alpha(0)=p$ which are  equivalent to a dilatation orbit curve at $p$, $[\delta_s y]$ for some $y\in N$, with the distance
\begin{align}\label{eq:dCC}
d([\alpha], [\beta]) = \lim_{s\to 0} s^{-1} d_C(\alpha(s),\beta(s)).
\end{align}

The above interpretation generalizes a standard interpretation of the tangent space to a differentiable manifold as the space of equivalence classes of differentiable curves.

In the next proposition we will consider the subspace topology on $\mc{NL}_n\subset \Hom(\Lambda^2\R^n,\R^n)$ instead of the coarser Zariski topology of the variety.

\begin{proposition}\label{prop:cont_of_tan_cone}
Let $(N,d_C)$ be a smooth sub-Riemannian manifold. Suppose that the metric $d_C$ comes from a generic smooth horizontal distribution $E$. Then the metrics, nilpotent structures, and the corresponding isomorphism classes, of the tangent cones of $N$ vary continuously on $N$.
\end{proposition}
\begin{proof}
	We need only show that the isomorphism class of the tangent cones vary continuously locally. By Theorem \ref{thm:Mitchell}, the nilpotent structures depend only on the horizontal distribution, not the choice of basis above used to construct the structure. Hence we may choose the initial basis freely. We now refer back to the notation at the beginning of Section \ref{subsec:Carnot}.
	
	For a neighborhood $U$ of a point $p\in N$, by the genericity assumption, we may choose the basis of fields $X_i$ for $i=1,\dots,d_0$ so that the commutators of $X_i$ in $E^k$ do not degenerate to $E^{k-1}$ at any point of $U$ if they do not do so at the point $p$. In particular the same indices can be used to choose the same fields $Y_i$, for $i=1,\dots,n$ at each point $x\in U$. Therefore, the construction of the $\what{X}_{k,j;x}$ all vary smoothly in a neighborhood of $p$. This implies that the structure constants of the Lie algebras of the tangent cones vary smoothly on $U$ and hence this implies that the nilpotent structure varies continuously on the space of nilpotent structures described in \ref{subsec:nil_structures}. This continuity thus passes to the orbit quotient space of nilpotent isomorphism classes as well.
	
	Lastly since the metric on $TC_xN$ is a \CC metric for the distributions induced by the inclusion of $E(x)$ in $\mf{n}_x$, the metric also varies continuously on $U$, and thus on all of $N$.
\end{proof}
\begin{remark}
	Observe that continuity in the subspace topology on $\mc{NL}_n$ equipped with metrics coming from the inner product on $E$ as it varies in an open set implies continuity in the pointed Gromov-Hausdorff topology as it implies convergence on metric nets.
\end{remark}

By Theorem \ref{thm:Mitchell}, the metric on $TC_pN$ is a \CC metric for the right invariant distribution induced from $\mf{n}^0_p\subset \mf{n}_p$. Next, we identify the norm on the distribution that gives rise to the metric.

\begin{lemma}\label{lem:isometricinclusion}
The \CC metric on $TC_pN$ is induced by the right invariant distribution corresponding to $\mf{n}^0_p\subset \mf{n}_p$ equipped with the Riemannian inner product coming from $E(p)$ under the above identification. 
\end{lemma}
 
We will denote the norm on $E(p)$ by $\norm{\cdot}_p$ and that on $\mf{n}^0_p$ by $\norm{\cdot}_{\mf{n}^0_p}$.

\begin{proof}

Given a unit speed sub-Riemannian geodesic segment $\gamma$ starting at $p$, let $X$ be its tangent field extended to a smooth horizontal field on the neighborhood $U$ of $p$. By Proposition 5.6 of \cite{MargulisMostow00}, the \CC metric induced on $TC_pN$ does not depend on the choice of basis vector fields (satisfying certain assumptions). Using $X$ as a coordinate field we obtain,
\[
d_C(p,\delta_s\exp(X)(p))=d_C(p,\exp(s X)(p))=d_C(p,\gamma(s))=s,
\]

This implies in $TC_pN$ that $d([p],[\delta_s \exp( X)(p)])=1$ by \eqref{eq:dCC}. Since the identity element in $TC_pN$ is the class $[p]$ of the constant curve, and $\what{X}=X$ along its integral curve, $[\delta_s \exp( t X)(p)]=e^{t\bar{X}}$ and therefore $d(1,e^{t \bar{X}})=t$. The induced norm on the first level $\mf{n}^0_p$ of the Lie algebra then gives $\|\bar{X}\|_{\mf{n}^0_p}=1$. 

Thus the identification preserves the norm of vectors. Since inner products are uniquely determined by their norms, the inner products are preserved as well.
\end{proof}

\subsection{Carnot Derivatives of Maps Preserving Horizontal Distributions.}

Next we recall the notion of the Carnot derivative which will be crucial to our purposes. This object was first introduced by Pansu in the setting of nilpotent Lie groups in \cite{Pansu} based on earlier work of Mostow on quasiconformal maps of boundaries of rank one symmetric spaces (\cite{Mostow73}). We will need the following version due to Margulis and Mostow in the setting of sub-Riemannian manifolds.

\begin{definition}\cite[Definition 10.3.1]{MargulisMostow95}\label{def:Carnot}
Let $(N,d_C)$ and $(N',d_C')$ be sub-Riemannian manifolds, and let $f:(N,d_C)\to (N',d_C')$ be a map. For every $p\in N$, we denote by $f_t: (N,p,td_C)\to (N',f(p),td_C')$ the induced map for every $t>0$. The function $f$ is called {\em Carnot differentiable} at $p$ if the sequence $f_t$ converges uniformly as $t\to \infty$ on compact sets to a map between tangent cones at $p$ and at $f(p)$. And we call $f$ {\em Carnot differentiable} if it is Carnot differentiable at every point.
\end{definition}

In the next few propositions we explore the relationship between the Carnot derivative and the ordinary derivative.

\begin{proposition}\label{prop:df_isom}
Let $f:(N,d_C)\to (N',d_C')$ be a $C^1$ local diffeomorphism between sub-Riemannian manifolds generated by generic horizontal distributions $E$ and $E'$. If $f$ preserves the distributions, i.e. $Df(E)=E'$, then $f$ is Carnot differentiable everywhere. Moreover, the Carnot derivative $f_*$ is an isomorphism between the corresponding tangent cones.
\end{proposition}

\begin{proof}
	Since the statement is local, it is enough to consider any sufficiently small precompact neighborhood $U$ of an arbitrary point $p\in N$ where $f$ is a diffeomorphism and we use $f\rest{U}$ instead of $f$. We assume $U$ is chosen such that there exist smooth horizontal vector fields $X_1,\dots,X_{d_0}$ on $U$ which form a basis for $E(x)$ at each $x\in U$. We let $X_i'=Df(X_i)$ for each $i=1,\dots,d_0$. By hypothesis, $\set{X_1',\dots,X_{d_0}'}$ forms a basis of $E'(y)$ at each $y\in U'$. As above, these vector fields give us normal coordinates and we use them to define our respective dilatations $\delta_s$ on $U$ and $\delta_s'$ on $U'=f(U)$.
	
	We note that $f:U\to U'$ is a quasi-conformal diffeomorphism since $f$ has uniformly bounded derivatives on the compact set $\bar{U}$. By Theorem 10.5 of \cite{MargulisMostow95} and its proof, $f$ induces an isomorphism of tangent cones at each point $x$ belonging to a certain full measure set $U^0\cap f^{-1}(\Omega_{X_1'}\cap\dots\cap \Omega_{X_{d_0}'})$ to be defined shortly. (Note that the $f^{-1}$ was inadvertently left off of the $\Omega_{X_i'}$ in the proof of \cite{MargulisMostow95}.) We will show that both $U^0$ and $f^{-1}(\Omega_{X_i'})$ for any $i=1,\dots,d_0$ are all of $U$ in our case. 
	
	The definition of $U^0$ given in Section 10.1 of \cite{MargulisMostow95} simplifies in our situation. It consists of the density points for the Hausdorff measure of $d_C$ (cf. Definition 2.1.4 of \cite{MargulisMostow95}) for the set $A$ where the Lipschitz constant of $f$ is uniformly bounded. However, since $f$ is $C^1$ with uniformly bounded derivatives $A=U$. Moreover, every point $x\in U$ is a density point since $U$ is open.

	We now recall the definition of the $\Omega_{X_i'}$ from \cite{MargulisMostow95}. For $i=1,\dots, d_0$, the right invariant vector fields $\bar{X}_i'\in \mf{n}_y$ on $TC_yU'$ correspond to the homogeneous degree one fields $\what{X}_{i;y}'$, namely $\bar{X}_i'=\mc{M}_y(\what{X}_{i;y}')$. Recall the $X_i'$ were used to define the normal coordinates and the dilatation $\delta_s'$. Hence along the integral curve of $X_i'$ through $y$, we have $\frac1{t} (\delta_t')_* X_i'=X_i'$ and hence  $\what{X}_{i;y}'=\lim\limits_{t\to\infty} \frac1{t} (\delta_t')_* X_i'=X_i'$ along the same curve. Let $e^{s \bar{X}_i'}$ denote the one parameter subgroup in $TC_y U'$ of the right invariant field $\bar{X}_i'$, viewed as an element of the Lie algebra  $\mf{n}_y$. Furthermore, following Section 8.6 of \cite{MargulisMostow95}, let $\eps_y$ denote the homeomorphism of a neighborhood of $1\in TC_y U'$ to $U'$ which takes the equivalence class of the path ${\set{\delta_s' z}}_{s\in[0,1]}$ in $TC_y U'$ to the point $z\in U'$. (The equivalence classes of the germs as $s\to 0$ of the paths $\delta_s' z$ correspond to points in the tangent cone by Lemmas 3.1 and 3.2 of \cite{Mitchell}.)
	
	As in \cite{MargulisMostow95}, for $i=1,\dots,d_0$ let $A_{X_i'}$ be the set consisting of all points $y\in U'$ such that $s^{-1} d_C'((\exp s X_i')(y),\eps_y(e^{s \bar{X}_i'}))$ tends to $0$ uniformly as $s\to 0$.  Finally, the set $\Omega_{X_i'}$ consists of all the density points in $A_{X_i'}$ for the Hausdorff measure of $d_C'$. 
	
	Finally observe that since $\what{X}_{i;y}'=X_i'$ along its integral curve through $y$, we have 
	\[
	\delta'_t \exp(s X_i')(y)=\delta'_t \exp(s \what{X}_{i;y}')(y).
	\] 
	This latter curve in $s$ converges in the pointwise Gromov-Hausdorff sense through the metric spaces $(U',y,\frac{d'}{t})$ to the curve $e^{s\bar{X_i'}}\in TC_yN'$ as $t\to 0$. In other words, the curve $\eps_y(e^{s \bar{X_i'}})$ coincides with the curve $\exp(s X_i')(y)$ for every $y\in U'$. In particular, $A_{X_i'}=U'$ and therefore $\Omega_{X_i'}=U'$ since every point of $U'$ is a density point for the Hausdorff measure.

\end{proof}

Diffeomorphisms preserving the horizontal distribution $E$ share certain  properties with their corresponding automorphisms. To investigate this relationship we will need to make certain correspondences more explicit.

\begin{definition}
	Let $f:N\to N'$ be a diffeomorphism of sub-Riemannian manifolds sending the horizontal distribution $E$ to the horizontal distribution $E'$. Assume that $E$ is generic at a point $p\in N$. For a graded basis $Y_1,\dots,Y_n$ of coordinate vector fields on $N$ around $p$ (resp. $Y_1',\dots,Y_n'$ around $f(p)$) with $X_1=Y_1,\dots,X_{d_0}=Y_{d_0}$ horizontal, let $\mc{M}_p:\op{span}\set{\what{X}_{0,1},\dots,\what{X}_{r,d_r}}\to \mf{n}_p$ be the induced M\'etivier correspondence. We then define
	\[
	\what{f}_*:\op{span}\set{\what{X}_{0,1},\dots,\what{X}_{r,d_r}}\to  \op{span}\set{\what{X'}_{0,1},\dots,\what{X'}_{r,d_r}}
	\]
	by setting $\what{f}_*(Z)=\mc{M}_{f(p)}^{-1}\of df_* \of \mc{M}_p(Z)$ where $df_*:\mf{n}_p\to\mf{n'}_{f(p)}$ is the Lie derivative of the isomorphism $f_*:TC_pN\to TC_{f(p)}N'$.
\end{definition}

\begin{lemma}\label{lem:homogDf}
Let $f:N\to N'$ be a local $C^\infty$ diffeomorphism of a sub-Riemannian manifolds $N$ and $N'$ sending the horizontal distribution $E$ of $N$ to the horizontal distribution $E'$ of $N'$. Assume that $E$ is generic at a point $p\in N$, and let $f_*:TC_pN\to TC_{f(p)}N'$ be the induced isomorphism.
For any horizontal vector field $X$ we have, 
\[
\what{f}_*\what{X}=\what{Df(X)}.
\]
Moreover, if $({X}_{0,1},\dots,{X}_{r,d_r})$ is a graded basis around $p$ then $\what{f}_*[\what{X}_{i,k},\what{X}_{j,l}]=[\what{f}_*\what{X}_{i,k},\what{f}_*\what{X}_{j,l}]$.
\end{lemma}

\begin{proof}
By Proposition \ref{prop:df_isom} the limit of the maps $f_t$ given in Definition \ref{def:Carnot} exists and is $f_*:TC_pN\to TC_{f(p)}N'$.

The push-forward under $f$ of a horizontal vector field $X$ is also horizontal and therefore converges under metric rescaling to a right invariant horizontal field $\bar{Df(X)}$ by Proposition 8.10 of \cite{MargulisMostow95}. As shown in the proof of Theorem 10.5 of \cite{MargulisMostow95}, the image under $f_*$ of the one parameter subgroup $e^{s\bar{X}}$ is the one parameter subgroup of this limit push-forward field. In other words, for all $g\in TC_pN$ and $s\in \R$ we have,
\[
f_*(e^{s \bar{X}}g)=e^{s \bar{Df(X)}}f_*(g).
\]
Taking Lie derivatives we have
$df_*(\bar{X})=\bar{Df(X)}$ and therefore under the M\'etivier correspondence we have $\what{f}_*\what{X}=\what{Df(X)}$ as claimed.

The second statement follows directly from the M\'etivier correspondence and \eqref{eq:brackhat} and the fact that ${df}_*[\bar{X}_{i,k},\bar{X}_{j,l}]=[{df}_*\bar{X}_{i,k},{df}_*\bar{X}_{j,l}]$.
\end{proof}

By Lemma \ref{lem:isometricinclusion} the inclusion $\imath:E(p)\hookrightarrow \mf{n}_p$ defined earlier is an isometry onto its image $\mf{n}^0_p$ with respect to the Riemannian norm and norm inducing the \CC metric.
The following corollary gives further relationships between some of the objects defined thus far.

\begin{corollary}\label{cor:der-comparison}
	In the setting of Proposition \ref{prop:df_isom},   the following diagram commutes,
	\[
	\begin{tikzcd}
	E(p) \arrow[hookrightarrow]{d}{\imath} \arrow[r, "D_pf"]  & E'(f(p)) \arrow[hookrightarrow]{d}{\imath'} \\
	\mf{n}_p \arrow[r,"d(f_*)"] \arrow[d,"\mbf{exp}"] & \mf{n'}_{f(p)} \arrow[d,"\mbf{exp}"]\\
	TC_pN \arrow[r,"f_*"] & TC_{f(p)}N' \\
	\end{tikzcd}
	\]
	where $\mbf{exp}$ is the Lie exponential map $\mbf{exp}(X)=e^X$.
	Moreover, for all $\xi\in E(p)$, we have $\norm{D_pf(\xi)}_{f(p)}=\norm{df_*(\imath\xi)}_{\mf{n'}^0_{f(p)}}$.
\end{corollary}

\begin{proof}
	The lower part of the diagram commutes as a standard consequence of Lie theory.
	
	By Lemma \ref{lem:homogDf}, we have $\what{f}_*\what{X}=\what{Df(X)}$. Since $X(p)=\what{X}_p(p)\in E(x)$ for a horizontal field $X$ we have, 
	\[
	df_*(\imath(X(p))=df_*(\bar{X})=df_*\mc{M}_p(\what{X})=\mc{M}_{f(p)}\what{f}_*\what{X}=\mc{M}_{f(p)}\what{Df(X)}=\bar{Df(X)}=\imath'(D_{p}f(X(p))),
	\]
	where the third equality is by the definition of $\what{f}_*$. This establishes commutativity of the upper diagram.

	The final statement follows from the commutativity of the diagram and the fact that the inclusions are isometric. 

\end{proof}

The following relationship between the differential of a local $C^\infty$ diffeomorphism at a fixed point and its induced automorphism will prove to be crucial. A priori there is no clear connection between the Carnot derivative and the ordinary derivative of a smooth sub-Riemannian map. However, we are able to establish such a connection in some important special cases. 
\begin{proposition}\label{prop:eval_products}
Let $f:N\to N$ be a local $C^\infty$ diffeomorphism of a sub-Riemannian manifold $N$ preserving the horizontal distribution $E$. Assume that $E$ is generic at a point $p\in N$ where $f(p)=p$, and let $f_*:TC_pN\to TC_pN$ be the induced automorphism. If $df_*\rest{\mf{n}^0}$ is diagonalizable with real eigenvalues $(\la_1,\dots,\la_{d_0})$, then there exists a basis of $T_pN$ such that $D_pf:T_pN\to T_pN$ is upper triangular with diagonal values $(\lambda_1,\dots,\lambda_{d_0},\lambda_{i_1}\la_{i_2},\dots,\lambda_{i_1}\la_{i_2}\la_{i_3},\dots)$. 
\end{proposition}
\begin{proof}
Again let $U\subset N$ be a small neighborhood around the point $p$ and  $X_1,\dots, X_n$ be a graded basis where $X_1,\dots, X_{d_0}$ are horizontal vector fields spanning $E$ on $U$. As above, normal coordinates define a dilatation $\delta_t:U\to N$. Then if $X$ and $Y$ are horizontal, by M\'etivier \cite{Metivier}, $X$ and $Y$ are vector fields of degree at most $1$ at $p$ and $\what{X}=\lim\limits_{t\to\infty}\frac{1}{t}(\delta_t)_* X$ is the homogeneous degree $1$ part of $X$.  Note that $[X, Y]$ is a vector field of degree at most $2$ and we denote the homogeneous degree 2 part by $[X,Y]^{(2)}=\lim\limits_{t\to\infty}\frac{1}{t^2}(\delta_t)_*[ X, Y]$.

Without loss of generality we may assume that the chosen basis of fields $X_1,\dots,X_{d_0}$ have corresponding right invariant fields $\bar{X}_i$ which diagonalize $f_*$ on the first level of the Lie algebra $\mf{n}$. 
Indeed by assumption, there are $\lambda_i\in \R$ such that $df_*\bar{X_i}=\lambda_i \bar{X_i}$ for $i=1,\dots,d_0$.

By Corollary \ref{cor:der-comparison}, for $i=1,\dots,d_0$ we have 
$\imath D_pf(X_i(p))=df_*(\imath(X_i(p)))=df_*(\bar{X}_i)=\la_i \bar{X}_i$. Taking $\imath^{-1}$, we have $D_pf (X_i(p)) = \lambda_i X_i(p)$.

By Lemma \ref{lem:homogDf}, if $[X_i,X_j]$ is a graded basis element, then $\what{f}_*[\what{X}_i,\what{X}_j]=[\what{f}_*\what{X}_i,\what{f}_*\what{X}_j]$.
Therefore we have,
\begin{align*}
df_*\mc{M}_p[\what{X}_i,\what{X}_j]&=\mc{M}_p[\what{f}_*\what{X}_i,\what{f}_*\what{X}_j], \\
df_*[\mc{M}_p\what{X}_i,\mc{M}_p\what{X}_j]&=\mc{M}_p[\what{Df {X}_i},\what{Df {X}_j}], \\
[df_*\bar{X}_i,df_*\bar{X}_j]&=\mc{M}_p[{Df {X}_i},{Df {X}_j}]^{(2)},\\
\la_i\la_j[\bar{X}_i,\bar{X}_j]&=\mc{M}_p[{Df {X}_i},{Df {X}_j}]^{(2)}.
\end{align*}
It then follows that, 
\begin{align}\label{eq:deg2part}
[{Df {X}_i},{Df {X}_j}]^{(2)} &=\la_i\la_j\mc{M}_p^{-1}[\bar{X}_i,\bar{X}_j]=\la_i\la_j[\what{X}_i,\what{X}_j]=\la_i\la_j[X_i,X_j]^{(2)}.
\end{align}

\sloppy
By hypothesis $\set{X_1,\dots,X_{d_0}}\cup\set{[X_i,X_j]:1\le i,j \le d_0}$ spans $E^1$. Since the homogeneity degree of $Df[X_i,X_j]=[Df(X_i),Df(X_j)]$ and $[X_i,X_j]$ at $p$ is at most $2$,  ${Df[X_i,X_j]-\la_i\la_j[X_i,X_j]}$ has homogeneity degree at most $1$. Evaluating \eqref{eq:deg2part} at $p$ we have $[Df(X_i),Df(X_j)](p)-\la_i\la_j[X_i,X_j](p)\in E(p)=E^0(p)$.

Continuing by induction, if $X_{i}$ and $X_{j}$ are graded basis fields of degree at most $q$ and $s$ respectively such that $[X_{i},X_{j}]$ is a basis field, then a nearly identical computation shows,
\[
[{Df {X}_i},{Df {X}_j}]^{(q+s)}=\sigma_i\sigma_j[X_i,X_j]^{(q+s)},
\]
where by inductive hypothesis $\sigma_i=\lambda_{i_1}\dots\lambda_{i_\kappa}$ and $\sigma_j=\la_{j_1}\cdots\la_{j_\tau}$ for some indices $i_1,\dots,i_\kappa$ and $j_1,\dots,j_\tau$.

This shows that in the $\set{X_i(p)}$ basis, $Df$ is upper triangular. If restricted to the horizontal distribution, the diagonal values are $\la_1,\dots,\la_{d_0}$, then other diagonal elements are products of the $\lambda_i$'s as claimed.
\end{proof}

Now we are equipped to prove Theorem \ref{thm:Lyap_exp} in the introduction.
\begin{reptheorem}{thm:Lyap_exp}
Let $f:N\to N$ be a local $C^\infty$ diffeomorphism of a sub-Riemannian manifold $N$. Assume that the horizontal distribution is generic at a point $p\in N$ where $f(p)=p$, and that the graded automorphism $f_*:TC_pN\to TC_pN$ induced from $f$ induces is a homothety. Then for some $\la>1$, $D_pf\rest{E^{i}_p}$ has Lyapunov exponents $\log \la,2\log \la,\dots,(i+1)\log\la$ with corresponding multiplicities $d_0=\dim \mf{n}^0,\dots,d_i=\dim \mf{n}^{i}$ for $i=0,\dots,r$.
\end{reptheorem}

\begin{proof}

Let $({X}_{0,1},\dots,{X}_{r,d_r})$ be a graded basis of $N$ in a neighborhood of $p$. As before, we let $\ds{\mc{M}_p: \op{span}\set{\what{X}_{0,1},\dots,\what{X}_{r,d_r}}\to \mf{n}_p}$ be the M\'etivier correspondence. We will show that with respect to the basis $({X}_{0,1}(p),\dots,{X}_{r,d_r}(p))$, the matrix for $D_pf$ is a block upper triangular matrix. Suppose that $f_*=\delta_{\la} h$, where $\delta_\la$ is the dilation of scale $\la>0$ and $h$ is an isometry of the \CC metric of $TC_pN$ which necessarily lies in a compact subgroup of the graded automorphism group.

By Lemma \ref{lem:homogDf}, $\what{f}_*\what{X}_{0,j}=\what{Df(X_{0,j})}$. Thus,
$$\what{Df(X_{0,j})}=\mc{M}_p df_* \mc{M}_p^{-1}\what{X}_{0,j}=\mc{M}_p df_* \bar{X}_{0,j}=\la\mc{M}_p dh \bar{X}_{0,j}=\la\mc{M}_p dh \mc{M}_p^{-1}\what{X}_{0,j}.$$
Letting $\what{dh}=\mc{M}_p dh \mc{M}_p^{-1}$, we can rewrite the above equation as $\what{Df(X_{0,j})}=\la\what{dh}\what{X}_{0,j}$. Since $dh$ is a graded Lie algebra homomorphism, $\what{dh}$ is also a graded Lie algebra homomorphism of $ \op{span}\set{\what{X}_{0,1},\dots,\what{X}_{r,d_r}}$. In particular, $\what{dh}$ preserves $ \op{span}\set{\what{X}_{i,1},\dots,\what{X}_{i,d_i}}$ for every $i=0,\dots, r$. It follows that $\what{dh}\what{X}_{0,j}\in  \op{span}\set{\what{X}_{0,1},\dots,\what{X}_{0,d_0}}$ for every $j=1,\dots, d_0$. Therefore, $D_pf(\what{X}_{0,j}(p))=\what{Df(X_{0,j})}(p)=(\what{dh}\what{X}_{0,j})(p)\in \op{span}\set{\what{X}_{0,1}(p),\dots,\what{X}_{0,d_0}(p)}=E(p)$.

Consider $X_{1,l}$. Suppose that $X_{1,l}=[X_{0,i},X_{0,j}]$. By Lemma \ref{lem:homogDf}, $\what{f}_*\what{X_{1,l}}=\what{f}_*[\what{X}_{0,i},\what{X}_{0,j}]=[\what{f}_*\what{X}_{0,i},\what{f}_*\what{X}_{0,j}]=[\what{DfX}_{0,i},\what{DfX}_{0,j}]$. And thus,
\begin{align}\label{eq:part2}
(Df[{X}_{0,i},{X}_{0,j}])^{(2)}=[{DfX}_{0,i},{DfX}_{0,j}]^{(2)}=[\what{DfX}_{0,i},\what{DfX}_{0,j}]=\la^2\what{dh}\what{X_{1,l}}.
\end{align}

Set $V^{(i)}(p)=\op{span}\set{{X}_{i,1}(p),\dots,{X}_{i,d_i}(p)}$ so that $E^{i}(p)=V^{(0)}(p)\oplus\cdots\oplus V^{(i)}(p)$. We note that $\what{dh}\what{X}_{0,l}\in \op{span}\set{\what{X}_{0,1},\dots,\what{X}_{0,d_0}}$ since $\what{dh}$ is a graded Lie algebra homomorphism. We observe that for every $l=1\dots,d_1$, we have $\what{X}_{1,l}(p) \in \op{span}\set{X_{1,l}(p), X_{0,1}(p),\dots, X_{0,d_0}(p)}$. It follows that $(Df[{X}_{0,i},{X}_{0,j}])^{(2)}(p)\in \op{span}set{\what{X}_{0,1}(p),\dots,\what{X}_{0,d_0}(p),\what{X}_{1,1}(p),\dots,\what{X}_{1,d_1}(p)}=E^{1}(p)$. On the other hand, 
$$(Df[{X}_{0,i},{X}_{0,j}])(p)=(Df[{X}_{0,i},{X}_{0,j}])^{(2)}(p)+(Df[{X}_{0,i},{X}_{0,j}])^{(\le 1)}(p),$$
where $(Df[{X}_{0,i},{X}_{0,j}])^{(\le 1)}$ is the homogeneous degree at most 1 part of $(Df[{X}_{0,i},{X}_{0,j}])$. Evaluating a vector field of degree at most 1 at $p$ we get a vector in $E(p)= \op{span}\set{\what{X}_{0,1}(p),\dots,\what{X}_{0,d_0}(p)}$. Therefore, $(Df[{X}_{0,i},{X}_{0,j}])(p)\in \op{span}\set{\what{X}_{0,1}(p),\dots,\what{X}_{0,d_0}(p),\what{X}_{1,1}(p),\dots,\what{X}_{1,d_1}(p)}=E^{1}(p)$.

Inductively, for $X_{m+1,l}=[X_{m,i},X_{0,j}]$, we obtain
\begin{align}\label{eq:part3}
(Df[{X}_{m,i},{X}_{0,j}])^{(m+2)}=[{DfX}_{m,i},{DfX}_{0,j}]^{(m+2)}=[({DfX}_{m,i})^{(m+1)},\what{DfX}_{0,j}]=\la^{m+2}\what{dh}\what{X}_{m+1,l}.
\end{align}
And it follows that
\begin{align}\label{eq:eval_m}
\begin{split}
(D_pf)(({X}_{m+1,l})(p))&=\la^{m+2}(\what{dh}\what{X}_{m+1,l})(p)+(Df{X}_{m+1,l})^{(\le m+1)}(p) \in E^{m+1}(p),
\end{split}
\end{align}
where the notation $Y^{(\le m+1)}$ indicates the sum of the homogeneous parts of any vector field $Y$ of degree at most $m+1$.

Therefore, with respect to the basis $({X}_{0,1}(p),\dots,{X}_{r,d_r}(p))$, $D_pf$ is a block upper triangular matrix. We let $A_i$ be the blocks on the diagonal of $D_pf$ corresponding to the basis elements $\set{X_{i,1}(p),\dots,X_{i,d_i}(p)}$, for every $i=0,\dots, r$. We note that the size of $A_i$ is $\dim(\mf{n}^{i}_p)\times \dim(\mf{n}^{i}_p)$. Since $(Df{X}_{m+1,l})^{(\le m+1)}(p)\in  E^{m}(p)$, the equation \ref{eq:eval_m} implies that
$$\proj_{V^{(m+1)}(p)}(D_pf)(({X}_{m+1,l})(p))=\proj_{V^{(m+1)}(p)}(\la^{m+2}(\what{dh}\what{X}_{m+1,l})(p)),$$
where $\proj_{V^{(m+1)}(p)}$ is the projection of $T_pN$ onto $V^{(m+1)}(p)$ along the subspace $\check{V}^{(m+1)}(p):=V^{(0)}(p)\oplus\cdots\oplus V^{(m)}(p)\oplus V^{(m+2)}(p)\oplus\cdots\oplus V^{(r)}(p)$. Thus the $(j,l)$ entry of the block matrix $A_{m+1}$ is the $X_{m+1,j}(p)$ component of $\la^{m+2}(\what{dh}\what{X}_{m+1,l})(p)$, for $j,l\in\set{1,\dots,d_{m+1}}$.

Next, we estimate Lyapunov exponents of block matrices $A_0, \dots, A_r$. Applying equation \ref{eq:part3} for $f^n$ instead of $f$, we obtain
\begin{align}\label{eq:part4}
\begin{split}
(Df^n[{X}_{m,i},{X}_{0,j}])^{(m+2)}&=[{Df^nX}_{m,i},{Df^nX}_{0,j}]^{(m+2)}\\
&=[({Df^nX}_{m,i})^{(m+1)},\what{Df^nX}_{0,j}]=\la^{n(m+2)}\what{dh}^n\what{X}_{m+1,l}.
\end{split}
\end{align}
Since \[
(Df^n[{X}_{m,i},{X}_{0,j}])(p)=(Df^n[{X}_{m,i},{X}_{0,j}])^{(m+2)}(p)+(Df^n[{X}_{m,i},{X}_{0,j}])^{(\le m+1)}(p),
\] 
and 
\[
(Df^n[{X}_{m,i},{X}_{0,j}])^{(\le m+1)}(p)\in E^{m}(p),
\]
then the projection of $(Df^n[{X}_{m,i},{X}_{0,j}])(p)$ along $\check{V}^{(m+1)}(p)$ onto $V^{(m+1)}(p)$ is the same as the projection of $\la^{n(m+2)}\what{dh}^n\what{X}_{m+1,l}(p)$ to $V^{(m+1)}(p)$ along $\check{V}^{(m+1)}(p)$.

We claim that,
\begin{align}\label{eq:intersect}
\op{span}\set{\what{X}_{m+1,1}(p),\dots,\what{X}_{m+1,k_{m+1}}(p)}\cap \check{V}^{(m+1)}(p)=\{0\}.
\end{align} 
Indeed, we have $\op{span}\set{\what{X}_{0,1}(p),\dots,\what{X}_{i,d_{i}}(p)}= \op{span}\set{{X}_{0,1}(p),\dots,{X}_{i,d_{i}}(p)}$ for all $i=0,\dots,m+1$. Moreover, these sets of vectors form bases of their corresponding spanned vector spaces. This implies the claim.

This claim implies that for every $w\neq 0$ in $\op{span}\set{\what{X}_{m+1,1}(p),\dots,\what{X}_{m+1,d_{m+1}}(p)}$, its projection to $V^{(m+1)}(p)$ along $\check{V}^{(m+1)}(p)$ is nonzero. For every $v\neq 0$ in $V^{(m+1)}(p)$, there are $c_1,\dots , c_{d_{m+1}}\in \R$ not all being zero such that $v=c_1 {X}_{m+1,1}(p) +\dots + c_{d_{m+1}}{X}_{m+1,d_{m+1}}(p)$. We then have
\begin{align*}
D_pf^nv&=c_1 (Df^n{X}_{m+1,1})(p) +\dots + c_{k_{m+1}}(D_pf^n{X}_{m+1,k_{m+1}}(p))\\
&=c_1 (Df^n{X}_{m+1,1})^{(m+2)}(p) +\dots + c_{k_{m+1}}(D_pf^n{X}_{m+1,k_{m+1}}^{(m+2)}(p))\\
&\phantom{aaaaaaaaaaa}+ c_1 (Df^n{X}_{m+1,1}^{(\le m+1)})(p) +\dots + c_{k_{m+1}}(D_pf^n{X}_{m+1,k_{m+1}}^{(\le m+1)})(p)\\
&=\lambda^{n(m+2)}\what{dh}^n (c_1 \what{X}_{m+1,1}(p) +\dots + c_{k_{m+1}}\what{X}_{m+1,k_{m+1}}(p))\\
&\phantom{aaaaaaaaaaa}+ c_1 (Df^n{X}_{m+1,1}^{(\le m+1)})(p) +\dots + c_{k_{m+1}}(D_pf^n{X}_{m+1,k_{m+1}}^{(\le m+1)})(p).
\end{align*}
By same argument as above, $D_pf^nv$ and $\lambda^{n(m+2)}\what{dh}^n (c_1 \what{X}_{m+1,1}(p) +\dots + c_{d_{m+1}}\what{X}_{m+1,d_{m+1}}(p))$ have the same projection onto $\op{span}\set{{X}_{m+1,1}(p),\dots,{X}_{m+1,d_{m+1}}(p)}$ along $\op{span}\set{{X}_{1,1}(p),\dots,{X}_{m,d_{m}}(p)}$. Since $c_1, \dots c_{d_{m+1}}$ are not all zero, the vector $c_1 \what{X}_{m+1,1}(p) +\dots + c_{d_{m+1}}\what{X}_{m+1,d_{m+1}}(p)$ is nonzero, and thus has a nonzero projection on $\op{span}\set{{X}_{m+1,1}(p),\dots,{X}_{m+1,d_{m+1}}(p)}$ by \eqref{eq:intersect}.  

Since $h$ is in a compact subgroup of $\GL(n,\R)$, there is a subsequence $(n_k)$ such that $\what{dh}^{n_k}$ converges to the identity. Choose an arbitrary norm $\|\cdot\|$ on $V^{(m+1)}(p)$. Therefore, we may find two constants $C_1,C_2>0$ independent of $k$ such that
$$C_1\norm{v}<\|\proj_{V^{(m+1)}(p)}(\what{dh}^{n_k}(v))\|<C_2\norm{v}.$$
Hence,
$$C_1\lambda^{n_k(m+2)}\|v\|<\|\proj_{V^{(m+1)}(p)}(D_pf^{n_k}v)\|<C_2\lambda^{n_k(m+2)}\|v\|,$$
for every $v\neq 0$ in $V^{(m+1)}(p)$. It follows that $A_{m+1}$ has all Lyapunov exponents equal to $(m+2)\log \lambda$.

Finally, as is well known, the Lyapunov exponents of a block upper triangular matrix are the same as those of the corresponding block diagonal matrix.
\end{proof}

If $f_*$ fails to be a homothety, we are still able to recover a system of inequalities among the Lyapunov exponents. This is Theorem \ref{thm:spectrum_ineq} from the introduction.

\begin{reptheorem}{thm:spectrum_ineq}
Let $f:N\to N$ be a local $C^\infty$ diffeomorphism of a sub-Riemannian manifold $N$. Assume that the horizontal distribution is generic at a point $p\in N$ where $f(p)=p$. Suppose the tangent cone $TC_pN$ is $(r+1)$-step and set $n_{-1}=0$ and $n_i=\dim E^{i}(p)$ for $i\in \set{0,\dots,r}$. If $D_pf\rest{E^{i}_p}$ has Lyapunov exponents $(\log \la_1,\log \la_2,\dots,\log \la_{n_0},\log \la_{{n_0}+1},\dots,\log \la_{n_1},\dots,\log \la_{n_{i}})$ listed with multiplicity and with $\log \la_{n_{j-1}+1},\dots,\log \la_{n_{j}}$ in nondecreasing order for each $j\in\set{0,\dots,i}$, then for each $i\in \set{0,\dots, r}$ and $j\in\set{n_{i-1}+1,\dots,n_{i}}$, we have $(i+1) \log \la_1\leq \log \la_{j}\leq (i+1)\log \la_{n_0} $.
\end{reptheorem}

\begin{proof}

\sloppy
By switching $f$ with $f^{-1}$ we observe that it is sufficient to prove just the lower bound.
Let $\set{{X}_{0,1},\dots,{X}_{r,d_r}}$ be a graded basis of $N$ in a neighborhood of $p$. We let $\ds{\mc{M}_p: \op{span}\set{\what{X}_{0,1},\dots,\what{X}_{r,d_r}}\to \mf{n}_p}$ be the M\'etivier correspondence. Note that $n_i=\sum_{j=0}^i d_i$ for $i=0,\dots,r$.

Our assumptions imply that the induced map $f_*$ exists as a graded automorphism. Its derivative $df_*$ and the corresponding  $\what{f}_*=\mc{M}_{f(p)}^{-1}\of df_* \of \mc{M}_p$ are graded Lie algebra automorphisms. With respect to the basis $(\what{X}_{0,1},\dots,\what{X}_{r,d_r})$, $\what{f}_*$ is a block diagonal matrix. We further assume that we have chosen the initial horizontal basis $X_{0,1},\dots,X_{0,d_0}$ so that $\what{f}_*$ is in real-Jordan canonical form with respect to $\what{X}_{0,1},\dots,\what{X}_{0,d_0}$.

We note that for every $0\le j\le r$, $\op{span}\set{\what{X}_{0,1}(p)  \dots,\what{X}_{j,d_j}(p)}=\op{span}\set{{X}_{0,1}(p),\dots,{X}_{j,d_j}(p)}\subset T_pN$. Since $\what{f}_*$ is a block diagonal matrix, $(\what{f}_*\what{X}_{j,l})(p)\in \op{span}\set{\what{X}_{j,1}(p),  \dots,\what{X}_{j,d_j}(p)}$ for every $1\le l\le d_j$. On the other hand, $(Df{X}_{j,l} - \what{f}_*\what{X}_{j,l})(p)\in \op{span}\set{\what{X}_{0,1}(p),  \dots,\what{X}_{j-1,d_{j-1}}(p)}$ by Lemma \ref{lem:homogDf}. Here we used the convention that $\op{span}\set{\what{X}_{0,1}(p),  \dots,\what{X}_{j-1,d_{j-1}}(p)}=\{0\}$ when $j=0$. By replacing $f$ by $f^n$, we have that $(Df^n{X}_{j,l} - \what{f}_*^n\what{X}_{j,l})(p)\in \op{span}\set{\what{X}_{0,1}(p),  \dots,\what{X}_{j-1,d_{j-1}}(p)}$ for every $n\in \N$. In particular, in the $\what{X}_{j,l}$ basis, $D_pf$ is block upper triangular and with diagonal blocks coinciding with the matrix of $\what{f}_*$ in this basis. Because $\what{f}_*^n\what{X}_{j,l}(p)\in \op{span}\set{\what{X}_{j,1}(p),\dots,\what{X}_{j,d_j}(p)}$ and $\op{span}\set{\what{X}_{j,1}(p),\dots,\what{X}_{j,d_j}(p)}\cap \op{span}\set{\what{X}_{0,1}(p),\dots,\what{X}_{j-1,d_{j-1}}(p)}=\{0\}$, the growth rate of $Df^nX_{j,l}(p)=D_pf(X_{j,l}(p))$ is at least as large as the growth rate of $\what{f}_*^n\what{X}_{j,l}(p)$ for all $1\leq l \leq d_j$. Hence for any $v=\sum_{i=1}^{d_j} a_i {X}_{j,1}(p)$ the growth rate of $Df^n v$ is at least as large as the growth rate of  $\what{f}_*^n(w)$ where $w=\sum_{i=1}^{d_j} a_i \what{X}_{j,1}(p)$. By the block upper triangular structure of $D_pf$, it follows that the growth rate of $D_pf^n(v)$ is at least as large as the minimum of the growth rates of $\what{f}_*^n$ restricted to $\op{span}\set{\what{X}_{j,1}(p), \dots, \what{X}_{j,d_j}(p)}$.

The conclusion now follows from the fact that the growth rate of $\what{f}_*^n\what{X}_{j,i}(p)$ is at least $(j+1)$-times the minimum of growth rates of $\what{f}_*^n\what{X}_{0,1}(p), \dots, \what{f}_*^n\what{X}_{0,d_0}(p)$ due to the bracket relations. (Recall from the choice of initial basis that $\what{f}_*$ achieves its minimal growth rate on one of the vectors $\what{X}_{0,1},\dots,\what{X}_{0,d_0}$.) Since ${X}_{0,1},\dots,{X}_{0,d_0}$ are horizontal, $\what{f}_*^n\what{X}_{0,i}(p)=D_pf^n({X}_{0,i}(p))$. Thus, the minimum of growth rates of $\what{f}_*^n\what{X}_{0,1}(p), \dots, \what{f}_*^n\what{X}_{0,d_0}(p)$ is at least $\log \lambda_1$.
\end{proof}

\begin{remark}
Note also that in the ordering of the Lyapunov exponents given in Theorem \ref{thm:spectrum_ineq}, we may have $\la_{d_i+1}<\la_{d_i}$. We note that since $D_pf^n(X_{j,l}(p))= \what{f}_*^n\what{X}_{j,l}(p)+ Y_n$ for some $Y_n\in \op{span}\set{\what{X}_{0,1}(p),\dots,\what{X}_{j-1,d_{j-1}}(p)}$, the iterates of the $Y_n$ component may grow faster than $\what{f}_*^n\what{X}_{j,l}(p)$ except when $l=d_j$. Hence we may not have $(i+1)\log\lambda_2\leq \log\lambda_{j}$ for any $j\in \set{n_i+1,\dots,n_{i+1}}$.
\end{remark}

\begin{remark}
We note that Proposition \ref{prop:eval_products} and Theorem \ref{thm:Lyap_exp} may be combined and slightly generalized as follows. Suppose $df_*:\mf{n}_p\to \mf{n}_p$ has the form $df_*=D\cdot h$ where $D$ is diagonalizable with positive diagonal values $(\la_1,\dots,\la_k)$ when restricted to $\mf{n}^{0}_p$ and $h$ is an isometry that preserves the eigenspaces of $D$. Then we obtain Lyapunov exponents of $d_pF$ of the form $\sum_{j=0}^{k} \log{\la_{i_j}}$ for $k\leq r$ and certain combinations of indices $i_j$.
\end{remark}

Next we present the proof of Theorem \ref{thm:Hei-spec}
\begin{reptheorem}{thm:Hei-spec}
Let $f:N\to N$ be a local $C^\infty$ diffeomorphism of a sub-Riemannian manifold $N$. Assume that the horizontal distribution $E$ is generic at a point $p\in N$ where $f(p)=p$. Suppose the tangent cone $TC_pN$ is isomorphic to the Heisenberg group $H^{2n+1}$. If $\log \lambda_1, \log \lambda_2, \dots, \log \lambda_{2n}$ are the Lyapunov exponents of $D_pf\rest{E}$ listed with multiplicity and $\log \la_{2n+1}$ is the remaining Lyapunov exponent of $D_pf$, then $\log \lambda_1+\log \lambda_2+\dots+\log \lambda_{2n}=n \log \lambda_{2n+1}$.
\end{reptheorem}

\begin{proof}
The Heisenberg group $H^{2n+1}$ has a distinguished basis of right invariant vector fields,  $\bar{X}_1,\dots , \bar{X}_n, \bar{Y}_1, \dots, \bar{Y}_n, \bar{Z}$ such that $[\bar{X}_i,\bar{Y}_j]=\delta_{ij}\bar{Z}$ and $[\bar{X}_i,\bar{X}_j]=[\bar{Y}_i,\bar{Y}_j]=[\bar{X}_i,\bar{Z}]=[\bar{Z},\bar{Y}_j]=0$. By assumption, $TC_pN$ is isomorphic to $H^{2n+1}$. Hence in a neighborhood of $p$ there is a graded basis of fields $X_1,\dots , X_n, Y_1, \dots, Y_n,Z=[X_1,Y_1]$ such that $\mc{M}_p(\what{X}_i)=\bar{X}_i$, $\mc{M}_p(\what{Y}_i)=\bar{Y}_i$ and $\mc{M}_p([\what{X}_i,\what{Y}_i])=\bar{Z}$ for $i\in\set{1,\dots,n}$. In other words, setting $\what{Z}=[\what{X}_1,\what{Y}_1]$ we have
$[\what{X}_i,\what{Y}_j]=\delta_{ij}\what{Z}$ and $[\what{X}_i,\what{X}_j]=[\what{Y}_i,\what{Y}_j]=[\what{X}_i,\what{Z}]=[\what{Z},\what{Y}_j]=0$ for all $i,j\in\set{1,\dots,n}$. Note that by the identity \eqref{eq:brackhat}, we have $Z^{(2)}=\what{Z}$.

With respect to the chosen basis we have an identification $\mf{n}_p=\R^{2n}\oplus \R$. The graded automorphism group of $H^{2n+1}$ is $\Z_2\ltimes (\op{Sp}(2n,\R) \times \op{Dil})$ where $\op{Sp}(2n,\R)$ is the symplectic group, $\op{Dil}\cong \R_+$ are the dilations and the $\Z_2$ generator switches $\bar{X}_i$ and $\bar{Y}_i$ for $i=1,\dots n$ (\cite{Tilgner70}). Since $df_*:\mf{n}_p\to \mf{n}_p$ is a graded automorphism, there is an $A\in \op{Sp}(2n,\R)$, a block matrix $U=\begin{bmatrix} 0 & I_n\\ I_n &0\end{bmatrix}$ or $U=I_{2n}$, and $\sigma>0$ such that $df_*$ has the form $df_*(x,z)=(\sigma U A x, \sigma^2z)$. (Here the $U$ matrix represents the corresponding element of $\Z_2$.) Since Lyapunov exponents of $f^2$ are twice those of $f$ and $(UA)^2\in \op{Sp}(2n,\R)$, without loss of generality we may replace $f$ with $f^2$ and assume $df_*(x,z)=(\sigma A x, \sigma^2z)$. Next we show that with respect to the basis $\set{X_1(p),\dots, X_n(p),Y_1(p),\dots,Y_n(p),Z(p)}$ of $T_pN$, $D_pf=\begin{bmatrix}\sigma A&B\\0&\sigma^2\end{bmatrix}$ for some $2n\times 1$ matrix $B$.

It follows that with respect to the basis $\{\what{X}_1,\dots , \what{X}_n, \what{Y}_1, \dots, \what{Y}_n\}$, $\what{f}_*(\what{X}_i)=\sigma A\what{X}_i$ and $\what{f}_*(\what{Y}_j)=\sigma A\what{Y}_j$. We note that $\what{X}_i(p)=X_i(p)$ and $\what{Y}_j(p)=Y_j(p)$. By Lemma \ref{lem:homogDf}, $\what{f}_*(\what{X}_i)=\what{DfX_i}$ and $\what{f}_*(\what{Y}_j)=\what{DfY_i}$. Thus with respect to the basis $\{X_1(p),\dots , X_n(p), Y_1(p), \dots, Y_n(p)\}$, by restricting of $D_pf$ on $\op{span}\set{X_1(p),\dots , X_n(p), Y_1(p), \dots, Y_n(p)}$, we have $D_pf(X_i(p))=AX_i(p)$ and $D_pf(Y_j(p))=AY_j(p)$.

\sloppy
Moreover, we have 
\[
\sigma^2Z^{(2)}=\what{f}_*(Z^{(2)})=\what{f}_*([\what{X}_1,\what{Y}_1])=[\what{f}_*\what{X}_i,\what{f}_*\what{Y}_i]=[\what{DfX}_1,\what{DfY}_1]=(Df[X_1,Y_1])^{(2)}=(DfZ)^{(2)}.
\]
Thus, $D_pf(Z(p))= DfZ(p)=(DfZ)^{(2)}(p)+(DfZ)^{(\le 1)}(p)=\sigma^2Z^{(2)}(p)+(DfZ)^{(\le 1)}(p)$. We note that $(DfZ)^{(\le 1)}(p)\in \op{span}\set{X_1(p),\dots , X_n(p), Y_1(p), \dots, Y_n(p)}$. Also, $Z=Z^{(2)}+Z^{(\leq 1)}$ with $Z^{(\leq 1)}(p)\in \op{span}\set{X_1(p),\dots , X_n(p), Y_1(p), \dots, Y_n(p)}$. Hence, with respect to the basis $\{{X}_1(p),\dots , {X}_n(p), {Y}_1(p), \dots, {Y}_n(p), Z(p)\}$ of $T_pN$, there is a $2n\times 1$ matrix $B$ such that $D_pf=\begin{bmatrix}\sigma A&B\\0&\sigma^2\end{bmatrix}$.

Now it follows that the Lyapunov exponents of $f$ are $2 \log \sigma$ and the logarithms of the moduli of the eigenvalues of $\sigma A$. Since $\det(\sigma A)=\sigma^{2n}$, we conclude that $n\log \la_{2n+1}=2n\log \sigma$ is equal to the sum of all other exponents.
\end{proof}

\section{Sub-Riemannian dynamics}\label{sec:subRiem}

We will apply the ideas from the last section to distributions tangent to foliations invariant under flows and diffeomorphisms. We start with a very general and abstract statement which we then develop in the context of hyperbolic dynamics. We will make the following hypotheses throughout this section.

\begin{StandingAssumptions}\label{assump}
We assume: 
\begin{itemize}
	\item $X$ is a smooth compact manifold with $f: X\to X$ a $C^1$ diffeomorphism.
	\item  $\mc{F}$ is an $f$-invariant foliation tangent to a continuous distribution $D$. 
	\item The leaves of $\mathcal F$ are $C^\infty$ and $f$ is $C^\infty$ along leaves. 
	\item There is a $df$-invariant continuous subdistribution $E\subseteq D$. 
\end{itemize}
\end{StandingAssumptions}
We define
$$\Omega=\{v\in X: E \text{ is } C^\infty, \text{ horizontal and generic in a neighborhood of } v \text{ in } \mc{F}(v)\}.$$
Moreover, at each point $v\in \mc{F}(v)$ we have a filtration 
\[
E(v)=E^0(v)\subset E^1(v)\subset E^2(v) \subset \cdots
\] 
of vector subspaces of $T_v\mc{F}(v)$ that are formed by taking the span of $i$-fold brackets of vector fields tangent to $E$ at $v$. These form well-defined distributions $E^i$ on $\Omega$ and we let $r=r(v)$ be the index such that $E^r(v)=T_v\mc{F}(v)$ for $v\in \Omega$. Observe that $r$ is constant on each connected component of $\mc{F}(v)\cap  \Omega$. 

We first establish the following generalization of Theorem \ref{thm:tangentconeN}.

\begin{theorem}[Foliated Tangent Cone Structure Theorem]\label{thm:tangentcone}
	Assume Hypotheses \ref{assump}. Then:

	\begin{enumerate}
		\item $\Omega$ is $f$-invariant and intersects every leaf of $\mc{F}$ in a (possibly empty) open set. 
		\item For every $v\in \Omega$, the tangent cone $TC_v\mc{F}(v)$ of $\mc{F}(v)$ at $v$ exists, is a Carnot nilpotent Lie group with a left invariant \CC metric and varies continuously on the intersection of $\Omega$ with each leaf. 
		\item The map $f$ induces a Lie group isomorphism between the tangent cones $TC_v\mc{F}(v)$ of $\mc{F}(v)$ at $v\in \Omega$ and $TC_{f(v)}\mc{F}(f(v))$ of $\mc{F}(f(v))$ at $f(v)$. We call this the {\em Carnot derivative} of $f$ at $v$.
		\item If for a sufficiently large $k$, the distribution $E$ is uniformly $C^k$  along $\mc{F}$ (see Definition \ref{def:Ckunif} below), then $\Omega$ is open in $X$ and the tangent cone $TC_v\mc{F}(v)$ varies continuously in $v\in \Omega$. 
	\end{enumerate}
Similarly, for a $C^1$ flow $\varphi_t: X\to X$ and flow invariant distributions $E\subseteq D$ we the analogous statements hold.
\end{theorem}

Inspecting the proof below it will be clear that the degree of uniform regularity $C^k$ of $E$ suffices to be one less than the maximum nilpotency degree of the tangent cones on $\Omega$. These nilpotency degrees are in turn all bounded by the dimension of the leaves of $\mc{F}$ minus one.

We will break the proof into two main lemmas.

\begin{lemma}\label{lem:tangentcone}
Assume Hypotheses \ref{assump}. Statements (1)-(3) of Theorem \ref{thm:tangentcone} hold.
\end{lemma}

\begin{proof}
For every $w\in \mc{F}(v)$, we let $E^i(w)$ be the subspace of $D(w)$ spanned by all brackets of vector fields tangent to $E(w)$ of orders at most $i$. Let $n_i(w)=\dim E^i(w)$. Then $n_i(w)$ is $f$-invariant by the $f$-relation of brackets ($df[X,Y]=[df(X),df(Y)]$). Moreover, $n_i$ is lower semicontinuous. Indeed, if a bracket of vector fields of order $i$ does not vanish at a point then it does not vanish in a neighborhood in the unstable manifold of that point. The semicontinuity of $n_i$ implies that $n_i$ is locally constant on an open set of each leaf for every $i$. This latter open subset of each leaf is the intersection of $\Omega$ with that leaf. Because $n_i$ is $f$-invariant, $\Omega$ is also $f$-invariant. 

For every $v\in \Omega$, we consider a small enough neighborhood of $v$ in $\mc{F}(v)$. This neighborhood is a sub-Riemannian manifold on which the horizontal distribution $E$ is $C^\infty$ and generic. By Theorem \ref{thm:Mitchell}, the tangent cone $TC_v\mc{F}(v)$ exists. The continuity of tangent cones follows from Proposition \ref{prop:cont_of_tan_cone}. Moreover, the map $f$ is a $C^1$ diffeomorphism from this neighborhood of $v$ in $\mc{F}(v)$ onto a neighborhood of $f(v)\in \Omega$ in $\mc{F}(f(v))$. The map $f$ preserves the distribution $E$, and so by Proposition \ref{prop:df_isom}, $df$ induces an isomorphism $(df)_*:TC_v\mc{F}(v)\to TC_{f(v)}\mc{F}(f(v))$.

The flow case follows by restricting to a single time map $f=\varphi_t$.
\end{proof}

Note that while $\Omega$ is open in each leaf, the entire $\Omega$ may not be open in $X$. Next, we will discuss the continuity of tangent cones transversely to the foliation $\mc{F}$. We first recall the relevant topology. 

\begin{definition}\label{def:Ckunif}
	Assume a foliation $\mc{F}$ of $X$ has $C^k$ leaves. A function on $X$ is called {\em (transversely) uniformly $C^k$ along $\mc{F}$} if its $k$-th order derivatives exist along leaves of the foliation and are continuous on $X$. A distribution is called uniformly $C^k$ along $\mc{F}$ if there are local vector fields forming a basis of the distribution that are uniformly $C^k$ along $\mc{F}$.
\end{definition} 

We complete the proof of Theorem \ref{thm:tangentcone} with the following Lemma.

\begin{lemma}\label{lem:transverse_cont}
Assume Hypotheses \ref{assump}. If the distribution $E$ is uniformly $C^k$  along $\mc{F}$ for a sufficiently large $k$, then the set $\Omega$ is open in $X$ and the tangent cone $TC_v\mc{F}(v)$ varies continuously in $v\in \Omega$. In particular, statement (4) of Theorem \ref{thm:tangentcone} holds. 
\end{lemma}

\begin{proof}
	Let $v\in \Omega$. The distribution $E$ is horizontal and generic in a neighborhood of $v$ in $\mc{F}(v)$. Since $E$ is uniformly $C^k$, a neighborhood of $v$ in $X$ is contained in $\Omega$ and thus $\Omega$ is open. 
	
	For the second claim, choose $v_0\in \Omega$. As in the proof of \ref{prop:cont_of_tan_cone}, on any sufficiently small neighborhood $U\subset \Omega$ of $v_0$ we may choose a basis $X_{1},\dots,X_{d_0}$ of vector fields spanning $E$ such that:
	\begin{enumerate}
	\item the $X_j$ vary continuously in the $C^k$ topology, and 
	\item if $\set{(j_1,\dots,j_i)}\subset \set{1,\dots,d_0}^i$ is a family of indices  whose corresponding iterated brackets $[\what{X}_{j_1},[\what{X}_{j_2},[\what{X}_{j_3},\dots]\dots]$ form the chosen basis $\set{\what{X}_{i,j}}$ of $\mc{M}_{v_0}^{-1}(\mf{n}_{v_0}^{i-1})$, then the corresponding brackets also form a basis of $\mc{M}_{v}^{-1}(\mf{n}_{v}^{i-1})$ for all $v\in U$.
	\end{enumerate} 
	Indeed, for any fixed index $j$ the vector fields $\what{X}_{j}=\what{X}_{j,v}$ vary continuosly in $v$, as do the bracket operations, and thus the nondegeneracy of these iterated brackets is an open condition. It follows that the corresponding nilpotent structures vary continuously as do the isomorphism classes of the nilpotent tangent cones.
		
	Since the metric on each tangent cone $TC_v\mc{F}(v)$ is the \CC metric for the distributions induced by the inclusion of $E(v)$ in the Lie algebra $\mf{n}_v$ of $TC_v\mc{F}(v)$, the metric also varies continuously on the neighborhood $U$, and thus on all of $\Omega$. 

\end{proof}

Combining this result with good dynamical properties, we obtain the following.

\begin{corollary}\label{cor:uniqueclass}
Assume Hypotheses \ref{assump}. If the distribution $E$ is uniformly $C^k$ for a sufficiently large $k$, then the isomorphism class of tangent cones $TC_v\mc{F}(v)$ is continuous and $f$-invariant for $v\in \Omega$. Moreover, if $f$ is transitive and $E$ is $C^\infty$ somewhere on a leaf then $\Omega$ is open and dense and the isomorphism class of any tangent cone is constant on $\Omega$.  The same statements hold for $f$ replaced by a flow $\varphi_t$.	
\end{corollary} 

In this setting, we call the isomorphism class of tangent cones obtained in Corollary \ref{cor:uniqueclass} the {\em Carnot nilpotent structure associated to $f$ or $\varphi_t$}.

\begin{proof}
	The first claim follows from Theorem \ref{thm:tangentcone}. The topological transitivity of $f$ implies that there is a dense orbit in $X$ which must intersect $\Omega$ since $\Omega$ is open. If $E$ is $C^\infty$ on any open set of a leaf then $\Omega$ is also nonempty and therefore dense since it is invariant under $f$. The isomorphism class of tangent cones on $\Omega$ is constant because isomorphism classes are constant on a dense orbit, depend continuously on the underlying point, and the space of isomorphism classes is Hausdorff by Theorem \ref{thm:BorelSerre}. Lastly, the proof holds verbatim for flows by replacing $f$ by $\varphi_t$.
\end{proof}
\
Kalinin (\cite{Kalinin11}) defined a notion of closing property for diffeomorphisms. We define the analogous notion of {\em closing property} for flows. 
\begin{definition}\label{def:SCP}
	A flow $\varphi_t:X\to X$ of a compact manifold $X$ has the {\em stable closing property} (with exponent $\la$) if there are $c,\lambda, \delta_0>0$ with $\lambda>0$ such that for every $x\in X$ and some $T\in \R$ with $d(x,\varphi_T(x))<\delta_0$ then there are $p,y\in X$ such that
\begin{enumerate}
\item $p$ is periodic with period $T_0$ where $|T_0-T|<\delta_0$, and
\item $d(\varphi_t(x),\varphi_t(p))<c e^{-\la\min\{t,T_0-t\}}d(x,\varphi_T(x))$ for $0\le t\le T_0$, and
\item $d(\varphi_t(p),\varphi_t(y))<c e^{-\la t}d(x,\varphi_T(x))$ and $d(\varphi_t(x),\varphi_t(y))<c e^{-\la(T_0-t)}d(x,\varphi_T(x))$.
\end{enumerate}
\end{definition}
Here $y$ plays a connecting role of being stably related to $p$ and unstably related to $x$.

Examples of dynamical systems having the closing property include Anosov flows and diffeomorphisms and flows over subshifts of finite type (cf. \cite{Kalinin11}).

We deduce Theorem \ref{thm:equal} from the introduction as a conclusion of this section.

\begin{reptheorem}{thm:equal}
	Assume Hypotheses \ref{assump}. Suppose the distribution $E$ is uniformly $C^k$ for a sufficiently large $k$ and $\mc{F}$ is transversally H\"{o}lder continuous. Also assume that the $C^1$-diffeomorphism $f$ (or flow $\varphi_t$) is topologically transitive, satisfies the stable closing property and that $df\rest{D}$ is transversally H\"{o}lder continuous. Lastly, suppose that the graded nilpotent group associated to $D$ on $\Omega$ is asymmetric and $(r+1)$-step. Let $\mu$ be any finite $f$-invariant (resp. $\varphi_t$-invariant) ergodic measure whose support $\op{Supp}(\mu)$ satisfies $\op{Supp}(\mu)\cap\Omega\neq \emptyset$. Then the Lyapunov exponents of $f$ (resp. $\phi_1$) along $\mc{F}$ are $\log\lambda, 2\log\lambda, \dots, (r+1)\log\lambda$ with multiplicities $\dim E,\dim E^{1}-\dim E,\dots,\dim E^{r}-\dim E^{r-1}$ for some $\la=\la_\mu>1$.
\end{reptheorem}

\begin{proof}
Recall $\Omega$ is the subset obtained from Corollary \ref{cor:uniqueclass}. Let $v\in \Omega$ be a periodic point of period $T$. It follows that $(f^T\rest{\mc{F}(v)})_*:TC_v\mc{F}(v)\to TC_{v}\mc{F}(v)$ is a graded automorphism of the nilpotent group $TC_v\mc{F}(v)$. On the other hand, the isomorphism class of $TC_v\mc{F}(v)$ is asymmetric. Hence $(f^T\rest{\mc{F}(v)})_*$ is a homothety.  Suppose $TC_v\mc{F}(v)$ is an $(r+1)$-step Carnot nilpotent Lie group. By Theorem \ref{thm:Lyap_exp}, there exists $\lambda(v)$ such that the Lyapunov exponents of $f^T$ along $\mc{F}$ for $v$ are $T\log\lambda(v), 2T\log\lambda(v), \dots, (r+1)T\log\lambda(v)$.

Since $\op{Supp}(\mu)\cap\Omega\neq \emptyset$, $\Omega$ has full $\mu$-measure as it is invariant and $\mu$ is ergodic. By \cite[Proposition 3.2]{Kalinin11}, Lyapunov exponents of $df\rest{D}$ with respect to any invariant ergodic measure are approximated by those at periodic points ones. (This proposition is stated for standard $\op{GL}(n,\R)$-valued cocycles, however following the discussion after Proposisition 2.6 of \cite{GS2} these estimates may be applied to derivative cocycles as well.)  Therefore, there is $\lambda >1$ such that Lyapunov exponents of $f$ along $\mc{F}$ with respect to $\mu$ are $\log\lambda, 2\log\lambda, \dots, (r+1)\log\lambda$. The proof in the case of flows follows verbatim after replacing $f$ by $\varphi_t$.
\end{proof}

While we need greater regularity in Theorem \ref{thm:equal} the H\"{o}lder continuity of $E$ often follows from the dominated splitting assumption. For example, we have the following (cf. also Theorem 5.3.2 of \cite{BarreiraPesin}).

\begin{theorem}[Theorem 4.11 and Remark 4.12 of \cite{CrovisierPotrie}]\label{thm:Holder-dominated}
	Let $f: X \rightarrow X$ be a $C^{2}$-diffeomorphism and $\Lambda \subset X$ a compact $f$-invariant set admitting a dominated splitting of the form $T_{\Lambda} X=E \oplus F .$ Then, there is $\theta \in(0,1]$ such that the bundles $E$ and $F$ are $\theta$
	H\"{o}lder continuous. Moreover, $\theta \in(0,1]$ can be chosen to be any number such that for all sufficiently large $N\ge 1$, we have
	\[
	\|Df^N\rest{E(x)}\|\|Df^N\rest{F(x)}\|^\theta<\|Df^{-N}\rest{F(x)}\|^{-1} \text{ for every } x\in \Lambda.
	\]
\end{theorem}
	
\begin{remark}
	The analogous statement for flows $\varphi_t:X\to  X$ follows from the above result applied to the time one map. Moreover, the same proof works for arbitrary $C^{\alpha}$ cocycles over $C^1$-diffeomorphisms or flows admitting a dominated splitting for $\alpha\in (0,1]$.
		
	Lastly, note that Anosov flows with $E=E^s$ is a special case (cf.  \cite[Appendix by Brin, Proposition 4.4]{BallmannBook}).
\end{remark}

In the next sections we consider applications of our previous results to sundry geometric settings.

\section{Application: Local Rigidity of Smooth Dominated Splittings}

In this section we will present some cases where the existence of a leafwise smooth invariant splitting of the unstable bundle leads to some form of rigidity. This provides some evidence for positive answers to the Questions \ref{ques:splitting}. However, we first show that smooth time changes of flows with uniformly $C^k$ slow distributions along the unstable foliation maintain this property. Hence, in general, the best rigidity we may expect is smooth orbit equivalence.

\subsection{Time Changes}\label{subsec:time-change}
Let $\varphi_t:{N}\to {N}$ be the $C^{\infty}$ flow generated by the vector field $x\mapsto X_x$. Given a positive $C^\infty$ function $\ell:{N}\to (0,\infty)$ consider the $C^\infty$ flow $\psi_t$ whose generating vector field is $x\mapsto\ell(x) X_x$. The cumulative time change $\tau:{N}\times \R\to \R$ is given by $\tau(x,t)=\int_0^t \ell(\psi_s x)ds$, and hence $\psi_t(x)=\varphi_{\tau(x,t)}(x)$. The inverse function $\sigma(x,t)=\tau^{-1}(x,t)$ satisfies $\varphi_t(x)=\psi_{\sigma(x,t)}(x)$ and corresponds to the cumulative time change for the field rescaling function $k(x)=\frac{1}{\ell(x)}$.

Anosov and Sinai (\cite{AnosovSinai67}) showed that if $\varphi_t$ is a $C^{\infty}$ Anosov flow and $\psi_t$ is obtained from $\varphi_t$ by the multiplication of speeds by $\ell$, then $\psi_t$ is Anosov. 
In fact, they showed that we have a hyperbolic splitting,
\[
T{N}=E^{u}_{\psi_t}\oplus E^{s}_{\psi_t}\oplus E^0_{\psi_t}
\]
for $D \psi_{t}$ where 
\[
E^{u}_{\psi_t}(x)=\set{ v+z(x,v)X_x \,:\, v\in E^u_{\varphi_t}(x) }
\]
for some real valued function $z(x,\cdot):E^u_{\varphi_t}(x)\to \R$ which is linear in its second argument and continuous in $x$. Following Parry \cite{Parry86} we may express the function $z$ as,
\begin{align*}
z(x, \xi)=-\ell(x) \int_{0}^{\infty}\left(D \varphi_{s} \xi(\ell) / \ell\left(\varphi_{s} x\right)^{2}\right) d s=\ell(x) \int_{0}^{\infty} \xi\left(k \circ \varphi_{s}\right) ds.
\end{align*}

We now observe that if $\ell$ is smooth then so are $\tau$, $k$ and $\sigma$ as functions on $N$ or $N\times \R$. Moreover, if the invariant splitting for $E^u_{\varphi_t}=E^u_{slow,\varphi_t}\oplus E^u_{fast,\varphi_t}$ is leafwise smooth, then the graphs,
\[
E^u_{slow,\psi_t}=\set{v+ z(x,v)X_x \,:\, v\in E^u_{slow,\varphi_t}}\quad\text{and}\quad E^u_{fast,\psi_t}=\set{v+ z(x,v)X_x \,:\, v\in E^u_{fast,\varphi_t}}
\]  
remain leafwise smooth and invariant under ${\psi_t}$. This last statement follows from the formula,
\begin{align*}
	D \psi_{\sigma(x,t)}\left(v+z(x, v) X_{x}\right)=D \varphi_{t} v+z\left(\varphi_{t} x, D \varphi_{t} v\right) X_{\varphi_{t}(x)}
\end{align*}
which we apply in the forward time to $E^u_{slow,\psi_t}$ and the backward time in the case of $E^u_{fast,\psi_t}$. From the explicit formula for $z(x,v)$ we see that its overall regularity restricted to the subbundle $E^u_{slow,\varphi_t}$ will be the same as that of $E^u_{slow,\varphi_t}$ since $\ell$, $k$ and $\varphi_s$ are smooth.

We summarize this discussion in the following proposition.
\begin{proposition}
If a $C^\infty$ Anosov flow $\varphi_t$ has a uniformly $C^k$ slow unstable distribution along its unstable foliation, then a smooth time change $\psi_t$ also has this property. 
\end{proposition}

Observe that, in general, such $C^\infty$ time changes $\psi_t$  are not conjugate to $\varphi_t$ as periodic orbits may have different periods. 
To obtain explicit examples, we may apply the above proposition to geodesic flows of compact locally symmetric spaces or suspensions of Anosov diffeomorphisms with smooth slow distributions. Flows of the latter type arise from algebraic examples of splittings, e.g. from products or toral automorphisms. 

\subsection{Geodesic Flows}
We next apply the discussion on sub-Riemannian dynamics from the last section to geodesic flows of closed manifolds of negative sectional curvature.
We will prove a local rigidity statement for certain perturbations of a  locally quaternionic hyperbolic or octonionic hyperbolic closed manifold.

\begin{definition}\label{def:unifCr}
	Let $\mc{A}$ be a topological space and let $\set{\mc{F}_\rho}_{\rho\in\mc{A}}$ be a parametrized family of foliations of a manifold $M$. We say that a parametrized family of dominated splittings $\set{T\mc{F}_\rho=E_\rho\oplus F_\rho}_{\rho \in \mc{A}}$ is {\em uniformly $C^k$ continuous} in the parameter if $E_\rho$ and $F_\rho$ are $C^k$ along each leaf of the foliation $\mc{F}_{\rho}$ and  their derivatives up to order $k$ vary continuously in the parameter $\rho$ with respect to the given topology on the index set $\mc{A}$. We say two splittings $E_\rho\oplus F_\rho$ and $E_{\rho_0}\oplus F_{\rho_0}$ from a uniformly $C^k$ family of splittings are {\em $\eps$-uniformly $C^k$ close} if their derivatives up to order $k$ along their respective foliations $\mc{F}_{\rho}$ and $\mc{F}_{\rho_0}$ are $\eps$-close.
\end{definition}

\begin{remark}
	By the Permanence Theorem \cite[Theorem 6.8]{HirschPughShub} the $W^u_g$ foliations for the geodesic flow of $g$, though only transversely H\"older on all of $M$, have leafwise smooth distributions that vary continuously in the $C^k$ topology as $g$ varies continuously in the $C^{k+1}$ topology. Even though $E^u_{fast,g}$ is uniformly $C^k$ along $W^u_g$ on each $(M,g)$, $E^u_{slow,g}$ is only continuous along $W^u_g$ in general. However, even when $E^u_{slow,g}$ is assumed to be $C^\infty$ along leaves of $W^u_g$, it is not necessarily uniformly $C^k$ on $M$ and may not even be uniformly $C^k$ in the metric parameter $g$ as it varies in the $C^{k+1}$ topology for any $k>0$. We will need to assume this latter property in some of the statements below.
\end{remark}

	Consider a smooth n-manifold $X$. For $n_1\leq n_2\leq \dim X$, let $\op{Flag}(n_1,n_2)\to X$ be the smooth fiber bunde whose fiber $\op{Flag}_x(n_1,n_2)$ is the space of all 2-flags of $n_1$ dimensional subspaces of $n_2$ dimensional subspaces of $T_xX$. Now let $\scr{EF}$ be the subset of $C^0$ sections $x\to (E_x\subset F_x) \in\op{Flag}_x(n_1,n_2)$ such that the corresponding distribution $F$ generates a $C^0$ foliation $\mc{F}$ with $C^\infty$ leaves and the $E$ distribution is $C^\infty$ along the leaves of the foliation $\mc{F}$. We write these sections as $E\subset F$ for corresponding distributions $E$ and $F$. The $C^{0,k}$ topology on $\scr{EF}$ is the maximal refinement of the subspace topology inherited from the compact-open topology on $C^0$ sections of $\op{Flag}(n_1,n_2)$ such that the partial derivatives up to order $k$ of $E$ along the leaves of $\mc{F}$ vary continuosly.

	\begin{lemma}\label{lem:C0k}
		Let $X$ be a smooth closed manifold with a pair of distributions $(E_0\subset F_0) \in \scr{EF}$ such that $E_{0}$ restricted to each leaf is generic of order $r+1$ along the leaves of $\mc{F}$. Then there is an open set $U\subset \scr{EF}$ in the $C^{0,r}$ topology such that each $(E\subset F)\in  U$ is generic of order $r+1$ and the subbundles $E=E^0,E^1,\dots,E^r$ of iterated bracket spaces vary continuously as subbundles of $TX$.
	\end{lemma}
	
	\begin{proof}
		Since $X$ is closed, $F_0$ defines a foliation $\mc{F}$, and $E_0$ is a given distribution, we may find a finite open cover of foliation charts for $(X,\mc{F})$ where on each open chart $V$, the distribution $E_0$ is spanned by continuous vector fields $X_1,\dots,X_{d_0}$ which are $C^\infty$ on each leaf of $\mc{F}$ and whose successive brackets at each point $x$ generate a filtration $E_0(x)=E^0_0(x)\subset E_0^1(x)\subset E_0^r(x)=F_x$ and fit together to form subbundles locally and hence globally. 
		
		Since  the $X_i$ vary continuously and the iterated brackets along each leaf vary continuously in the parameters of the $C^{0,r}$ topology, there is an open neighborhood $U\subset\scr{EF}$ of $(E_0\subset F_0)$ in the $C^{0,r}$ topology for which the iterated brackets of the nearby $X_i$ spanning $E$ continue to be nonzero on all of $V$. (For sufficiently close foliations $\mc{F}$ for the distribution $F$, the open set $V$ still serves as a chart.)
		
		Since there are only finitely many charts $V$, by intersecting the resulting open sets $U$ we may find a common neighborhood. 
	\end{proof}

\begin{lemma}\label{lem:generic}
	Let $(M,g_0)$ be a closed negatively curved manifold admitting a leafwise $C^\infty$ dominated splitting $E^u_{\phi_t^0}=E^u_{fast,\phi_t^0}\oplus E^u_{slow,\phi_t^0}$ for its geodesic flow $\phi_t^0$. Assume $(E^u_{slow,\phi_t^0}\subset E^u_{\phi_t^0})\in \scr{EF}$ and that $E^u_{slow,\phi_t^0}$ is generic and horizontal of order $r+1$ along the leaves of $W^u_{\phi_t^0}$. Recall by Proposition \ref{prop:dom-perturb} there is a $C^{1}$ neighborhood $U$ of $\phi_t^0$ such that any flow $(SM,\phi_t)\in U$ has a dominated splitting $E^u_{fast,\phi_t}\oplus E^u_{slow,\phi_t}$. If we further assume that this splitting is $C^\infty$ along unstable leaves and sufficiently uniformly $C^r$ close (along the unstable foliation) to that of $\phi_t^0$ then $(E^u_{slow,\phi_t}\subset E^u_{\phi_t})\in \scr{EF}$ and $E^u_{slow,\phi_t}$ is generic and horizontal of order $r+1$ along the leaves of $\mc{F}$. Furthermore, the tangent cone map $v\mapsto TC_vW^u_{\phi_t}$ associated to $(M,\phi_t)$ exists and is $C^0$ close to  $v\mapsto TC_vW^u_{\phi_t^0}$, the tangent cone map for $\phi_t^0$, with respect to the Hausdorff topology on the moduli space $\mc{NL}_{n-1}/\GL(n-1,\R)$.
\end{lemma}

\begin{proof}
	Recall that the dominated splitting condition is an open condition in the $C^r$ topology and thus it holds on some neighborhood $U$ of $\phi_t^0$. 
		
	As in the proof of Lemma \ref{lem:transverse_cont}, we exploit the fact that both the Carnot metric and algebraic structure of the tangent cone depends on the construction of the M\'{e}tivier fields $\what{X}_{j,k;x}$.
		
	By assumption, the 2-flag of distributions $E^u_{slow,\phi_t}\subset E^u$ are $C^\infty$ along $W^u_{\phi_t}$ and uniformly close to $E^u_{slow,\phi_t^0}\subset E^u_{\phi_t}$ in the $C^r$ topology. Provided they are close enough, then by Lemma \ref{lem:C0k} $E^u_{slow,\phi_t}$ is generic of order $r+1$, with nearby subbundles of iterated bracket spaces. In particular $E^u_{slow,\phi_t}$ is horizontal as well. Moreover, the construction of the $\what{X}_{r-i,j;x}$  for $i=0,\dots,r$ vary continuously in the $C^i$ topology on the vector fields $X_{0,1;x},\dots, X_{0,d_0;x}$ locally spanning $E^u_{slow,\phi_t}$. By hypothesis these latter fields depend continuously in the $C^r$ topology in $\phi_t$. We conclude as in the proof of Lemma \ref{lem:transverse_cont} and \ref{prop:cont_of_tan_cone} that the nilpotent structure in $\mc{NL}_{n-1}/\GL(n-1,\R)$ is $C^0$ close as claimed.
\end{proof}

\begin{remark}\label{rem:metrics}
	We observe that if metrics vary in the $C^{r+1}$ topology, then the geodesic flows vary in the $C^r$ topology. Hence the same Lemma holds for metrics except that $U$ is a $C^2$ neighborhood of $g_0$.
\end{remark}

\begin{reptheorem}{thm:locally_rigid_flow}
	Let $\phi_t^0$ be the geodesic flow on a locally quaternionic hyperbolic or octonionic hyperbolic closed manifold $M$. Then if $\phi_t$ is any $C^\infty$ flow $C^1$ close to $\phi_t^0$ for which $E^u_{slow,\phi_t}$ remains $C^\infty$ along unstable leaves and is sufficiently uniformly $C^1$ close (in the sense of Definition \ref{def:unifCr}) to that of $\phi_t^0$, then $\phi_t$ is $C^\infty$ orbit equivalent to $\phi_t^0$. 

\end{reptheorem}

\begin{proof}
	We observe that the symmetric flow $\phi_t^0$ has a dominated splitting $E^u=E^u_{fast}\oplus E^u_{slow}$ which is globally $C^\infty$ and for which $E^u_{slow}$ is generic of order $1$ and horizontal. By Lemma \ref{lem:generic} there is a $C^1$ open neighborhood $U$ of $\phi_t^0$ where $E^u_{slow,\phi_t}$ is generic of order $1$ and horizontal whenever $\phi_t\in U$ has a dominated splitting which is $C^\infty$ along unstable leaves and uniformly $C^1$ close along the unstable foliation to that of $\phi_t^0$. Moreover the associated tangent cones $TC_vW^u(v)$ for $\phi_t$ are $C^0$ close to those of $\phi_t^0$.
	
	Recall that by Proposition \ref{examp:nilrigid} the tangent cone associated to $(M,\phi_t^0)$ are asymmetric. By Corollary \ref{cor:stable-asymmetry}, after possibly shrinking $U$, the tangent cone associated to $(M,\phi_t)$ is also asymmetric. We can shrink $U$ to an even smaller set so that the tangent cone associated to $(M,\phi_t)$ is 2-step. Thus by Theorem \ref{thm:equal}, the unstable Lyapunov exponents are $\log\lambda$ and $2\log\lambda$ for a $\lambda>1$. The multiplicity of the exponents are the dimensions of the slow and the fast unstable distribution. Therefore, the Lyapunov spectra of the geodesic flows on $(M,\phi_t)$ and $(M,\phi_t^0)$ are proportional.
	
	We wish to apply Theorem 3.2 of \cite{Butler:17} which is stated for the so-called {\em horizontal measure}, $\mu_M$. We need to verify that this measure has full support. We observe that by Section 2.8 of \cite{Butler:17} this is a Gibbs state for a H\"older potential. These have full support in our setting (\cite[p.50]{PaulinEtAl15} or \cite[Theorem 4.11]{ClimenhagaEtAl:18}). 
\end{proof}

\begin{remark}
	By the Livsic Theorem whenever two orbit equivalent flows have the same periods on periodic points they are conjugate. The regularity of the conjugacy will be at least one less than the regularity of the time change. 
\end{remark}

\begin{reptheorem}{thm:locally_rigid}
	Let $g_0$ be a locally quaternionic hyperbolic or octonionic hyperbolic metric on a smooth closed manifold $M$. Then $g_0$ is locally rigid within the family of $C^2$ close $C^\infty$ metrics whose $E^u_{slow}$ remains $C^\infty$ along unstable leaves and is sufficiently uniformly $C^1$ close (in the sense of Definition \ref{def:unifCr}) to that of $g_0$.

\end{reptheorem}

\begin{remark}
	The above theorem is reminiscent of the global rigidity result (Theorem 3) in \cite{BFL92}. While we do not assume any transverse regularity of the splitting in our theorem, we do require that the metric is nearby the locally symmetric one. We also note that there is a $C^2$ open neighborhood of $g_0$ such that the unstable distributions of the perturbed manifolds admit dominated splittings by Proposition \ref{prop:dom-perturb}.
\end{remark}

\begin{proof}
	The proof is nearly identical to that of Theorem \ref{thm:locally_rigid_flow}, except that we note by Remark \ref{rem:metrics} there is a $C^2$ open neighborhood $U$ of $g_0$ with the desired properties. Moreover, we finish using Theorem 1.5 of \cite{Butler:17} instead of Theorem 3.2.
\end{proof}

\begin{remark}
	If $g_0$ is any metric as in Lemma \ref{lem:generic} whose associated nilpotent tangent cone has an asymmetric Lie algebra, then by Theorem \ref{thm:equal}  its unstable Lyapunov spectrum is of the form $\log \la,2\log \la, \dots, (r+1) \log \la$ for some $\la>1$. We can conclude by the proof of Theorem \ref{thm:locally_rigid} that any metric that is sufficiently uniformly $C^2$ close along unstable leaves will also have a Lyapunov spectrum of the same form. 
\end{remark}

\subsection{Diffeomorphisms of nilmanifolds}
In this section we apply our results to perturbations of automorphisms of certain nilmanifolds with smooth dominated splittings. 

\noindent {\em Standing assumption:} Let $M$ be a closed manifold and let $f_0$ be a transitive $C^\infty$ Anosov diffeomorphism such that tangent cones of unstable leaves exist everywhere and are isomorphic to a fixed asymmetric $(r+1)$-step Carnot nilpotent Lie group $N$. Let $E^u_{slow,f_0}$ be the horizontal distribution along unstable leaves that gives rise to the tangent cone structure.

We obtain the following local rigidity theorem.
\begin{reptheorem}{thm:samespec}
	 There is a $C^1$ open neighborhood $U$ of $f_0$ in $\Diff^\infty(M)$ such that if $f \in U$ admits a smooth splitting $E^u_f=E^u_{fast,f}\oplus E^u_{slow,f}$ along unstable leaves with $\dim(E^u_{slow,f})=\dim(E^u_{slow,f_0})$, and $E^u_{slow,f}$ is sufficiently uniformly $C^{r}$ close along unstable leaves to $E^u_{slow,f_0}$, then for any invariant ergodic measure $\mu$ there is $\lambda_\mu>1$ such that the unstable Lyapunov exponents of $f$ with respect to $\mu$, are $\log \lambda_\mu,2\log\lambda_\mu,\dots,(r+1)\log\lambda_\mu$ with the same multiplicity as for $f_0$.
\end{reptheorem}

\begin{proof}
By the construction of $f_0$, tangent cones of unstable manifolds of $f_0$ exist everywhere and are isomorphic as graded nilpotent Lie groups to $N$.

We argue as in the proof of Theorem \ref{thm:locally_rigid}. The tangent cones of unstable manifolds of $f$ exist everywhere, and moreover such tangent cones are close to $N$ in the variety of Carnot nilpotent Lie groups. By Corollary \ref{cor:stable-asymmetry}, those tangent cones are asymmetric. By Theorem Theorem \ref{thm:equal}, there is $\la_\mu>1$ such that all unstable Lyapunov exponents, without multiplicity, of $f$ with respect to $\mu$, are $\log\lambda_\mu,2\log\lambda_\mu,\dots,(r+1)\log\lambda_\mu$. The multiplicities are given by the dimensions of the bracket spaces starting with $E^u_{slow,f}$ which have the same dimensions as for $f_0$.
\end{proof}

Recall that for any transitive Anosov diffeomorphism $f$ there is a unique invariant SRB measure, which we call $\sigma$ (See 20.3.8 of \cite{HasselblattKatok} and \cite{Sinai:72}). For $f_0$ as in Theorem \ref{thm:samespec}, we denote by $\sigma_0$ the unique SRB measure for $f_0$. In the next corollary we obtain a rigidity statement under an additional assumption which is analogous to a higher hyperbolic rank condition (cf. below).
 
\begin{corollary}
	Suppose we are in the setting as Theorem \ref{thm:samespec} with $M$ a nilmanifold and $f_0$ an automorphism. In addition assume $\la_{\sigma}(f)\geq \la_{\sigma_0}(f_0)$, then $\sigma$ is the measure of maximal entropy for $f$ and has the same Lyapunov spectrum as $\sigma_0$ for $f_0$.
\end{corollary}

\begin{proof}
	Set $\dim E^{-1}=0$ and let $d_i=\dim E^{i}-\dim E^{i-1}$ for $i=0,\dots,r$. By the Pesin formula and Theorem \ref{thm:samespec}, we also have 
	$h_{\sigma}(f)=\sum_{i=0}^r (i+1) \log\la_{\sigma}(f) d_i$. Moreover, by Structural Stability $f$ and $f_0$ are conjugate and so they share the same topological entropies $h_{top}(f)=h_{top}(f_0)$. Since $f_0$ is algebraic, the Lyapunov exponents are independent of the measure and so Haar measure $\sigma_0$ is the measure of maximal entropy. Together with our assumptions we have,

	\[
	h_{top}(f)\geq h_{\sigma}(f)=\sum_i i \log\la_{\sigma}(f) d_i \geq  \sum_i i \log\la_{\sigma_0}(f_0) d_i =h_{\sigma_0}(f_0)=h_{top}(f_0)=h_{top}(f).
	\]
	Hence all of the inequalities are equalities. Therefore $h_{top}(f)=h_{\sigma}(f)$ and hence $\sigma$ is the measure of maximal entropy for $f$. Moreover, $\la_{\sigma}(f)=\la_{\sigma_0}(f_0)$ and so the spectra coincide.
\end{proof}

As an application we describe a class of examples of special Anosov nil-automorphisms.
\begin{definition}\label{def:rationalG}
	A nilpotent Lie algebra $\mf{n}$ is called rational if it is a Lie algebra over $\Q$. A nilpotent Lie group $N$ is called rational if there is a rational Lie algebra $\mf{n}$ such that $Lie(N)=\mf{n} \otimes \R$.
\end{definition}

\begin{example}\cite[Theorem 3.1]{Lauret03}\label{ex:alg-Anosov}
	Let $N$ be a graded rational nilpotent group. Then there is a lattice $\Gamma<N\times N$ such that the nilmanifold $M=(N\times N)/\Gamma$ admits an Anosov automorphism in which unstable and stable manifolds are vertical and horizontal copies of of $N$.
\end{example}
We obtain the following as an immediate corollary to Theorem \ref{thm:samespec}.
\begin{repcorollary}{cor:nilexamp}
	Let $N$ be an asymmetric $k$-step Carnot rational nilpotent group. Let $f_0$ be the Anosov automorphism of $M=(N\times N)/\Gamma$ obtained from Example \ref{ex:alg-Anosov}. Then there is a $C^1$ open neighborhood $U$ of $f_0$ in $\Diff^\infty(M)$ such that if $f \in U$ admits a smooth dominated splitting $E^u_f=E^u_{fast,f}\oplus E^u_{slow,f}$ along unstable leaves and $E^u_{slow,f}$ is sufficiently uniformly $C^{r}$ close along unstable leaves to $E^u_{slow,f_0}$, then for any invariant ergodic measure $\mu$ there is $\lambda_\mu>1$ such that the unstable Lyapunov exponents of $f$ with respect to $\mu$, are $\log\lambda_\mu,2\log\lambda_\mu,\dots,(r+1)\log\lambda_\mu$ with the corresponding multiplicities as for $f_0$.
\end{repcorollary}

\begin{remark}
	In the above corollary, one may replace the hypothesis of a smooth dominated splitting with that of a slow distribution $E^u_{slow,f}$ smooth along unstable leaves. This readily follows from Proposition \ref{prop:dom-perturb}.
\end{remark}

Even though the Heisenberg group is not asymmetric, we may still obtain resonance of Lyapunov exponents for Anosov diffeomorphisms on products of this group.

\begin{example}\label{ex:Smale-Anosov}
	The following example due to Borel \cite[Example 1 in I-(3.8)]{Smale67} provides an Anosov diffeomorphism on a nilmanifold whose unstable tangent cones are not asymmetric nilpotent groups. Let $N$ be the 3-dimensional Heisenberg group. Consider $(N\times N)/\Gamma$ where $\Gamma$ is the lattice subgroup whose elements (with respect to the standard basis on the Lie algebra) are of the form $(\alpha,\beta,\gamma,\alpha^{\sigma},\beta^{\sigma},\gamma^{\sigma})$ for $\alpha,\beta,\gamma\in \Z[\sqrt{3}]$ and where $\sigma$ represents the Galois automorphism with $(a+b\sqrt{3})^\sigma=a-b\sqrt{3}$. Let $f_0$ be the Anosov diffeomorphism induced from the Lie algebra map $A(a,b,c,d,e,f)=(\la a,\la^2 b,\la^3 c,\la^{-1} d,\la^{-2} e,\la^{-3}f)$ for $\la=2+\sqrt{3}$. Here the unstable corresponds to the first factor of $N$ with coordinates $(a,b,c)$ and the slow distribution has coordinates $(a,b)$ which is the standard horizontal (and generic) distribution of $N$.
\end{example}

We may obtain a similar corollary to Corollary \ref{cor:nilexamp} for the above example.

\begin{repcorollary}{cor:nilexamp2}
	Let $N$ be the 3-dimensional Heisenberg group. Let $\Gamma$ be the lattice and $f_0$ the Anosov automorphism of $M=(N\times N)/\Gamma$ obtained from Example \ref{ex:Smale-Anosov}. Then there is a $C^1$ open neighborhood $U$ of $f_0$ in $\Diff^\infty(M)$ such that if $f \in U$ admits a smooth splitting $E^u_f=E^u_{fast,f}\oplus E^u_{slow,f}$ along unstable leaves and $E^u_{slow,f}$ is sufficiently uniformly $C^{1}$ close along unstable leaves to $E^u_{slow,f_0}$, then for any invariant ergodic measure $\mu$ the unstable Lyapunov exponents for $f$ are $\log \lambda_1,\log \lambda_2$ and $\log \lambda_1+\log \lambda_2$ for some $\la_1>1$ and $\la_2>1$ depending on $\mu$.
\end{repcorollary}

\begin{proof}
We note that this result follows immediately from the analogue of Theorem \ref{thm:samespec} for Heisenberg groups in place of asymmetric groups and resonance in place of arithmeticity in the conclusion. The proof may be duplicated except for two points.

First, we need to show that, under the perturbation, the unstable tangent cones for $f$ are still isomorphic to the Heisenberg group.  Since $E^u_{slow,f}$ is sufficiently uniformly $C^1$ close along unstable leaves to $E^u_{slow,f_0}$, the horizontal distribution of $TC_vW^u_f(v)$ remains codimension one, horizontal and generic. Hence $TC_vW^u_f(v)$ is a 2-step Carnot nilpotent Lie group with one dimensional center for every $v\in SM$. Moreover, the tangent cones for $f$ are $C^0$ close to those of $f_0$. The Lie bracket induces a  bilinear form  $[\cdot,\cdot]:\mf{n}_v^0\times \mf{n}_v^0\to \mf{n}_v^1$ between  the first and second levels $\mf{n}_v^0$ and $\mf{n}_v^1$ of the Lie algebra $\mf{n}_v$ of $TC_vW^u_f(v)$, and it remains nondegenerate. Consequently, as the center is one-dimensional and there is only one non-degenerate skew symmetric bilinear form on $\mf{n}_v^0$ for every $v\in SM$, $TC_vW^u(v)$ is  isomorphic as a graded nilpotent group to the Heisenberg group $N$.

Second, we need an analogous result to Theorem \ref{thm:equal}. In turn, we observe that in the proof of Theorem \ref{thm:equal}, we may replace Theorem \ref{thm:Lyap_exp} with Theorem \ref{thm:Hei-spec} to obtain the corresponding analogue for the Heisenberg group in the hypotheses and the resonance in the conclusion. Having done so, the analogous result follows.
\end{proof}

It is easy to construct examples of Anosov diffeomorphisms with integrable slow distributions on tori or products of nilmanifolds. One might ask whether for irreducible Anosov automorphisms of nonabelian nilpotent groups the slow distribution must be horizontal and generic. However, this is not the case. Indeed, the slow unstable distribution of \cite[Example 2 on p.189]{Shub1969} is integrable.

As our final application of smooth splittings we will investigate the local rigidity of the projective action of a quaternionic or octonionic hyperbolic lattice on its ideal boundary sphere.

\subsection{Local Rigidity of Projective Boundary Actions}\label{subsec:boundary_actions}

Sullivan initiated the study of local rigidity of lattice actions on the boundary in \cite{Sullivan85}. Ghys (\cite{Ghys93}) introduced a suspension construction which relates actions of fundamental groups of surfaces on the circle to Anosov flows. This was later adapted by Yue (\cite{Yue95}) to the study of lattice actions on higher dimensional spheres. We will employ this construction to obtain a rigidity theorem for boundary actions (cf. also \cite[Theorem 6.1]{Asaoka17}). 

Consider a rank one symmetric space $X$ with connected component $G<\Isom(X)$ of its isometry group. As is well known, the ideal boundary $\partial X$ of $X$ may be identified as $\partial X=G/P$ for a (minimal) parabolic subgroup $P$. For a discrete subgroup $\Gamma < \Isom(X)$, let $\rho_0$ be the  action of $\Gamma$ on $\partial X$ induced from the action of $\Gamma$ on $X$. Note that $\rho_0$ is nothing but the left action on $G/P$ and thus is $C^\infty$. This action preserves the projection to the boundary of the slow unstable distribution $E_0$ for the geodesic flow since the stable foliation and the holonomy along it are $C^\infty$ and $\Gamma$-equivariant. 

We briefly recall the suspension construction mentioned above. Let $\rho:\Gamma\to \op{Diff}^\infty(\partial X)$ be a $C^1$ close perturbation of  $\rho_0$. More precisely, we assume that $\rho$ is $C^1$-close to $\rho _0$ on a fixed finite set of generators of $\Gamma$ (we recall that lattices in semisimple groups  are  always finitely generated). Consider the unit tangent bundle $SM$ for the locally symmetric space and note $\til M=X$. We have a diffeomorphism between $SX$ and $(\partial X\times \partial X \bs D)\times \R$ where $D$ is the diagonal of  $\partial X\times \partial X$.

We also have a diffeomorphism $\what{q}:SX\to X\times \partial X$ via $v\mapsto \what{q}(v)=(\pi(v),v(\infty))$, where $\pi(v)$ is the projection of $v$ to its base point. The map $\what{q}$ descends to a diffeomorphism $q:SM\to (X\times\partial X)/\what{\rho}_0$ where $\what{\rho}_0(\gamma)(x,\xi)=(\gamma x,\rho_0(\gamma) \xi)$ for all $\gamma\in \Gamma$ and $(x,\xi)\in X\times \partial X$. (Here $\Gamma=\pi_1(M)$ acts by deck transformations on $X$.)

As $\rho$ is $C^1$-close to $\rho_0$, by Proposition \ref{prop:Benven} there is a $C^\infty$ diffeomorphism $f:(X\times\partial X)/\what{\rho}\to (X\times\partial X)/\what{\rho}_0$ where $\what{\rho}(\gamma)(x,\xi)=(\gamma x,\rho(\gamma)\xi)$. 
Observe that the actions of $\what{\rho}$ and $\what{\rho}_0$ both preserve the leaves of the first-factor foliation $\set{X\times \set{\xi}:\xi\in \partial X}$.  Moreover, the map $f$ maps horizontal leaves in $(X\times\partial X)/\what{\rho}$ to leaves that are $C^1$-close to horizontal leaves in $(X\times\partial X)/\what{\rho}_0$. The manifolds $\what{q}^{-1}(X\times \set{\xi})$, as $\xi$ varies, are the leaves of the center-stable foliation $W^{cs}$ for the geodesic flow $\phi_t$. For each $x\in X$ we let $\sigma_x: \partial X\to \partial X$ be the map which takes $\xi\in\partial X$ to the opposite endpoint of the geodesic line through $x$ and $\xi$. Then $\what{q}^{-1}(\set{(y,\sigma_y\of\sigma_x(\xi)): y\in X})$ is precisely the center-unstable manifold $W^{cu}(\what{q}^{-1}(x,\xi))\subset SX$ for the geodesic flow $\phi_t$ on $S\til M$. 

Let $V^{cs}$ be the $C^\infty$ foliation of $SM$ whose leaves are the manifolds $V^{cs}(v)=q^{-1}\of f([X\times{v(\infty)}])$  for $v\in SM$, where square brackets denote the $\what{\rho}$-equivalence class. 
Note similarly that the center-stable foliation $W^{cs}$ on $SM$ for the locally symmetric metric has leaves $W^{cs}(v)=q^{-1}([X\times{v(\infty)}])$ where square brackets denote the $\what{\rho}_0$-equivalence class.

Intersecting $V^{cs}\cap W^{cu}$ gives a $C^\infty$ one dimensional foliation $\scr{F}$ which is $C^1$ close to the locally symmetric geodesic foliation.
Note that for all $y$ in the geodesic line through $x$ and $\xi$, $\sigma_y\of\sigma_x(\xi)=\xi$. Hence the foliation $\scr{F}$ is the image under $q^{-1}\of f$ of the quotient by $\what{\rho}$ of the symmetric geodesic foliation (of the first factor) in $X\times \partial X$.

Using the locally symmetric metric, we obtain a norm on the tangent space to $\scr{F}$ from which we may construct a unit vector field $Y$ on $SM$ tangent to these curves. By structural stability, the field $Y$ defines a smooth Anosov flow $\psi_t$ which is H\"older orbit equivalent to the locally symmetric flow $\phi_t$. (Note that by invariance of $V^{cs}$ and $W^{cu}$, and the fact that $\psi_t$ is $C^1$ close to $\phi_t$, it follows that $V^{cs}$ and $W^{cu}$ are the center-stable and center-unstable foliations for $\psi_t$.)

\begin{reptheorem}{thm:boundary_rigid}
	Let $X$ be a quaternionic hyperbolic space or the Cayley plane, and let  $M=X/\Gamma$ for a cocompact lattice $\Gamma<\Isom(X)$. Let $\rho:\Gamma\to \op{Diff}^\infty(\partial X)$ be a $C^1$ close perturbation of  $\rho_0$ which still preserves a $C^\infty$ distribution $E$ nearby to $E_0$. Then $\rho$ is $C^\infty$ conjugate to $\rho_0$.
\end{reptheorem}

\begin{proof}
	Consider the slow distribution $E^u_{slow}$ for $\psi_t$. Since $V^{cs}$ is $C^\infty$, the center-stable holonomy is $C^\infty$. Moreover $E^u_{slow}\subset V^u$ is holonomy invariant (e.g. by Lemma 4 of \cite{FeresKatok1990} and cf. Lemma 5.5 \cite{CNS2018}). Since $\phi_t$ has a dominated splitting on $W^{u}$ and is $C^1$-close to $\psi_t$, the latter flow also admits a dominated splitting of $V^{u}\subset W^{cu}$ and $E^u_{slow}\subset V^{u}$ is also continuous. In what follows let $\psi_t^{\rho_0}$ denote the flow on $(X\times \partial X)/\what{\rho_0}$ given by  pulling back $\psi_t$ under $q^{-1}$. Similarly, let $\psi_t^{\rho}$ denote the flow on  $(X\times \partial X)/\what{\rho}$ given by pulling back $\psi_t^{\rho_0}$ under $f$.
	
	It is easier to understand the center-stable holonomy using $\psi_t^{\rho}$ instead of $\psi_t$. On $(X\times\partial X)/\what{\rho}$, the center-stable leaves are quotients of leaves of form $X\times \{b\}$, for $b\in \partial X$. The center-stable holonomy on the cover $X\times \partial X$ of $(X\times\partial X)/\what{\rho}$ is the map preserving the second component in $\partial X$. There is also a smooth map from each stable leave in $X\times \partial X$ into $\partial X$ by projecting to the second factor. This map is invariant under center-stable holonomy and $\psi_t^{\rho}$. It follows that $df^{-1}\of q(E^u_{slow})$ projects to a well-defined and smooth distribution on $\partial X$. Consequently, we get a distribution on $\partial X$ that is $C^1$-close to $E$ and invariant under $\rho(\Gamma)$. Inverses of projections from the boundary $\partial X$ onto unstable leaves maps $E$ to a sub-distribution of $E^u$ that is $C^1$-close to $df^{-1}\of q(E^u_{slow})$ and is invariant under the flow $\psi_t^{\rho}$. Thus these two distributions coincide. It follows that $df^{-1}\of q(E^u_{slow})$ projects to $E$ under the projection onto the second factor $\partial X$. Since $E$ is $C^\infty$ and is $C^1$-close to $E_0$, we have that $E^u_{slow}$ is horizontal and generic everywhere. By Theorem \ref{thm:Mitchell}, we therefore have a tangent cone at each point of the leaves of $V^{cs}$ and this is a perturbation of the quaternionic or octonionic Heisenberg group by Corollary \ref{cor:uniqueclass}. The corresponding nilpotent group is asymmetric by Corollary \ref{cor:stable-asymmetry}. Finally by Proposition \ref{prop:eval_products}, the Lyapunov exponents all have ratio 2.
	
	By \cite[Theorem 3.6]{Butler:17}, the flow $\psi_t$ is $C^\infty$ orbit equivalent to the flow $\phi_t$ by a diffeomorphism $G:SM\to SM$. We lift $G$ to the universal covers and conjugate by $\what{q}$ to obtain a diffeomorphism $\what{G}:X\times \partial X \to X\times \partial X$ which intertwines $(q^{-1})^*\phi_t$ and $\psi_t^\rho$ as well as their corresponding center-stable foliations whose leaves are of the form $X\times \{b\}$ for $b\in \partial X$. Thus $\what{G}$ induces a well-defined diffeomorphism from $\partial X$ to $\partial X$ that intertwines the $\rho_0$ and $\rho$ actions. In other words, the perturbed $\Gamma$-action is smoothly conjugate to the original $\Gamma$-action.

\end{proof}

\section{Local Hyperbolic Rank Rigidity}\label{sec:hyprank}

In this section we prove the local rigidity results of Theorems \ref{thm:local-hyp}, \ref{thm:complexcase} and \ref{thm:boundary_rigid}. We first show in Section \ref{sec:holonomy} that perturbations of locally symmetric metrics with higher hyperbolic rank have the same hyperbolic rank as the locally symmetric metric. This is a key step for Theorem \ref{thm:local-hyp}, whose proof we provide in Section \ref{subsec:proof-local-hyp}.

We recall some facts from \cite{CNS2018} about hyperbolic rank here. 
\begin{lemma}\label{lem:facts}
	Suppose $M$ is a closed Riemannian manifold with sectional curvatures $0>\kappa \ge -1$ and of hyperbolic rank $k$. Let $\what{\mc{E}}(v)\subset v^\perp$ be the subspace consisting of initial vectors $w$ of unstable Jacobi fields of the form $t\mapsto e^t\|_{\varphi^t v}w$.
\begin{enumerate}
	\item $\what{\mc{E}}(v)$ is of dimension at least $k$ for every $v$ --- by definition.
	\item The set $\what{\mc{R}}\subset SM$ where $\what{\mc{E}}(v)$ is of constant dimension $k$ is open and dense --- by Lemma 2.4 of \cite{CNS2018}.
	\item The map $v\mapsto \what{\mc{E}}(v)$ is $C^\infty$ on $\what{\mc{R}}$ --- by Proposition 3.3 of \cite{CNS2018}.
	\item The lift $E_1^u(v)$ of $\what{\mc{E}}(v)$ to $ T_vW^u(v)=E^u(v)\subset TSM$ is a distribution on $\what{\mc{R}}$ which is uniformly $C^\infty$ along $W^u$ restricted to $\what{\mc{R}}$ which consists precisely of the Lyapunov spaces in $E^u$ of exponent $1$ --- by Lemma 4.2 of \cite{CNS2018}. (We will sometimes refer to the $E^u_1(v)$ space as $E^u_{fast}(v)$.) 
	\item Denote by $\what{\mc{E}}(v)^\perp$ the orthogonal complement of $\what{\mc{E}}(v)$ within $v^\perp$ with respect to the Sasaki metric. The lift $E_{<1}^u$ of $\what{\mc{E}}^\perp$ to $E^u$ is the direct sum of unstable Lyapunov spaces of exponents different from $1$ --- by Lemma 4.3 of \cite{CNS2018}. (We will sometimes refer to the $E^u_{<1}(v)$ space as $E^u_{slow}(v)$.) 
\end{enumerate} 

\end{lemma}

We remark that $E_{<1}^u$ is uniformly $C^\infty$ along $W^u$ restricted to $\what{\mc{R}}$ since it is orthogonal to $E^u_1$ in $E^u$.

\subsection{Holonomy Groups of Perturbations and Hyperbolic Rank} \label{sec:holonomy}
The goal of this subsection is to prove the following proposition and its Corollary \ref{cor:perturb-rank}.

\begin{proposition}\label{prop:perturb-rank}
Let $(M, g_0)$ be a closed complex, quaternionic or octonionic hyperbolic locally symmetric manifold. Then there is an open $C^2$ neighborhood $U$ of $g_0$ such that for any $g\in U$, if $(M,g)$ has higher hyperbolic rank and $\kappa_g\ge -1$ then $(M,g)$ has hyperbolic rank at least $1$, $3$ or $7$ respectively.
\end{proposition}

The proof of this proposition will rely on an analysis of the Brin-Pesin asymptotic holonomy group. Before giving the proof we will establish some notation and several lemmas.

Let $FM\to SM$ denote the full othonormal frame bundle which is a (right) principal $\SO(n-1)$ bundle over $SM$. Let $F_2M$ denote the $2$-frame bundle over $SM$ which is a fiberwise quotient of $FM$ by $\SO(n-2)$ acting on the right, i.e. we have bundle quotient maps 
\[
\begin{tikzcd}
FM \arrow{rd}[swap]{\pi_{n\to 1}} \arrow[r, "\pi_{n\to 2}"]  & F_2M \rlap{$ =FM/\SO(n-2)$} \arrow[d, "\pi_{2\to 1}"] \\
& SM \rlap{$=FM/\SO(n-1).$}
\end{tikzcd}
\]

We endow $FM$ and $F_2M$ with the natural extensions of the Liouville measure by the Haar measures on the fiber. These measures are invariant under their respective frame flows. The $k$-frame flows commute with the corresponding right action of $\SO(n-1)$ on $F_kM$. Consider the ergodic component $\mc{E}(v,w)$ of the 2-frame flow containing $(v,w)$ with respect to this measure. Moreover, assume $(v,w)$ is a generic  2-frame in the ergodic component $\mc{E}(v,w)$. (This means that time averages over the orbit of $(v,w)$ coincide with the space average over $\mc{E}(v,w)$ with respect to its ergodic measure.) Without loss of generality we may assume $(v,w)$ is the first two vectors of a full frame $f$ which is also generic in its ergodic component $\mc{E}(f)\subset FM$ for the full frame flow. Lemmas 5.1 (see the formulation of the ergodic component $Q(\omega)$ in the proof) and 5.2 of Brin \cite{Brin82} show that
\[
\mc{E}(f)=\overline{\cup_{t\in\R}\Phi_t(C(f))}
\]
where $C(f)$ is the smallest set saturated by entire leaves of $W^s_f,W^u_f$ and containing $f$. Moreover, there is a compact subgroup $B^\infty_f< \SO(n-1)$, called the {\em ergodic component group}, satisfying
\[
\mc{E}(f)\cap \pi^{-1}_{n\to 1}(v)=f\cdot B^\infty_f.
\]
This immediately implies that,
\[
\mc{E}(v,w)\cap \pi^{-1}_{2\to 1}(v)=\pi_{n\to 2}(\mc{E}(f))\cap \pi^{-1}_{2\to 1}(v)=\pi_{n\to 2}(f\cdot B^\infty_f)\cap \pi^{-1}_{2\to 1}(v)=f\cdot B^\infty_f \SO(n-2).
\]
It is clear that if $f'=f\cdot h$ for $h\in \SO(n-1)$ then $B^\infty_{f'} = h^{-1}B^\infty_fh$. And by definition if $f$ and $f'$ are in the same ergodic component, then $B^\infty_f=B^\infty_{f'}$. We call $B^\infty_f$ the {\em (Brin-Pesin) asymptotic holonomy group}, and denote its conjugacy class in $SO(n-1)$ by $B^\infty$.

Suppose $v\in SM$ and $v'\in W^s_{\phi_t}(v)$. Fix $k\in \set{2,\dots,n}$ and let $p\left(v, v^{\prime}\right)$ be the map from the fiber of the $k$-frame bundle, $F_{k} M$, over $v$ to the fiber over $v^{\prime}$ that takes each frame $f$ to $p\left(v, v^{\prime}\right)(f)=\pi^{-1}_{k\to 1}\left(v^{\prime}\right) \cap W_{\Phi_t}^{s}(f).$ The map $p\left(v, v^{\prime}\right)$ corresponds
to a unique isometry between $v^{\perp}$ and $v^{\prime \perp}$ which commutes with the right action of $\SO(n-1)$. We will mainly consider the case of the map for $k=2$. After lifting to the universal cover, one can think of $p\left(v, v^{\prime}\right)(f)$ as the result of ``parallel transporting'' $f$ along $\gamma_{v}(t)$ out to the boundary at infinity of $\tilde{M}$ and then back to $v^{\prime}$ along $\gamma_{v^{\prime}}(t)$. 

Similarly, when $v^{\prime}$ and $v$ belong to the same leaf of $W_{\phi_t}^{u}$, there is a map $p(v,v')(f)=\pi^{-1}_{k\to 1}\left(v^{\prime}\right) \cap W_{\Phi_t}^{u}(f)$. Following Brin (see Definition 4.4 of \cite{Brin82}) we define the transitivity group at $v$ as follows:

\begin{definition}\label{def:transitivity_group}
	Given any sequence $\sigma=\left\{v_{0}, v_{1}, \ldots, v_{k}\right\}$ with $v_{0}=v, v_{k}=\phi_{T}(v)$ for some $T\in \R$ such that each pair $\left\{v_{i}, v_{i+1}\right\}$ lies on the same leaf of $W_{\phi_t}^{s}$ or $W_{\phi_t}^{u}$ we have an isometry of $v^{\perp}$ given by
\[
I(\sigma,T)=\Phi_{-T} \circ \prod_{i=0}^{r} p\left(v_{i}, v_{i+1}\right)
\]
The closure of the group generated by all such isometries with $v_0=v$ is denoted by $H_{v}<\Isom(v^\perp)$ and is called the {\em transitivity group}. 
\end{definition}

After picking an orthonormal full frame $f$ with $v$ as its first vector, $H_{v}$ can be identified as a subgroup of $\SO(n-1)$. This identification depends on the choice of frame $f$. Different choices give conjugate identifications. It is well-known that the ergodic component group $B^\infty_f<\SO(n-1)$ at a frame $f$ over $v$ coincides with the transitivity group $H_v$ under this identification (see the remark before Lemma 5.2, or Remark 2 of \cite{Brin75a} and Proposition 2 of \cite{Brin75b}). For clarity, we provide the details of this equivalence below.
\begin{lemma}\label{lem:erg=hol}
 For each $k\in\set{2,\dots,n}$, the ergodic component group coincides with the transitivity group at each $v\in SM$ and $f\in \pi_{k\to 1}^{-1}(v)$, i.e. $B^\infty_f=H_v$, using the identification provided by the frame $f$.

\end{lemma}
\begin{proof}
	Fix $v\in SM$, and consider any continuous function $\alpha$ on the homogeneous space $\pi_{k\to 1}^{-1}(v)/H_v$. Lift $\alpha$ to a function $\bar{\alpha}:\pi_{k\to 1}^{-1}(v)\to \R$. Using the maps $p(v,v')$ for $v\in SM$ we obtain a continuous extension $\what{\alpha}: F_kM\to \R$ given by $\what{\alpha}(f')=\bar{\alpha}(p(v',v)(f'))$ for any $f'\in \pi_{k\to 1}^{-1}(v')$. Observe that $\what{\alpha}$ is well-defined and invariant under both the action of $H_v$ and the frame flow. Since continuous functions on $\pi_{k\to 1}^{-1}(v)/H_v$ separate points, and continuous flow invariant functions on $F_kM$ are constant on ergodic components, we obtain a surjection from the space of ergodic components to $\pi_{k\to 1}^{-1}(v)/H_v$. Therefore, $B^\infty_f< H_v$ for any $f\in \pi_{k\to 1}^{-1}(v)$. On the other hand, as explained in Remark 2 of \cite{Brin75a}, we also have $H_v< B^{\infty}_f$.
\end{proof}

\begin{remark}\label{rem:Brinhyp}
Note that the entire formulation of the ergodic component groups, asymptotic holonomy groups and transitivity groups makes sense in the broader context of compact (isometric) group extensions of Anosov flows. Moreover the above lemma also holds (see \cite{Brin75b}).
\end{remark}

For the next lemma, recall that an $\eps$-net $F\subset X$ of a metric space $X$ is a subset whose points are pairwise separated with distance at least $\eps$, and such that every point $x\in X$ is at distance at most $\eps$ to a point of $F$.

\begin{lemma}\label{lem:group-net}
For any connected semisimple Lie subgroup $K<SO(n-1)$, let $F\subset K$ be a $\delta$-net of $K$ that generates a dense subgroup of $K$. There is a sufficiently small $\eps_0>0$ such that for any $\eps<\eps_0$ every $\eps$-perturbation of $F$ generates a subgroup of $SO(n-1)$ whose closure contains a conjugate $h_\eps K h_\eps^{-1}$ of $K$. Moreover, $h_\eps$ can be chosen so that $h_\eps\to e$ as $\eps\to 0$.
\end{lemma}

\begin{proof}
First suppose there is no such $\eps_0$, then there is a sequence $\eps_n\to 0$ and an $\eps_n$-perturbation $F_{\eps_n}\subset \SO(n-1)$ of $F$, with respect to the Hausdorff metric on closed subsets, such that the closure of the group generated by $F_{\eps_n}$ does not contain a conjugate of $K$. Let $H_n=\overline{\inner{F_{\eps_n}}}^0$ be the identity component of the closure of the group generated by $F_{\eps_n}$. Consider the Levi decomposition of the connected Lie group $H_n$, namely $H_n=H_n^{ss} H_n^{a}$ where $H_n^{ss}$ is maximal semi-simple and $H_n^{a}$ is a closed connected solvable normal subgroup (see Theorem 3.18.13 of \cite{Varadarajan} and 1.4.3 of \cite{Gorbatsevich:94}). Since $H_n^a$ is compact connected and solvable, it is an abelian torus. The automorphism group of this torus is discrete. Since $H_n$ is connected the image of the conjugation action $H_n\to \Aut(H_n^a)$ is trivial. Note also that $H^{ss}\cap H^a$ is semisimple and therefore trivial. This implies that $H_n$ is a group product of $H_n^{ss}$ and $H_n^a$. 

Since the topology induced by the Hausdorff metric on compact subsets of a compact space is compact, we may pass to a subsequence such that $H_n$ converges to $H=\lim_{n} H_n$. Let $H^{ss}=\lim_{n} H^{ss}_n$ and $H^a=\lim_{n}H^a_n$. The limit of abelian groups is abelian so $H^a$ is abelian. Hence $H=H^{ss} H^a$ since the limit of a product is a product. Since the generators converge, $H$ therefore contains $\overline{\inner{F}}^0$, but this is just $K$. The projection of $K$ to $H^a$ is trivial since $K$ is semisimple and $H^a$ is connected, and hence $H^{ss}$ contains $K$. 

There are only finitely many conjugacy classes of connected semi-simple subgroups of $\SO(n-1)$. For each $n$ the corresponding connected semi-simple group $H_n^{ss}$ must belong to one of these. Since $H^{ss}$ contains $K$, the groups $H^{ss}_n$ must eventually belong to the same conjugacy class, also containing $H^{ss}$. Therefore for all sufficiently large $n$, the group $H^{ss}_n$ must contain a conjugate of $K$. This contradiction establishes the existence of $\eps_0$.

Moreover, we have shown that for every $\eps<\eps_0$, if $F_\eps$ is an $\eps$-perturbation of $F$ generating a subgroup with closure $H_{\eps}<\SO(n-1)$, then  $H_{\eps}^{ss}={h_\eps} Q_\eps h_\eps^{-1}$ for a compact semisimple group $Q_\eps$ containing $K$ and belonging to one of finitely many conjugacy classes. In particular $h_\eps K h_\eps^{-1}< H_\eps^{ss}$. Moreover $Q_\eps$ is close to $H^{ss}_\eps$ as $\eps\to 0$, and so we may assume we have chosen $h_\eps$ so that  $h_\eps\to e$ as $\eps\to 0$.
\end{proof}

Given two Riemannian metrics $g, g' \in \operatorname{Sym}\left(T^{*}M \otimes T^{*}M\right)$ there is a unique positive-definite vector bundle endomorphism $A$ of $TM$ such that $g'_x(v, w)=g_x(A_x v, w)$ at each point $x\in M$. Let $q:TM\to TM$ be the bundle endomorphism $q:=A^{-\frac12}$. If $f$ is a g-orthonormal frame then $q(f)$ is an $g'$-orthonormal frame. Thus $q$ defines a smooth $\SO(n)$-equivariant map of $g$-orthonormal frames to $g'$-orthonormal frames. In other words $q$ induces a principal $\SO(n)$-bundle isomorphism $q^F:(FM)_g\to (FM)_{g'}$ (see Proposition 1 of \cite{Bourguignon-Gauduchon}).

Now consider a $C^2$ perturbation $g_\eps$ through $C^\infty$ metrics of a negatively curved metric $g_0$. As indicated above, there is a vector bundle isomorphism $q_\eps:TM\to TM$ which restricts to a fiber bundle isomorphism $q_\eps:(SM)_\eps\to (SM)_0$ between the unit tangent bundle $(SM)_\eps\to M$ for the $g_\eps$ metric and $(SM)_0\to M$ for the metric $g_0$. Moreover, we have a principal $\SO(n)$-bundle isomorphism $q_\eps^F:(FM)_\eps\to (FM)_0$ between the orthonormal frame bundles $(FM)_\eps\to (SM)_\eps$ and $(FM)_0\to (SM)_0$ which fibers over $q_\eps$. We may conjugate the geodesic flow $\phi^\eps_t$ for $g_\eps$ by $q_\eps$ to obtain a flow $\what{\phi}^\eps_t:(SM)_0\to (SM)_0$. Similarly, we may conjugate the frame flow $\Phi^\eps_t$ for $g_\eps$ by $q_\eps^F$ to obtain a flow $\what\Phi^\eps_t:FM\to FM$ which extends $\what\phi^\eps_t$. 

Using Remark \ref{rem:Brinhyp}, we may consider the transitivity group $H_v^\eps<\Isom(v^\perp)$ for the flow $\what{\Phi}^\eps_t$. 
Also, since the map $q_\eps^F$ is equivariant with respect to the right $\SO(n)$ action, the asymptotic holonomy group for $\Phi^\eps_t$ is the same as $B^\infty_\eps$ up to conjugacy.

\begin{lemma}\label{lem:perturb_grow}
Let $g_\eps$ be a $C^2$ perturbation through $C^\infty$ metrics of a negatively curved metric $g_0$. Let $H_v^\eps$ be the transitivity group of a vector $v$ for the metric $g_\eps$, and suppose transitivity group $H_v^0$ for $g_0$ is connected and semisimple. Then for all sufficiently small $\eps$, $H_v^\eps$ contains a conjugate of $H_v^0$ which limits to $H_v^0$ as $\eps\to 0$.
\end{lemma}

\begin{proof}

In order to show there is a relation between $H_v^\eps$ for different $\eps$, we use the description of the transitivity group. Namely, we may construct a net of elements in $H_v^\eps$ by following a sequence of paths along leaves of $W^u_{\what\Phi_t}$ and $W^s_{\what\Phi_t}$ starting from $v$ which lie in a compact set. By the Permanence Theorem \cite[Theorem 6.8]{HirschPughShub}, the unstable and stable manifolds for $\what\Phi^\eps_t$, as a compact extension of $\what\phi^\eps_t$ on $(SM)_0$, vary continuously in the $C^1$ topology under $C^2$ perturbations. As these manifolds converge on compacta, using sufficiently long $u-s$ paths, we obtain for any $\eps>0$ an $\eps$-net in $H_v^\eps$ that is $\eps$-close to an $\eps$-net in $H_v^0$.

By assumption, the transitivity group $H_v^0$ is connected and semisimple. Consequently, by Lemma \ref{lem:group-net}, the transitivity group $H_v^\eps$ contains a conjugate of $H_v^0$ for sufficiently small perturbation parameter $\eps>0$, and these subgroups converge to $H_v^0$ as $\eps\to 0$.
\end{proof}

The following lemma can be found in \cite{Constantine08}; we provide a short proof for convenience. We call a 2-frame $(u,w)$ \emph{higher hyperbolic rank} if the sectional curvatures satisfy  $\kappa(\phi_t(u)\wedge\|_t w)\equiv -1$ for all $t\in \R$ where $\|_t$ is parallel translation along $\phi_t(u)$.

\begin{lemma}\label{lem:preserve_rank}
For a flow-invariant, full Liouville-measure subset of $v\in SM$, $H_v$ sends higher hyperbolic rank 2-frames over $v$ to higher hyperbolic rank 2-frames over $v$.
\end{lemma}

\begin{proof}
The set of 2-frame $(u,w)$ of higher hyperbolic rank remains invariant under parallel transport. Any 2-frame $(u',w')$ forward asymptotic to the parallel frame field of $(u,w)$ must therefore also have curvature limiting to $-1$. If $(u',w')$ is recurrent then $(u',w')$ must have higher hyperbolic rank. If $(u'',w'')$ forward limiting to the parallel frame field of $(-u',w')$ is recurrent then it must similarly have higher hyperbolic rank. Lastly if the frame $(-u,\rho w)$ for $\rho\in H_u<\Isom(u^\perp)$ which forward limits to $(u'',w'')$ is recurrent then it has higher hyperbolic rank as well. Finally, the set of bi-recurrent vectors is a flow-invariant full Liouville measure set.
\end{proof}

The following lemma is trivial, but we include it for completeness.
\begin{lemma}\label{lem:stabs_converge}
If a sequence of subgroups $K_n<\SO(n)$ converges to $K < \SO(n)$, and a sequence of vectors $v_n$ converges to a vector $v$ then the stabilizers in $K_n$ of vectors $v_n$, $\op{Stab}_{\SO(n)}(v_n)\cap K_n$, converge to a subgroup of $\op{Stab}_{\SO(n)}(v)\cap K$. 
\end{lemma}

\begin{proof}
	If $k_n\in \op{Stab}_{\SO(n)}(v_n)\cap K_n$ converges to $k\in K$, then clearly $k\in \op{Stab}_{\SO(n)}(v)$.
\end{proof}

Before proving Proposition \ref{prop:perturb-rank}, we describe the Brin-Pesin groups and their actions in the locally symmetric case. 

Let $(M, g_0)$ be a closed complex, quaternionic or octonionic hyperbolic locally symmetric manifold. Hence $M=\Gamma\bs \til{M}$ for a cocompact lattice $\Gamma<G:=\op{Isom}^0(\til{M},g_0)$ and let $K_x<G$ be the stabilizer of a point $x$ and $L_v<K_x$ the stabilizer of a tangent vector $v\in T_xM$. Observe that the left action of $G$ on the full (oriented) orthonormal frame bundle decomposes $(F\til{M})_{g_0}$ into $G$-orbits, which descend to closed subsets of $(FM)_{g_0}$ since $G$ acts cocompactly on $(F\til{M})_{g_0}$. Since the frame flow corresponds to a right action by the 1-parameter split Cartan subgroup, by Moore's Theorem (Theorem 2.2.6 of \cite{Zimmer}) the quotient of the $G$ orbits by $\Gamma$ coincide with the ergodic components of the frame flow on $(FM)_{g_0}$ for the lift of the Liouville measure. Since the intersection of the fiber over $v$ with a $G$ orbit coincides with the action of the stabilizer $L_v$, which is faithful on frames, we have $H_v=L_v$. 

For the quaternionic hyperbolic space $\til{M}=\KH_{\KH}^n$ we have $H_v=L_v\cong B^\infty=\Sp(n-1)\Sp(1)$
(here $\Sp(k):=\Sp(2k,\C)\cap U(2k)$), and for the Cayley plane $\til{M}=\KH_{\KO}^2$ we have $H_v=L_v\cong B^\infty=\Spin(8)\Spin(1)$.
Quaternionic symmetric space comes endowed with  a quaternionic K\"ahler structure and the Cayley plane with an octonionic K\"ahler structure (see e.g. \cite{Besse87}). These structures are parallel and thus preserved under the full holonomy group. Thus each higher hyperbolic rank 2-frame is fixed by $\Sp(n-1)$ and $\Spin(8)$ respectively. However, the set of all such frames is acted on transitively by the $\Sp(1)$ and $\Spin(1)$ groups of dimension 3 and 7 respectively since these groups act transitively on the perpendicular to $v$ in the $\KH$-line (resp. $\KO$-line) through $v$.

\begin{proof}[Proof of Proposition \ref{prop:perturb-rank}]
By the hyperbolic rank condition, the statement for the complex hyperbolic case is trivial. We proceed by contradiction and suppose there is a sequence of metrics $g_{\eps_n}$ of higher hyperbolic rank with $\norm{g_{\eps_n}-g_0}_{C^2}<\eps_n$ where $\eps_n$ tends to $0$, but whose hyperbolic ranks are all less than the hyperbolic rank of $g_0$.

By Lemma \ref{lem:preserve_rank}, for almost every $v\in SM$, $H_{\eps_n,v}$ preserves the hyperbolic rank 2-frames over $v$. Consider a sequence of generic 2-frames $(v_n,w_n)$ for $g_{\eps_n}$ that limit to a hyperbolic 2-frame $(v,w)$ for $g_0$. By Lemma \ref{lem:perturb_grow}, if $\eps_n$ is sufficiently small, the asymptotic holonomy group $H_{\eps_n,v_n}$ for the perturbed metric $g_{\eps_n}$ contains a conjugate copy $L_{\eps_n}$ of $H_{v}$, and these converge to $H_{v}$ as $\eps_n \to 0$.  Let $Q_{\eps_n}<L_{\eps_n}$ be the stabilizers in $L_{\eps_n}$ of $(v_n,w_n)$. By Lemma \ref{lem:stabs_converge}, $Q_{\eps_n}$ converges to a subgroup of a conjugate of $\Sp(n-1)$ or $\Spin(8)$ respectively. Hence if $\eps_n$ is sufficiently small and $(v_n,w_n)$ is chosen to be a hyperbolic 2-frame for $g_{\eps_n}$ then the $L_{\eps_n}$-orbit of $(v_n,w_n)$, consisting of hyperbolic 2-frames, has dimension at least as large as the dimension of the $H_{v}$-orbit of $(v,w)$. This contradicts the assumption that hyperbolic ranks of $g_{\eps_n}$ are less than that of $g_0$.
\end{proof}

\begin{corollary}\label{cor:perturb-rank}
Let $(M, g_0)$ be a closed complex, quaternionic or octonionic hyperbolic locally symmetric manifold. Then there is an open $C^2$ neighborhood $U$ of $g_0$ such that for any $g\in U$, if $(M,g)$ has higher hyperbolic rank and $\kappa_g\ge -1$ then the hyperbolic rank of every tangent vector to $(M,g)$ is exactly $1$, $3$ or $7$ respectively.
\end{corollary}

\begin{proof}
If $v$ is a unit tangent vector of $(M,g_0)$ then $v^{\perp_{g_0}}=\what{\mc{E}}_{g_0}(v)\oplus \what{\mc{E}}_{g_0}(v)^{\perp_{g_0}}$, where the sectional curvature $\kappa(v\wedge w)=-1$ for every $w\in \what{\mc{E}}_{g_0}(v)$ and $\kappa(v\wedge u)=-\frac{1}{4}$ for every $u\in \what{\mc{E}}_{g_0}(v)^{\perp_{g_0}}$.

Similarly, if $v$ is a unit tangent vector of $(M,g)$, we also have $v^{\perp_{g}}=\what{\mc{E}}_{g}(v)\oplus \what{\mc{E}}_{g}(v)^{\perp_{g}}$, where the sectional curvature $\kappa(v\wedge w)=-1$ for every $w\in \what{\mc{E}}_{g}(v)$. Since $g$ is sufficiently $C^2$-close to $g_0$, the continuity of curvatures in the $C^2$-topology of Riemannian metrics implies that $\dim(\what{\mc{E}}_{g}(v))\le \dim (\what{\mc{E}}_{g_0}(v))$ for every $v$. In particular, the hyperbolic rank of every tangent vector of $(M,g)$ is at most the hyperbolic rank of $(M,g_0)$.

On the other hand, by Proposition \ref{prop:perturb-rank}, the hyperbolic rank of $(M,g)$ is not smaller than the hyperbolic rank of $(M,g_0)$. By item (1) of Lemma \ref{lem:facts}, the hyperbolic rank of every tangent vector of $(M,g)$ is exactly the hyperbolic rank of $(M,g_0)$.
\end{proof}

\subsection{Quaternionic and octonionic hyperbolic local rigidity}\label{subsec:proof-local-hyp}
We are now ready to prove the main result about higher hyperbolic rank stated in the introduction.
\begin{reptheorem}{thm:local-hyp}
	Let $(M, g_0)$ be a closed quaternion or octonionic hyperbolic locally symmetric manifold. Then there is an open $C^3$ neighborhood $U$ of $g_0$ such that for any $g\in U$, if $(M,g)$ has higher hyperbolic rank and $\kappa_g\ge -1$ then $(M,g)$ is locally symmetric and isometric to $(M,g_0)$. 
\end{reptheorem}

	\begin{proof}[Proof of Theorem \ref{thm:local-hyp}]
	By Corollary \ref{cor:perturb-rank}, we may choose $U$ such that the hyperbolic rank of every tangent vector of $(M,g)$ is the same as the hyperbolic rank of $(M,g_0)$.
	
	Next, let $J_{\xi}(t)$ denote the unstable Jacobi field along $\phi_t(v)$ with initial vector $\xi\in \what{\mc{E}}(v)$. By Lemma \ref{lem:facts}, we can make the following observations about Oseledets spaces,	
	\[
	E_1^u(v)=\set{(\xi, J_{\xi}'(0)): \xi\in \what{\mc{E}}(v)} \quad\text{and}\quad E_{slow}^u(v)=\set{(\xi, J_{\xi}'(0)): \xi\in \what{\mc{E}}(v)^\perp}.
	\]
	By Corollary \ref{cor:perturb-rank}, $\dim E_1^u$ is constant and by Lemma \ref{lem:facts}, $E_1^u$ is a globally $C^\infty$ distribution on $SM$. In particular, $E_1^u$ is uniformly $C^\infty$ along $W^u$ for each such metric $g$. Moreover, as the curvature operator, and hence Jacobi fields, vary $C^1$ in the $C^3$ topology on smooth metrics, $E_1^u$ and its first derivatives along $W^u$ vary continuously in the metric parameter with respect to the $C^3$ topology on smooth metrics. By \cite[Lemma 4.3]{CNS2018}, $E_1^u$ and $E_{slow}^u$ are perpendicular with respect to the Sasaki metric for $g$. Hence the dominated splitting $E^u=E_1^u\oplus E_{slow}^u$ is uniformly $C^\infty$ along unstable leaves.
	By the Permanence Theorem \cite[Theorem 6.8]{HirschPughShub}, $E_{slow}^u$ has $C^1$ dependency on $g$ as $g$ varies in the $C^3$ topology.  Therefore by Theorem \ref{thm:locally_rigid}, $M$ is symmetric for sufficiently $C^3$ nearby smooth metrics.
	\end{proof}

\begin{remark}
One can show that $E^u_1$ varies continuously in the $C^1$-topology as the metric $g$ varies in the $C^2$-topology. However, since $E^u$ may not vary continuously in the $C^1$-topology as the metric $g$ varies in the $C^2$-topology, we are not able to conclude that $E^u_{slow}$ varies $C^1$ in this topology. The latter is necessary to obtain the isomorphism of tangent cones. Hence we require the metrics to be  $C^3$ close.
\end{remark}

We next observe that the higher hyperbolic rank condition in the last theorem may be replaced by a condition on Lyapunov exponents. 
\begin{corollary}\label{cor:lyapunovcase}
	In Theorem \ref{thm:local-hyp}, the conclusion still holds if the higher hyperbolic rank assumption is replaced with the assumption that some ergodic measure of full support has a Lyapunov exponent of $1$.
\end{corollary}

This follows from the following lemma, whose proof is essentially that of Corollary 1.4 in \cite{CNS2018}. (See also Theorem 1.3 of \cite{Connell03} for the case $\kappa \leq -1$.) 

\begin{lemma}
	Let $M$ be a closed Riemannian manifold with $\kappa\geq -1$. If, with respect to the geodesic flow on $SM$, some ergodic measure of full support has a Lyapunov exponent equal to $1$, then $M$ has higher hyperbolic rank.
\end{lemma}

\subsection{Measure of Maximal Entropy for Perturbations} \label{sec:complex-entropy}
For the case of complex hyperbolic metrics, we are not able to obtain full local rigidity. Indeed,  our methods require that the tangent cones of unstable leaves are asymmetric which the Heisenberg group is not. However, we still obtain equality of the Liouville measure and the Bowen-Margulis measure (the unique measure of maximal entropy) for a higher rank perturbation.  The precise result of this subsection is Theorem \ref{thm:complexcase} from the introduction.

\begin{reptheorem}{thm:complexcase}
Let $(M,g_0)$ be a closed complex hyperbolic manifold. There is an open neighborhood $U$ of $g_0$ in the $C^3$-topology among $C^\infty$ metrics such that if $g\in U$ and $(M,g)$ has higher hyperbolic rank and sectional curvature $\kappa\ge -1$ then the Liouville measure on $SM$ coincides with the (unique) measure of maximal entropy for the geodesic flow of $g$ on $SM$.
\end{reptheorem}
\begin{proof}
We first show that the tangent cone $TC_vW^u(v)$ exists for every $v\in SM$. 

By Corollary \ref{cor:perturb-rank}, $E_1^u\subset E^u$ is one dimensional everywhere. By (4) of Lemma \ref{lem:facts}, $E_1^u$ forms a smooth distribution along each leaf of the foliation $W^u$. By item (5) of Lemma \ref{lem:facts}, the lift $E_{slow}^u=E_{<1}^u$ of $\what{\mc{E}}^\perp$ to $E^u$ is the direct sum of unstable Lyapunov spaces of exponents different from $1$. Therefore, $E_{slow}^u$ is uniformly $C^\infty$ along $W^u$ since it is orthogonal to $E^u_1$ in $E^u$ with respect to the Sasaki metric.

By the same argument used in the proof of Theorem \ref{thm:local-hyp}, $E_{slow}^u$ has $C^1$ dependency on $g$ as $g$ varies in the $C^3$ topology. In particular, $E^u_{slow}$ is generic everywhere on $SM$. Also, by Theorem \ref{thm:Mitchell}, the tangent cone $TC_vW^u _g$ exists everywhere.

By Corollary \ref{cor:perturb-rank}, the horizontal distribution of $TC_vW^u_g(v)$ has codimension 1.  Hence $TC_vW^u_g(v)$ is a 2-step Carnot nilpotent Lie group with one dimensional center for every $v\in SM$. By Lemma \ref{lem:generic}, $TC_vW^u_g(v)$ depends continuously on the metric $g$ in the $C^3$ topology. Recall that the Lie bracket induces a  bilinear form  $[\cdot,\cdot]:\mf{n}_v^0\times \mf{n}_v^0\to \mf{n}_v^1$ between  the first and second levels $\mf{n}_v^0$ and $\mf{n}_v^1$ of the Lie algebra $\mf{n}_v$ of $TC_vW^u_g(v)$.  As it is non-degenerate for $g_0$, it will also be non-degenerate for $g$ sufficiently close to $g_0$. Consequently, as the center is one-dimensional and there is only one, up to isomorphism, non-degenerate skew symmetric bilinear form on $\mf{n}_v^0$ for every $v\in SM$, $TC_vW^u(v)$ is  isomorphic as a graded nilpotent group to the Heisenberg group $H^{2k-1}$ where $\dim M=2k$.

By Theorem \ref{thm:Hei-spec}, if $v\in SM$ is a periodic vector then the sum of the positive Lyapunov exponents (with multiplicity) at $v$ is $k$. By \cite[Theorem 1.4]{Kalinin11}, it follows that with respect to any ergodic invariant probability measure, the entropy of geodesic flow is bounded above by $k$. On the other hand, by Pesin's entropy formula, the entropy of the geodesic flow with respect to Liouville measure (rescaled to have total volume 1) is $k$. Therefore the (rescaled) Liouville measure is the measure of maximal entropy.
\end{proof}

\begin{remark}
	 
Local rigidity for higher hyperbolic rank metrics near the complex hyperbolic one would follow from  Katok's well-known Entropy  Conjecture stating that the Liouville measure has maximal entropy precisely for locally symmetric manifolds among closed negatively curved manifolds. Alternately, and possibly simpler,  it would  follow from strict convexity of the difference of topological and measure theoretic entropy near the complex hyperbolic metric. For real hyperbolic metrics, this was proven by Flaminio  (\cite[Theorem A]{Flaminio95}).
\end{remark}

\begin{remark}
	As in Corollary \ref{cor:lyapunovcase} we may replace the condition on higher hyperbolic rank in Theorem \ref{thm:complexcase} with the condition that there is a Lyapunov exponent $1$ for some invariant ergodic measure of full support.
\end{remark}

\appendix
\section{A Variation on Benveniste's Lemma}

In this appendix we present the proof of the following Proposition which may be considered by some to be folklore. However, we could not locate a precise reference for it in the literature. (For the case of $C^0$ semiconjugacy see \cite{BowdenMann19}.) In what follows, if $\Gamma$ acts by $C^k$ diffeomorphisms on a $C^\infty$ manifold $X$, and $\tau:\Gamma\to \op{Diff}^k(F)$ is a representation for a $C^\infty$ manifold $F$, then we denote by $\what{\tau}:\Gamma\to \op{Diff}^k(X\times F)$ the representation given by $\what{\tau}(\gamma)(x,f)=(\gamma x,\tau(\gamma)(f)).$

\begin{proposition}(cf. \cite[Lemma 5.2]{Benveniste00})\label{prop:Benven}
Suppose $\Gamma$ acts freely, properly discontinuously and cocompactly by $C^k$ diffeomorphisms on a $C^\infty$ manifold $X$ for $k\geq 1$. Let $F$ be a closed $C^\infty$ manifold. Let $\tau_0$ and $\tau$ be homomorphisms from $\Gamma\to \op{Diff}^k(F)$ such that $\tau$ be $C^\ell$-close to $\tau_0$ for $\ell\geq 1$. Then there is $C^\infty$ diffeomorphism between suspensions $f:(X\times F)/\what{\tau}\to (X\times F)/\what{\tau}_0$. Moreover, the push-forward by $f$ of the horizontal foliation $\{X\times\{z\}: z\in F\}$ in $(X\times F)/\what{\tau}$ is $C^\ell$ close to the horizontal foliation in $(X\times F)/\what{\tau_0}$.
\end{proposition}

Here the map $f$ is $C^\infty$ with respect to the unique compatible $C^\infty$ structures on $(X\times F)/\what{\tau}$ and $(X\times F)/\what{\tau}_0$ compatible with the respective $C^k$ structures. The above result was proved by Benveniste in the $C^{\infty}$-case in \cite{Benveniste00} using the tame Fr\'{e}chet group structure of $\Diff ^{\infty} (F)$ and the Hamilton-Nash-Moser implicit function theorem. We follow his argument closely. However, even though the Banach manifolds arising in the $C^k$ setting are still tame Fr\'{e}chet manifolds, $\Diff^k(F)$ is not a tame Fr\'{e}chet Lie group  and we must instead use   implicit function theorems for  Banach manifolds, in Lemma \ref{lem:implicit}.

\begin{lemma}\label{lem:implicit}
    Let $\mc X, \mc Y$, and $\mc Z$ be Banach manifolds, and  $H:\mc X\to \mc Y$ and $G:\mc Y\to \mc Z$ be $C^1$ maps such that $G\circ H(x)=z_0$ for every $x\in \mc X$. Suppose that $P:T_{H(x_0)}\mc Y\to T_{x_0}\mc X$ and $Q:T_{z_0}\mc Z\to T_{H(x_0)}\mc Y$ are bounded linear transformation such that $D_{x_0}H\circ P+Q\circ D_{H(x_0)}G=\id$. Then there exist neighborhoods $U$ and $V$ of $x_0$ and $H(x_0)$ respectively such that $G^{-1}(z_0)\cap V \subset H(U)$. 
\end{lemma}

\begin{proof}
Using local charts, we can assume that $\mc X, \mc Y$, $\mc Z$ are Banach spaces and $x_0, H(x_0), z_0$ are zero vectors in the Banach spaces. Since $G\circ H =0$ we have $ \im D_0H\subset \ker D_0G$. On the other hand if there is $v\in \ker D_0G- \im D_0H$, then $v=\id(v)=D_0H\circ P(v)+Q\circ D_0G (v)\in \im D_0H$, which is not possible. Thus, $\im D_0H= \ker D_0G$. Now the conclusion follows immediately from \cite[Theorem 2.1]{AnNeeb09}.
\end{proof}

Given  closed manifolds $M$ and $N$, then $\Diff^k(M,N)$ is a Banach manifold for every $k\in \N$ (\cite{Wittmann19}). We note that by Banach manifold, we always mean a Banach manifold of class $C^\infty$.

We also need a discussion about the space of mappings. Let $E$ be a Banach space and let $U$ be a domain in $\mathbb R^n$ or in a smooth manifold such that $U$ has smooth boundary. For $\ell\in \N$, we let $C^\ell(U,E)$ denote the space of $C^\ell$ functions from the closure of $U$ to $E$. Then $C^\ell(U,E)$ is a Banach space. By \cite[Theorem 5.1]{Eliasson67} (with $\mf{S}=C^k$ and $s=\infty$) and the remark before Theorem 5.2 of \cite{Eliasson67}, if $D$ is a smooth Banach manifold admitting a $C^\infty$ connection, then $C^\ell(U,D)$ is a smooth Banach manifold. This condition is satisfied for $D=\Diff^k(F)$. Indeed the latter is open in $C^k(F,F)$ for $k\geq 1$ by the inverse function theorem and $C^k(F,F)$ satisfies the condition by the discussion at the top of p.170 in \cite{Eliasson67}. In particular $C^\ell(U, \Diff^k(F))$ is a smooth Banach manifold. Moreover, by \cite[Theorem 5.2]{Eliasson67}, the tangent bundle to $C^\ell(U,D)$ is naturally isomorphic to $C^\ell(U,TD)$.
We note that if $h\in C^\ell(U, \Diff^k(F))$ then the map $U\times F\ni(u,f)\mapsto (u,h(u)(f))\in U\times F$ is $C^{\min\{\ell,k\}}$.

\begin{lemma}\label{lem:Benven1}
Let $F\to P\to M$ be a fiber bundle, where $F$ and $M$ are closed manifolds. Let $\{U_1,\dots, U_k\}$ be a cover of $M$ so that $P|_{U_\alpha}$ is trivializable for each $\alpha$. Let $\{\phi_\alpha\}$ be trivializations, and let $\{\phi_{\alpha\beta}\}$ be the corresponding transition functions. Let $\{\phi_{\alpha\beta}'\}$ be a collection of transition functions which are close to the $\{\phi_{\alpha\beta}\}$ in the $C^\ell$-topology. Let $V_\alpha\subset U_\alpha$ be compactly contained open sets such that $\{V_\alpha\}$ cover M. Then there are maps $\{h_\alpha :
V_\alpha \to \Diff^k(F)\}$, $C^\ell$-close to the constant map to the identity, such that $h_\alpha(u)\phi_{\alpha\beta}(u) = \phi_{\alpha\beta}'(u)h_\beta(u)$ for all $u\in V_\alpha\cap V_\beta$, and all $\alpha$ and $\beta$.
\end{lemma}

\begin{proof}
For each $\alpha$, we let $W_\alpha$ be an open set with a smooth boundary such that $V_\alpha\subset\subset W_\alpha\subset\subset U_\alpha$. Inductively, we choose $W_{\alpha_0\alpha_1\dots\alpha_i}$ with smooth boundary such that $V_{\alpha_0}\cap V_{\alpha_1}\cap\dots\cap V_{\alpha_i}\subset\subset W_{\alpha_0\dots\alpha_i}\subset\subset W_{\alpha_0\dots\alpha_{i-}}\cap U_{\alpha_i}$. For each $i=0,1\dots$, we let $\mc C^i=\oplus_{\alpha_0\neq\alpha_1\neq\dots\neq\alpha_i}C^\ell(W_{\alpha_0\dots\alpha_i},\Diff^k(F))$. Each of $\mc C^i$ is a Banach manifold. Hence we have,
\[
T\mc{C}^i=\oplus_{\alpha_0\neq\alpha_1\neq\dots\neq\alpha_i}C^\ell(W_{\alpha_0\dots\alpha_i},T\Diff^k(F)).
\]

We define the following coboundary maps:
$$\Delta^0:\mc C^0\to \mc C^1$$
by $(\Delta^0(u))_{\alpha\beta}=u_\alpha\phi_{\alpha\beta}u_\beta^{-1}$ for every $u\in \mc C^0$; and
$$\Delta^1:\mc C^1\to \mc C^2$$
by $(\Delta^1(v))_{\alpha\beta\gamma}=v_{\alpha\beta}v_{\beta\gamma}v_{\alpha\gamma}^{-1}$ for every $v\in \mc C^1$. We note that since $\{\phi_{\alpha\beta}\}$ and $\{\phi_{\alpha\beta}'\}$ are transition maps of fiber bundles, $\Delta^1(\phi)=\Delta^1(\phi')=\id$, where $\id$ here denotes the constant map to the identity of $\Diff^k(F)$. Moreover, we also have $\Delta^0(\id)=\phi$.
To see that $\Delta^i$ are $C^\ell$ maps, we may compute the derivatives:

$$(D_u\Delta^0(\xi))_{\alpha\beta}=\xi_\alpha(\phi_{\alpha\beta}u_\beta^{-1})-(u_\alpha\phi_{\alpha\beta})\xi_{\beta},$$
for every $\xi\in T_u\mc C^0$, and
$$D_v\Delta^1(\zeta)_{\alpha\beta\gamma}=\zeta_{\alpha\beta}(v_{\beta\gamma}v_{\alpha\gamma}^{-1})+v_{\alpha\beta}(\zeta_{\beta\gamma}v_{\alpha\gamma}^{-1})-(v_{\alpha\beta}v_{\beta\gamma})\zeta_{\alpha\gamma},$$
for every $\zeta\in T_1\mc C^1$, where here we have used the following conventions: 
if $u: Q \rightarrow Q$ and $v: Q \rightarrow Q$ are smooth maps and $\xi$ is a section of $v^{*} T Q$, we write $u \xi$ for the section of $(u v)^{*} T Q$ defined by $u \xi(x)=d u_{x}\left(\xi_{x}\right)$ and write $\xi u$ for the section of  $(u v)^{*} T Q$ given by $\xi u(x)=\xi_{u(x)}.$ 

From the formulas, we see $D_u\Delta^0$ and $D_v\Delta^1$ are bounded and continuous in $u\in \mc{C}^0$ and $v\in \mc{C}^1$ respectively.

Let $\left\{\lambda_{\alpha}\right\}$ be a partition of unity subordinate to the cover $\left\{V_{\alpha}\right\}$. We define maps $\kappa^0:T_\phi\mc C^1\to T_{\id}\mc C^0$ and $\kappa^1:T_\phi\mc C^2\to T_{\id}\mc C^1$  by the formula:
\[
\left(\kappa^{0}(\zeta)\right)_{\alpha}=\sum_{\gamma} \lambda_{\gamma} \zeta_{\alpha \beta} \phi_{\alpha \gamma}^{-1}
\]
and
\[
\left(\kappa^{1}(\sigma)\right)_{\alpha \beta}=\sum_{\gamma} \lambda_{\gamma} \sigma_{\alpha \beta \gamma} \phi_{\alpha \beta}.
\]
It follows that $\kappa^0$ and $\kappa^1$ are bounded.

Following the  proof given on page 523 of \cite{Benveniste00} verbatim, these satisfy the identity:
\[
D_{\id}\Delta^0\circ \kappa^0+ \kappa^1\circ D_{\phi}\Delta^1=\id_{T_\phi\mc C^1}.
\]
By Lemma \ref{lem:implicit}, there is an $h=\{h_\alpha\}\in \mc C^0$ close to $\set{\id}\in\mc C^0$ such that $\Delta^0(h)=\phi'$. Equivalently, $h_\alpha(u)\phi_{\alpha\beta}(u) = \phi_{\alpha\beta}'(u)h_\beta(u)$ for all $u\in V_\alpha\cap V_\beta$, and all $\alpha$ and $\beta$.
\end{proof}

Now we are prepared to finish the proof of the main result of this appendix.

\begin{proof}[Proof of Proposition \ref{prop:Benven}]
Let $\pi :X \to X/\Gamma$ be the natural covering map. Following \cite{Benveniste00}, we start with an appropriately chosen open cover $\{U_\alpha\}$ of $X/\Gamma$ with lifts $\{\til U_\alpha\subset X\}$ such that $\til U_\alpha\to U_\alpha$ is a $C^\infty$  diffeomorphism  for each index $\alpha$. The choice of $\til U_\alpha$ determines  trivialisations of the bundles $(X\times F)/\what{\tau}_0\to X/\Gamma$ and $(X\times F)/\what{\tau}\to X/\Gamma$. If $U_\alpha\cap U_\beta\neq \varnothing$, then $\pi|_{\til U_\beta}^{-1}\circ\pi|_{\til U_\alpha\cap\pi^{-1}(U_\beta)}$ is the restriction of an element $\gamma_{\alpha\beta}\in\Gamma$.
The transition functions for the bundles $(X\times F)/\what{\tau}_0$ and $(X\times F)/\what{\tau}$ are $\tau_0(\gamma_{\alpha\beta})$ and $\tau(\gamma_{\alpha\beta})$ respectively. By Lemma \ref{lem:Benven1},  there exists an open  cover $V_\alpha\subset\subset U_\alpha$ of $X/\Gamma$ and a family $\{h_\alpha:V_\alpha\to \Diff^k(F)\}$ of $C^\ell$ maps that are $C^\ell$-close with the constant  map
to the identity.
Since $h_\alpha \tau_0(\gamma_{\alpha\beta})=\tau(\gamma_{\alpha\beta})h_\beta$, the family $\{h_\alpha\}$ defines a $C^\ell$ bundle isomorphism $f':(X\times F)/\what{\tau}\to (X\times F)/\what{\tau}_0$ covering the identity map on $X$.

By Whitney's Theorem, the manifolds $(X\times F)/\what{\tau}_0$ and $(X\times F)/\what{\tau}$ admit $C^\infty$ structures compatible with their natural $C^k$ structures. Moreover, there is a $C^\infty$ diffeomorphism $f$ that is $C^\ell$-close to $f'$ since $\ell\geq 1$. The last statement regarding the foliations follows immediately.
\end{proof}

\providecommand{\bysame}{\leavevmode\hbox to3em{\hrulefill}\thinspace}
\providecommand{\MR}{\relax\ifhmode\unskip\space\fi MR }
\providecommand{\MRhref}[2]{%
  \href{http://www.ams.org/mathscinet-getitem?mr=#1}{#2}
}
\providecommand{\href}[2]{#2}


\end{document}